\edef\savecatcodeat{\the\catcode`@}
\catcode`\@=11

\def\tb@ifSpecChars#1#2{#1}
\def\tb@ifNoSpecChars#1#2{#2}

\def\tableau{%
  \bgroup
  \@ifstar{\let\Tif\tb@ifNoSpecChars\tb@tableauB}
          {\let\Tif\tb@ifSpecChars\tb@tableauB}}

\def\tb@tableauB{
  \@ifnextchar[{\tb@tableauC}{\tb@tableauC[]}}

\def\tb@tableauC[#1]{\hbox\bgroup%
    \let\\=\cr
    \def\bl{\global\let\tbcellF\tb@cellNF}%
    \def\tf{\global\let\tbcellF\tb@cellH}
    \dimen2=\ht\strutbox \advance\dimen2 by\dp\strutbox%
    \ifx\baselinestretch\undefined\relax%
    \else%
       \dimen0=100sp \dimen0=\baselinestretch\dimen0%
       \dimen2=100\dimen2 \divide\dimen2 by\dimen0%
    \fi%
    \let\tpos\tb@vcenter
    \tb@initYoung
    \tb@options#1\eoo
    \let\arrow\tb@arrow%
    \dimen0=\Tscale\dimen2%
    \dimen1=\dimen0 \advance\dimen1 by \tb@fframe%
    \lineskip=0pt\baselineskip=0pt
    \def\tb@nothing{}%
    \def\endcellno{$\rss\egroup\bss\egroup}
    \def\endcell{\endcellno\kern-\dimen0}
    \def\begincell{\vbox to\dimen0\bgroup\vss\hbox to\dimen0\bgroup\hss$}%
    \let\overlay\tb@overlay%
    \let\fl\tb@fl%
    \let\lss\hss\let\rss\hss\let\tss\vss\let\bss\vss
    \def\mkcell##1{
        \let\tbcellF\tb@cellD
        \def\tb@cellarg{##1}
        \ifx\tb@cellarg\tb@nothing\let\tb@cellarg\tb@cellE\fi%
        \begincell\tb@cellarg\endcellno
        \tbcellF}
    \let\savecellF\tbcellF
     \Tif{\catcode`,=4\catcode`|=\active}{}\tb@tableauD}%

\let\tb@savetableauD\tableauD
{
    \catcode`|=\active \catcode`*=\active \catcode`~=\active%
    \catcode`@=\active
\gdef\tableauD#1{%
  \Tif{
    \mathcode`|="8000 \mathcode`*="8000%
    \mathcode`~="8000 \mathcode`@="8000%
    \def@{\bullet}%
    \let|\cr
    \let*\tf
    \let~\sk
  }{}%
  \tpos{\tabskip=0pt\halign{&\mkcell{##}\cr#1\crcr}}%
  \global\let\tbcellF\savecellF
  \egroup
  \egroup}
}
\let\tb@tableauD\tableauD
\let\tableauD\tb@savetableauD
\let\tb@savetableauD\undefined

\def\tb@options#1{\ifx#1\eoo\relax\else\tb@option#1\expandafter\tb@options\fi}

\def\tb@option#1{%
  \if#1t\let\tpos\tb@vtop\fi
  \if#1c\let\tpos\tb@vcenter\fi
  \if#1b\let\tpos\vbox\fi
  \if#1F\tb@initFerrers\fi
  \if#1Y\tb@initYoung\fi
  \if#1s\tb@initSmall\fi
  \if#1m\tb@initMedium\fi
  \if#1l\tb@initLarge\fi
  \if#1p\tb@initPartition\fi
  \if#1a\tb@initArrow\fi
}

\def\tb@vcenter#1{\ifmmode\vcenter{#1}\else$\vcenter{#1}$\fi}

\def\tb@vtop#1{\hbox{\raise\ht\strutbox\hbox{\lower\dimen0\vtop{#1}}}}

\def\tb@initPartition{\def\Tscale{.3}}
\def\tb@initSmall{\def\Tscale{1}}
\def\tb@initMedium{\def\Tscale{2}}
\def\tb@initLarge{\def\Tscale{3}}

\def\tb@initArrow{\dimen2=1.25em}

\def\tb@initYoung{%
  \def\tb@cellE{}
  \let\tb@cellD\tb@cellN
  \def\sk{\global\let\tbcellF\tb@cellNF}}
\def\tb@initFerrers{%
  \def\tb@cellE{\bullet}
  \let\tb@cellD\tb@cellNF
  \def\sk{\bullet}}

\tb@initMedium

\def\tb@sframe#1{%
  \vbox to0pt{
    \vss
    \hbox to0pt{%
      \hss
      \vbox to\dimen1{
        \hrule depth #1 height0pt
        \vss
        \hbox to\dimen1{
          \vrule width #1 height\dimen1
          \hss
          \vrule width #1
          }%
        \vss
        \hrule height #1 depth 0in
        }%
      \kern-\tb@hframe
      }%
    \kern-\tb@hframe}}

\def\tb@hframe{.2pt}\def\tb@fframe{.4pt}\def\tb@bframe{2pt}
\def\tb@cellH{\tb@sframe{\tb@bframe}}       
\def\tb@cellNF{}                            
\def\tb@cellN{\tb@sframe{\tb@fframe}}       
\let\tbcellF\tb@cellN                       

\def\tb@rpad{1pt}
\def\tb@lpad{1pt}
\def\tb@tpad{1.8pt}
\def\tb@bpad{1.8pt}

\def\tb@overlay{\endcell\@ifnextchar[{\tb@overlaya}{\begincell}}
\def\tb@overlaya[#1]{\vbox to\dimen0\bgroup%
  \tb@overlayoptions#1\eoo%
  \tss\hbox to\dimen0\bgroup\lss$}
\def\tb@overlayoptions#1{\ifx#1\eoo\relax\else\tb@overlayoption#1\expandafter\tb@overlayoptions\fi}

\def\tb@overlayoption#1{
  \if#1t\def\tss{\vskip\tb@tpad}\let\bss\vss\fi
  \if#1c\let\tss\vss\let\bss\vss\fi
  \if#1b\def\bss{\vskip\tb@bpad}\let\tss\vss\fi
  \if#1l\def\lss{\hskip\tb@lpad}\let\rss\hss\fi
  \if#1m\let\lss\hss\let\rss\hss\fi
  \if#1r\def\rss{\hskip\tb@rpad}\let\lss\hss\fi
}

\def\tb@fl{\endcell\begincell\vrule depth 0pt width \dimen0 height \dimen0 \endcell\begincell}

\@ifundefined{diagram}{}{
\def\tb@arrowpad{.5}

\newoptcommand{\tb@arrow}{\@ne}[2]{%
  \endcell
   \begingroup%
   \let\dg@getnodesize\tb@getnodesize
   \dg@USERSIZE=#1\relax%
   \ifnum\dg@USERSIZE<\@ne \dg@USERSIZE=\@ne \fi%
   \dg@parse{#2}%
   \dg@label{\tb@draw{#1}{#2}}}

\def\tb@getnodesize#1#2#3#4#5{\dimen3=\tb@arrowpad\dimen2 #4=\dimen3 #5=\dimen3\relax}
\def\tb@getnodesize#1#2#3#4#5{\ifnum#2=0\ifnum#3=0\tb@getnodesizetail{#4}{#5}\else\tb@getnodesizehead{#4}{#5}\fi\else\tb@getnodesizehead{#4}{#5}\fi}
\def\tb@getnodesizetail#1#2{\dimen3=.5\dimen2 #1=\dimen3 #2=\dimen3}
\def\tb@getnodesizehead#1#2{\dimen3=.5\dimen2 #1=\dimen3 #2=\dimen3}

\def\tb@draw#1#2#3#4{%
        \dg@X=0\dg@Y=0\dg@XGRID=1\dg@YGRID=1\unitlength=.001\dimen0%
        \dg@LBLOFF=\dgLABELOFFSET \divide\dg@LBLOFF\unitlength%
        \dg@drawcalc
        \begincell
        \let\lams@arrow\tb@lams@arrow
        \begin{picture}(0,0)\begingroup\dg@draw{#1}{#2}{#3}{#4}\end{picture}%
        \endcell
        \endgroup
        \begincell}
}

\def\tb@lams@arrow#1#2{%
 \lams@firstx\z@\lams@firsty\z@
 \lams@lastx#1\relax\lams@lasty#2\relax
 \lams@center\z@
 \N@false\E@false\H@false\V@false
 \ifdim\lams@lastx>\z@\E@true\fi
 \ifdim\lams@lastx=\z@\V@true\fi
 \ifdim\lams@lasty>\z@\N@true\fi
 \ifdim\lams@lasty=\z@\H@true\fi
 \NESW@false
 \ifN@\ifE@\NESW@true\fi\else\ifE@\else\NESW@true\fi\fi
 \ifH@\else\ifV@\else
  \lams@slope
  \ifnum\lams@tani>\lams@tanii
   \lams@ht\ten@\p@\lams@wd\ten@\p@
   \multiply\lams@wd\lams@tanii\divide\lams@wd\lams@tani
  \else
   \lams@wd\ten@\p@\lams@ht\ten@\p@
   \divide\lams@ht\lams@tanii\multiply\lams@ht\lams@tani
  \fi
 \fi\fi
 \ifH@  \lams@harrow
 \else\ifV@ \lams@varrow
 \else \lams@darrow
 \fi\fi
}

\catcode`\@=\savecatcodeat
\let\savecatcodeat\undefined

\documentclass[tbtags,reqno]{amsart}
\usepackage{graphicx,color}
\usepackage{amsmath,amsfonts,amstext,amsbsy,amsopn,amsthm,amscd,amsxtra,dsfont,amssymb,pb-diagram}
\usepackage{amsmath,amsthm,amsopn,amsfonts,amssymb,epsfig,enumerate}
\usepackage{amsthm,amsopn,amsfonts,amssymb,epsfig,mathrsfs,cancel}
\usepackage[all]{xy}
\usepackage[active]{srcltx}

\numberwithin{equation}{section}

\makeatletter
\renewcommand{\subsubsection}{\@startsection
{subsubsection} {3} {0mm} {\baselineskip} {-0.5\baselineskip} {\normalfont\normalsize\bfseries}} \makeatother

\newtheorem{theorem}{Theorem}
\newtheorem{lemma}[theorem]{Lemma}
\newtheorem{proposition}[theorem]{Proposition}

\newtheorem{corollary}[theorem]{Corollary}

\newtheorem{remark}[theorem]{Remark}
\newtheorem*{acknow}{Acknowledgments}

\def\la{{\lambda}}

\def\cal L{{\mathcal L}}

\newcommand{\La}{\Lambda}

\def \part {\vdash}
\def\La {\Lambda}

\def \Om {\Omega}
\def \ta {\theta}
\def\cd{\circledast}

\newcommand{\tcercle}[1]{\ensuremath{\setlength{\unitlength}{1ex}\begin{picture}(2.8,2.8)\put(1.4,1.4){\circle{2.8}\makebox(-5.6,0){#1}}\end{picture}}}

\begin{document}

\title[Pieri rules for Schur functions in superspace]
{Pieri rules for Schur functions in superspace}

\author{Miles Jones }
\address{Instituto de Matem\'atica y F\'{\i}sica, Universidad de Talca,
Casilla 747, Talca, Chile} \email{mjones-at-inst-mat.utalca.cl}

\author{Luc Lapointe}
\address{Instituto de Matem\'atica y F\'{\i}sica, Universidad de Talca, Casilla 747, Talca,
Chile} \email{lapointe-at-inst-mat.utalca.cl}


\begin{abstract}   
The Schur functions in superspace $s_\Lambda$ and $\bar s_\Lambda$
are the limits $q=t=0$ and $q=t=\infty$ respectively of the Macdonald polynomials in superspace. 
We prove Pieri rules for the bases $s_\Lambda$ and $\bar s_{\Lambda}$ (which happen to
be essentially dual).  As a consequence, we derive the basic properties of these bases such as dualities, monomial expansions, and tableaux generating functions.
\end{abstract}


\maketitle

\section{Introduction}

An extension to superspace of the theory of symmetric functions was developed in \cite{BDLM2,DLM1,DLM2}.  In this extension, the polynomials
$f(x,\theta)$, where $(x,\theta)=(x_1,\dots,x_N,\theta_1,\dots,\theta_N)$, not only depend on the usual commuting variables
$x_1,\dots,x_N$ but also on the anticommuting variables 
$\theta_1,\dots,\theta_N$ ($\theta_i \theta_j=- \theta_j \theta_i, \text{~and~} \theta_i^2=0)$.  In this article, we
are concerned with two natural generalizations to superspace of the Schur functions that arise as special limits of
the Macdonald polynomials in superspace and whose combinatorics appears to be extremely rich.

The extension to superspace of the Macdonald polynomials, $\{ P_{\Lambda}(x,\theta;q,t) \}_\Lambda$, is a basis
of the ring $\mathbb Q(q,t)[x_1,\dots,x_N;\theta_1,\dots,\theta_N]^{S_N}$ of symmetric polynomials in superspace, where the superscript $S_N$ 
indicates that the elements of the ring are invariant under the diagonal action of the symmetric group $S_N$ (that is, invariant under
the simultaneous interchange of $x_i \leftrightarrow x_j$ and $\theta_i \leftrightarrow \theta_j$, for any $i,j$).  They are
indexed by superpartitions $\Lambda$  and defined as the unique basis 
such that 
\begin{enumerate}
\item $P_{\Lambda}(q,t) = m_\Lambda +$ smaller terms 
\item  $\langle  P_{\Lambda}(q,t), P_{\Omega}(q,t) \rangle \!  \rangle_{q,t}= 0$ \quad if \quad  $\Lambda \neq \Omega$
\end{enumerate}
where the scalar product  $ \langle \!  \langle \cdot , \cdot  \rangle \!  \rangle_{q,t}$ is given by
\begin{equation} \label{scalmac}
\langle \!  \langle  p_{\Lambda}, p_{\Omega} \rangle \!  \rangle_{q,t}= \delta_{\Lambda \Omega} \, q^{|\Lambda^a|} z_{\Lambda^s} \prod_{i}\frac{1-q^{\Lambda^s_i}}{1-t^{\Lambda^{s}_i}}
\end{equation}
on the power sum symmetric functions in superspace (see Section~\ref{secdef} 
for all the relevant definitions).  
It was shown in \cite{BDLM2} that even though the limits $q=t=0$ and $q=t=\infty$
of this scalar product are degenerate and not well-defined respectively, the corresponding 
limits $s_\Lambda:=P_{\Lambda}(0,0)$ and $\bar s_\Lambda:=P_\Lambda(\infty,\infty)$ of the Macdonald superpolynomials 
exist and are related to Key polynomials \cite{Ion,LS}.   As we will see,  the rich combinatorics of these functions 
makes them the genuine extensions to superspace of the Schur functions. 
In comparison, the a priori more relevant limit $q=t=1$ of the Macdonald polynomials in superspace,
which corresponds to the
limit $\alpha=1$ of the Jack polynomials in superspace $P_\Lambda(\alpha)$, does not seem to be very interesting from the combinatorial point of view (in the figure below, the limit $q=t=1$ of the Macdonald polynomials in superspace corresponds to $S_\Lambda$).

\begin{figure}
$$\qquad 
\dgARROWLENGTH=2.8em
\begin{diagram}
\node{} \node{{P_\Lambda{(q,t)}}} \arrow{ssw,t,..}{q=t=0} 
\arrow{sse,t,..}{q=t=\infty}
\arrow{s,b}{}
\node{}  \\
\node{} \node{{P_\Lambda{(\alpha)}}} \arrow{s,b}{\alpha\to 1} \node{} 
 \\
\node{ {s_\Lambda}}  \node{{S_\Lambda}}  \node{{\bar s_\Lambda}}   
\end{diagram}
\qquad \qquad 
$$
\end{figure}

The basis $s_\Lambda$ is especially relevant since it plays the role of the Schur functions
in the generalization to superspace of the original Macdonald positivity conjectures \cite{BDLM1}.  To be more specific, let
$J_\Lambda(q,t)=c_\Lambda(q,t) P_\Lambda(q,t)$ be the integral form of the Macdonald superpolynomials ($c_\Lambda(q,t)$ is a constant belonging to $\mathbb Z[q,t]$) and let $\varphi (s_{\Lambda})$ be a certain plethystically transformed version of 
the function $s_{\Lambda}$ (see \cite{BDLM1} for more details).  Then the coefficients $K_{\Omega \Lambda}(q,t)$ appearing in 
\begin{equation}\label{mackostka}
J_\Lambda(q,t)= \sum_{\Omega} K_{\Omega \Lambda}(q,t) \, \varphi (s_{\Omega})
\end{equation}
are conjectured to be polynomials in $q$ and $t$ with nonnegative integer coefficients (the conjecture is known to hold
when the degree in the anticommuting variables is either zero \cite{Hai}, which corresponds to the usual Macdonald case, or sufficiently large \cite{BLM}).

In this article, we will derive the basic properties of the bases $s_\Lambda$ and $\bar s_{\Lambda}$ (which as we will
see are essentially dual) such as Pieri rules, dualities, monomial expansions, tableaux generating functions,
and Cauchy identities.  It is important to note that the combinatorics of the bases $s_\Lambda$ and $\bar s_{\Lambda}$  was first studied in \cite{BM}.  Our work stems in large part from a desire to develop the right framework to 
prove the conjectures therein, especially those concerning
Pieri rules\footnote{The connection between the results of \cite{BM} and those of this article is
discussed in Remarks \ref{remarkpieri} and \ref{remarktableaux}.}. 

We are  confident that this work is only the tip of the iceberg and that
deeper properties of the bases $s_\Lambda$ and $\bar s_{\Lambda}$ will be uncovered in the future,
such as for instance a group-theoretical interpretation of the generalization to superspace of the Macdonald positivity
conjecture.  At the tableau level, we are hopeful that this work will eventually lead to
a Robinson-Schensted-Knuth insertion algorithm in superspace, and ultimately to a charge statistic on tableaux that would solve
the case $q=0$ of \eqref{mackostka}.

The most technical parts of this work
are the proofs of the Pieri rules.  
They rely on the correspondence between
the Schur functions in superspace and Key polynomials \cite{BDLM2} which allows to use the powerful machinery of
divided differences \cite{Las}.  Once the Pieri rules are assumed to hold, the remaining results follow somewhat easily from duality
arguments or well-known techniques of symmetric function theory \cite{Mac,Sta}.   To alleviate the presentation, the proofs of
many technical results have been relegated to the appendix.

\section{Symmetric function theory in superspace: basic definitions} \label{secdef}

Before introducing the basic 
concepts of symmetric function theory in superspace, we discuss an important identification (as vector spaces) between the ring of symmetric polynomials in superspace and the ring of bisymmetric polynomials.

A polynomial {in} superspace, or equivalently, a superpolynomial, is
a polynomial in the usual $N$ variables $x_1,\ldots ,x_N$  and the $N$ anticommuting variables $\ta_1,\ldots,\ta_N$ over a certain field, which will be taken in the remainder of this article to be $\mathbb Q$.  
A superpolynomial $P(x,\theta)$,
with $x=(x_1,\ldots,x_N)$ and $\theta=(\theta_1,\ldots,\theta_N)$, is said to be symmetric if the following is satisfied:
\begin{equation}
P(x_1,\dots,x_N,\theta_1,\dots,\theta_N) = P(x_{\sigma(1)},\dots,x_{\sigma(N)},\theta_{\sigma(1)},\dots,\theta_{\sigma(N)}) \qquad 
\forall \, \sigma \in S_N
\end{equation}
where $S_N$ is the symmetric group on $\{1,\dots, N \}$.  The ring of superpolynomials in $N$ variables has a natural grading with respect to the fermionic 
degree $m$ (the total degree in the anticommuting variables).
We will denote
by $\Lambda^m_N$ the ring of symmetric superpolynomials in $N$ variables and fermionic degree $m$ over the field $\mathbb Q$.  By symmetry, one can reconstruct
a superpolynomial in  $\Lambda^m_N$ from its coefficient 
$\theta_1 \cdots \theta_m$.  Furthermore, since that coefficient is necessarily antisymmetric in the variables $x_1,\dots,x_m$, and symmetric in the remaining variables, there is a natural identification as vector spaces 
between $\Lambda_N^m$ and the ring of bisymmetric polynomials 
$\mathbb Q [x_1,\dots,x_N]^{S_m \times S_{m^c}}$, where $S_m \times S_{m^c}$
is taken as the subgroup of $S_N$ given by the permutations that leave
the sets $\{1,\dots,m \}$ and  $\{m+1,\dots,N \}$ invariant.  Given an
element $P(x,\theta)$ of $\Lambda_N^m$, the identification is simply 
\begin{equation} \label{eqidenti}
P(x,\theta) \quad  \longleftrightarrow \quad \frac{P(x,\theta) \Big |_{\theta_1 \cdots \theta_m}}{\Delta_m(x)}
\end{equation}
where $\Delta_m(x)$ is the Vandermonde determinant $\prod_{1 \leq i<j \leq m} (x_i-x_j)$, and $P(x,\theta) \Big |_{\theta_1 \cdots \theta_m}$
is the coefficient of $\theta_1\cdots \theta_m$ in $P(x,\theta)$.

\subsection{Superpartitions}
We  first recall some definitions
related to partitions \cite{Mac}.
A partition $\lambda=(\lambda_1,\lambda_2,\dots)$ of degree $|\lambda|$
is a vector of non-negative integers such that
$\lambda_i \geq \lambda_{i+1}$ for $i=1,2,\dots$ and such that
$\sum_i \lambda_i=|\lambda|$.  
Each partition $\lambda$ has an associated Ferrers diagram
with $\lambda_i$ lattice squares in the $i^{th}$ row,
from the top to bottom. Any lattice square in the Ferrers diagram
is called a cell (or simply a square), where the cell $(i,j)$ is in the $i$th row and $j$th
column of the diagram.  
The conjugate $\lambda'$ of  a partition $\lambda$ is represented  by
the diagram
obtained by reflecting  $\lambda$ about the main diagonal.
We say that the diagram $\mu$ is contained in $\la$, denoted
$\mu\subseteq \la$, if $\mu_i\leq \la_i$ for all $i$.  Finally,
$\la/\mu$ is a horizontal (resp. vertical) $n$-strip if $\mu \subseteq \lambda$, $|\lambda|-|\mu|=n$,
and the skew diagram $\la/\mu$ does not have two cells in the same column
(resp. row). 

Symmetric superpolynomials  are naturally indexed by superpartitions\footnote{Superpartitions correspond, using a trivial bijection, to the
overpartitions studied in \cite{CL}.}. A superpartition $\Lambda$ of
degree $(n|m)$, or  $\Lambda \vdash (n|m)$ for short, 
 is a pair $(\Lambda^\circledast,\Lambda^*)$ of partitions
$\Lambda^\circledast$ and $\Lambda^*$ such
 that:\\ \\
 1. $\Lambda^* \subseteq \Lambda^\circledast$;\\
 2.  the degree of $\Lambda^*$ is $n$;\\
 3.  the skew diagram $\Lambda^\circledast/\Lambda^*$
is both a horizontal and a vertical $m$-strip\footnote{Such diagrams are sometimes called $m$-rook strips.}\\ \\
We refer to  $m$ and $n$ respectively as the fermionic degree 
and total degree 
of $\La$. 
 Obviously, if
$\Lambda^\circledast= \Lambda^*=\lambda$,
then $\Lambda=(\lambda,\lambda)$ can be interpreted as the partition
$\lambda$.

We will also need another characterization of a superpartition.
 A superpartition $\La$ is 
a pair of partitions $(\La^a; \La^s)=(\La_{1},\ldots,\La_m;\La_{m+1},\ldots,\La_N)$, 
 where $\La^a$ is a partition with $m$ 
distinct parts (one of them possibly  equal to zero),
and $\La^s$ is an ordinary partition (with possibly a string of zeros at the end).   The correspondence 
between $(\Lambda^\circledast,\Lambda^*)$ and 
$(\Lambda^a; \Lambda^s)$ is given explicitly as follows:
given 
$(\Lambda^\circledast,\Lambda^*)$, the parts of $\Lambda^a$ correspond to the
parts of $\Lambda^*$ such that $\Lambda^{\circledast}_i\neq 
\Lambda^*_i$, while the parts of $\Lambda^s$ correspond to the
parts of $\Lambda^*$ such that $\Lambda^{\circledast}_i=\Lambda^*_i$.

The conjugate of a superpartition 
$\Lambda=(\Lambda^\circledast,\Lambda^*)$ is $\Lambda'=((\Lambda^\circledast)',(\Lambda^*)')$.
A diagrammatic representation of $\La$ is given by 
the Ferrers diagram of $\La^*$ with circles added in the cells corresponding
to $\Lambda^{\circledast}/\Lambda^*$.
For instance, if $\La=(\Lambda^a;\Lambda^s)
=(3,1,0;2,1)$,  we have $\Lambda^\circledast=(4,2,2,1,1)$ and $\Lambda^*
=(3,2,1,1)$, so that 
\begin{equation} \label{exdia}
{\scriptsize     \La^\cd:\quad{\tableau[scY]{&&&\\&\\&\\\\ \\ }} \quad
         \La^*:\quad{\tableau[scY]{&&\\&\\ \\ \\ }} \quad
 \Longrightarrow\quad      \La:\quad {\tableau[scY]{&&&\bl\tcercle{}\\&\\&\bl\tcercle{}\\ \\
    \bl\tcercle{}}} \quad    \La':\quad {\tableau[scY]{&&&&\bl\tcercle{}\\&&\bl\tcercle{}\\ \\
    \bl\tcercle{}}},}
\end{equation}
where the last diagram illustrates the conjugation operation that corresponds, as usual, to replacing rows by columns.

From the dominance ordering on partitions
\begin{equation}\label{ordre}
   \mu \leq \la\quad\text{ iff }\quad |\mu|=|\la|\quad\text{
and }\quad \mu_1 + \cdots + \mu_i \leq \lambda_1 + \cdots + \lambda_i\quad \forall i \, . \end{equation}
we define the dominance ordering on superpartitions as
\begin{equation} \label{eqorder1}
 \Omega\leq\Lambda \quad \text{iff}\quad
 \deg(\La)=\deg(\Om) ,
 \quad \Omega^* \leq \Lambda^*\quad \text{and}\quad
\Omega^{\circledast} \leq  \Lambda^{\circledast}
\end{equation}
where we stress that the order on partitions is the dominance ordering.

\subsection{Simple bases}
Four simple bases  of the space of 
symmetric polynomials in superspace will be particularly relevant to our work \cite{DLM1}:
\begin{enumerate}\item 
the extension of the monomial symmetric 
functions, $m_\La$, defined by 
\begin{equation}
m_\La={\sum_{\sigma \in S_N} }' \theta_{\sigma(1)}\cdots\theta_{\sigma(m)} x_{\sigma(1)}^{\La_1}\cdots x_{\sigma(N)}^{\La_N},
\end{equation}
where the sum is over the  permutations of $\{1,\ldots,N\}$ that produce distinct terms, and where
the entries of $(\Lambda_1,\dots,\Lambda_N)$ are those of $\Lambda=(\La^a; \La^s)=(\La_{1},\ldots,\La_m;\La_{m+1},\ldots,\La_N)$;
\item 
 the generalization of the power-sum 
symmetric
functions $p_\La=\tilde{p}_{\La_1}\cdots\tilde{p}_{\La_m}p_{\La_{m+1}}\cdots p_{\La_{\ell}}\, $,
\begin{equation} 
\text{~~~{where}~~~}
\tilde{p}_k=\sum_{i=1}^N\theta_ix_i^k\qquad\text{and}\qquad p_r=
\sum_{i=1}^Nx_i^r \, , \text{~~~~for~~} k\geq 0,~r \geq 1;
\end{equation}  

\item 
 the generalization of the elementary 
symmetric
functions $e_\La=\tilde{e}_{\La_1}\cdots\tilde{e}_{\La_m}e_{\La_{m+1}}\cdots e_{\La_{\ell}}$ ,

\begin{equation} \text{~~~where~~~} \tilde{e}_k=m_{(0;1^k)}
\quad \text{and} \quad e_r=m_{(\emptyset;1^r)}
, \text{~~~~for~~} k\geq 0,~r \geq 1;
\end{equation}  
\item 
 the generalization of the homogeneous
symmetric
functions  $h_\La=\tilde{h}_{\La_1}\cdots\tilde{h}_{\La_m}h_{\La_{m+1}}\cdots h_{\La_{\ell}}\, ,$
\begin{equation}  \text{~~~where~~~}  \tilde{h}_k=\sum_{\Lambda \vdash (n|1)} (\Lambda_1+1)m_{\Lambda}
\qquad \text{and} \qquad h_r=\sum_{\Lambda \vdash (n|0)}m_{\Lambda}, 
 \text{~~~~for~~} k\geq 0,~r \geq 1
\end{equation}  
\end{enumerate}
Observe that when $\Lambda=(\emptyset; \lambda)$, we have that $m_\Lambda=m_\lambda$,  $p_\Lambda=p_\lambda$,  $e_\Lambda=e_\lambda$ and
 $h_\Lambda=h_\lambda$ are respectively the usual monomial, power-sum, elementary and homogeneous symmetric functions.  Also note
that if we define the operator $d=\theta_1 {\partial}/{\partial_{x_1}}+\cdots + \theta_N \partial/\partial_{x_N}$, we have
\begin{equation}
(k+1) \, \tilde p_k = d ( p_{k+1})\, , \qquad  \tilde e_k = d (e_{k+1})  \quad \text{and} \quad \tilde h_k = d(h_{k+1})  
\end{equation}
that is, the new generators in the superspace versions of the bases can be obtained from acting with $d$ on the generators of the usual symmetric function versions.

\subsection{Scalar product and duality}
The relevant scalar product in this article is the specialization $q=t=1$ of the scalar product (\ref{scalmac})\footnote{The scalar product (\ref{scalmac})
differs from that of \cite{BDLM2} by a sign depending on the fermionic degree.}, that is, 
\begin{equation} \label{scal1}
\langle \! \langle p_{\Lambda}, p_{\Omega} \rangle \! \rangle= 
\delta_{\Lambda \Omega}\, z_{\Lambda^s}
\end{equation}
where, as usual,  $z_\lambda=1^{n_\lambda(1)} n_\lambda(1)! \, 2^{n_\lambda(2)} n_\lambda(2)! \cdots $ with $n_\lambda(i)$ the number of
parts of $\lambda$ equal to $i$.  
The homogeneous and monomial bases are dual with respect to this scalar product \cite{DLM1}
\begin{equation} \label{hmdual}
\langle \! \langle h_{\Lambda}, m_{\Omega} \rangle \! \rangle = 
 \delta_{\Lambda \Omega}
\end{equation}
We define the homomorphism $\omega$ as
\begin{equation}
\omega(\tilde p_r) = (-1)^r \tilde p_r  \quad {\rm and}  \quad \omega(p_r)= (-1)^{r-1} p_r
\end{equation}
This homomorphism, which is obviously an involution and an isometry of the scalar product 
$\langle \! \langle \cdot, \cdot \rangle \! \rangle$, is such that \cite{DLM1}
\begin{equation} \label{hedual}
\omega(h_\Lambda)= e_\Lambda
\end{equation}

The Schur functions in superspace $s_\Lambda$ and $\bar s_\Lambda$ were defined in the introduction as the special limits
$q=t=0$ and $q=t=\infty$ respectively of the Macdonald polynomials in superspace.  Remarkably, the 
functions $s_\Lambda$ and $\bar s_\Lambda$ are essentially dual with respect to our scalar product.
\begin{proposition}[\cite{BDLM2}] \label{propdual}
Let $s_{\Lambda}^*$ and $\bar s_{\Lambda}^*$ be the bases dual to the bases  $s_{\Lambda}$ and $\bar s_{\Lambda}$ respectively, that is, let $s_{\Lambda}^*$ and $\bar s_{\Lambda}^*$ be such that
\begin{equation}
\langle \! \langle s_{\Lambda}^*, s_{\Omega} \rangle \! \rangle =
\langle \! \langle \bar s_{\Lambda}^*, \bar s_{\Omega} \rangle \! \rangle =
 \delta_{\Lambda \Omega}
\end{equation}
Then
\begin{equation}
s_{\Lambda}^*= (-1)^{\binom{m}{2}} \omega \bar s_{\Lambda'} \qquad  {\rm and} \qquad
\bar s_{\Lambda}^*= (-1)^{\binom{m}{2}} \omega  s_{\Lambda'}
\end{equation}
where $m$ is the fermionic degree of $\Lambda$.
\end{proposition}
When $m=0$, we have $s_\Lambda=s_{(\emptyset; \lambda)}=s_\lambda$ and  $\bar s_\Lambda=\bar s_{(\emptyset; \lambda)}=s_\lambda$.  In this case the proposition is simply stating the well known fact 
that the dual basis of the Schur basis $\{ s_\lambda\}$ with respect to the Hall scalar product is the basis $\{s_\lambda= \omega s_{\lambda'}\}$, that is, that the Schur basis is orthonormal.

\section{Relation with  Key polynomials}

A connection between Key polynomials and Schur functions in superspace
was established in \cite{BDLM2}.  Before stating that connection explicitly
 we introduce the Key polynomials \cite{Las,LS}\footnote{Key polynomials are also known as Type A Demazure atoms \cite{Mas}.}.

\subsection{Key polynomials}
 Let $\pi_i$ be the isobaric divided difference operator
\begin{equation}
 \pi_i = \frac{1}{(x_i-x_{i+1})} (x_i-x_{i+1} \kappa_{i,i+1}) 
\end{equation}
where $\kappa_{i,i+1}$ is the operator that interchanges the variables $x_i$ and
$x_{i+1}$: 
\begin{equation}
 \kappa_{i,i+1} \, f(x_1,\dots,x_i,x_{i+1},\dots,x_N) =  f(x_1,\dots,x_{i+1},x_{i},\dots,x_N)
\end{equation}
The isobaric divided difference operators satisfy the relations
\begin{equation}
\pi_i \pi_j = \pi_j \pi_i \quad {\rm if~}|i-j|> 1 \, , \qquad
\pi_i \pi_{i+1} \pi_i= \pi_{i+1} \pi_{i} \pi_{i+1} \, , \qquad \pi_i^2=\pi_i
\end{equation}
and
\begin{equation}
\pi_i \, x_i = x_{i+1} \pi_i +x_i \, , \qquad
\pi_i\,  x_{i+1} = x_{i} \, \pi_i -x_i
\end{equation}

Given $\eta \in \mathbb Z_{\geq 0}^N$,  the Key polynomial $K_\eta(x_1,\dots,x_N)$ 
is defined recursively as follows.  If $\eta$ is weakly decreasing, then
\begin{equation}
  K_\eta = x^\eta= x_1^{\eta_1} x_2^{\eta_2} \cdots x_N^{\eta_N}
\end{equation}
Otherwise, if $\eta_i < \eta_{i+1}$ then
\begin{equation}
K_{\eta} = \pi_i K_{s_i \eta }
\end{equation}
where $s_i \eta$ is equal to $\eta$ with the $i$-th and $i+1$-th entries interchanged.  It should be noted that if $\eta$ is weakly increasing then
$K_\eta$ is equal to a Schur function:
\begin{equation} \label{keyschur}
K_\eta(x_1,\dots,x_N) = s_{\lambda}(x_1,\dots,x_N)
\end{equation}
where $\lambda$ is the partition corresponding to the rearrangement of the entries of $\eta$.

The adjoint Key polynomial $\hat K_\eta(x_1,\dots,x_N)$ 
is defined similarly.  If $\eta$ is weakly decreasing, then
\begin{equation}
  \hat K_\eta = x^\eta= x_1^{\eta_1} x_2^{\eta_2} \cdots x_N^{\eta_N}
\end{equation}
Otherwise, if $\eta_i < \eta_{i+1}$ then
\begin{equation}
\hat K_{\eta} = \hat \pi_i \hat K_{s_i \eta }
\end{equation}
where $\hat \pi_i$ is the divided difference operator
\begin{equation}
\hat \pi_i = \pi_i-1  
\end{equation}
which satisfies the relations
\begin{equation}
\hat \pi_i \hat \pi_j = \hat \pi_j \hat \pi_i \quad {\rm if~}|i-j|> 1 \, , 
\qquad
\hat \pi_i \hat \pi_{i+1} \hat \pi_i= \hat \pi_{i+1} \hat \pi_{i} 
\hat \pi_{i+1} \, , \qquad \hat \pi_i^2=-\hat \pi_i\, .
\end{equation}
Note that $\pi_i \hat \pi_i=\hat \pi_i \pi_i = 0$ from the definition of $\hat \pi_i$ and the relation $\pi_i^2=\pi_i$.
\subsection{Key polynomials and Schur functions in superspace} 
To connect Key polynomials and Schur functions in superspace,
we will also need the usual divided difference operator
\begin{equation}
 \partial_i = \frac{1}{(x_i-x_{i+1})} (1- \kappa_{i,i+1}) 
\end{equation}
which is such that
\begin{equation}
\partial_i \partial_j = \partial_j \partial_i 
\quad {\rm if~}|i-j|> 1 \, , \qquad
\partial_i \partial_{i+1} \partial_i= \partial_{i+1} \partial_{i} 
\partial_{i+1} \, , \qquad \partial_i^2=0\, .
\end{equation}
Observe that $\pi_i=\partial_i x_i$, which implies that $\partial_i \pi_i=0$.

For $s_{i_1} \dots s_{i_\ell}$ a reduced decomposition of the permutation 
$\sigma$, we let $\partial_{\sigma}=\partial_{i_1} \cdots \partial_{i_\ell}$ (and similarly for
$\pi_\sigma$ and $\hat \pi_\sigma$). We also let
$\omega_m$ and $\omega_{m^c}$ be the longest 
 permutation of $\{1,\dots,m\}$ and  $\{m+1,\dots,N\}$ respectively.

We now establish the image of the
Schur functions in superspace $s_\Lambda$ and $\bar s_\Lambda$ in the
identification \eqref{eqidenti} between $\Lambda_N^m$ and $\mathbb Q [x_1,\dots,x_N]^{S_m \times S_{m^c}}$.

\begin{lemma} \label{lems}
Let $\Lambda$ be a superpartition of fermionic degree $m$.
In the identification \eqref{eqidenti}
between $\Lambda_N^m$ and $\mathbb Q [x_1,\dots,x_N]^{S_m \times S_{m^c}}$
we have
\begin{equation}
s_\Lambda \quad \longleftrightarrow \quad  (-1)^{\binom{m}{2}}
\partial_{\omega_m} \pi_{\omega_{m^c}} \hat K_{(\Lambda^a)^R,\Lambda^s}(x)
\end{equation}
where $(\Lambda^a)^R,\Lambda^s$ is the composition $(\Lambda^a_m\dots,
\Lambda^a_1,\Lambda^s_1,\Lambda_2^s,\dots)$ obtained by concatenating 
$\Lambda^a$ in reverse order and $\Lambda^s$.  
\end{lemma}
\begin{proof}
Let $A_m= \sum_{\sigma \in S_m} 
(-1)^{\ell(\sigma)} \kappa_\sigma$ be the antisymmetrization operator with respect
to the variables $x_1,\dots,x_m$.  It was established in \cite{BDLM2} that
\begin{equation} \label{eqsdev}
s_\Lambda(x,\theta) = (-1)^{\binom{m}{2}} \sum_{\sigma \in S_N /(S_m \times S_{m^c})}
{\mathcal K}_{\sigma} \theta_1 \cdots \theta_m A_m 
\left(\sum_{v \, \in S_{m^c}} \hat \pi_v \right) \hat K_{(\Lambda^a)^R,\Lambda^s}(x)
\end{equation}
where ${\mathcal K}_\sigma$ is the operator such that
\begin{equation}
{\mathcal K}_\sigma P(x_1,\dots,x_N,\theta_1,\dots,\theta_N) = P(x_{\sigma(1)},\dots,x_{\sigma(N)},\theta_{\sigma(1)},\dots,\theta_{\sigma(N)}) \qquad 
\forall \, \sigma \in S_N
\end{equation}
It is known that $\pi_{\omega_{m^c}}= \sum_{v \, \in S_{m^c}} \hat \pi_v$
and that $A_m ={\Delta_m(x)} \partial_{\omega_m}$ (see for instance \cite{Las}).
We thus have that
\begin{equation}\label{eqn:keyhat}
\frac{s_\Lambda(x,\theta) \big|_{\theta_1\cdots \theta_m}}{\Delta_m(x)} = (-1)^{\binom{m}{2}}
\partial_{\omega_m} \pi_{\omega_{m^c}} \hat K_{(\Lambda^a)^R,\Lambda^s}(x)
\end{equation}
and the lemma follows.
\end{proof}

\begin{lemma} \label{lemsb}
Let $\Lambda$ be a superpartition of fermionic degree $m$.
In the identification \eqref{eqidenti} between $\Lambda_N^m$ and $\mathbb Q [x_1,\dots,x_N]^{S_m \times S_{m^c}}$
we have
\begin{equation}
\bar s_\Lambda \quad \longleftrightarrow \quad  
\partial_{\omega_{(N-m)^c}}' K_{(\Lambda^s)^R,\Lambda^a}(y)
\end{equation}
where $(\Lambda^s)^R,\Lambda^a$ is the composition $(\Lambda^s_{N-m},\dots,
\Lambda^s_1,\Lambda^a_1,\Lambda_2^a,\dots)$ obtained by concatenating 
$\Lambda^s$ in reverse order with $\Lambda^a$, where 
$y$ stands for the variables $y_1=x_N,y_2=x_{N-1},\dots,y_N=x_1$ and where
the $'$ indicates that the divided differences act on the y variables.
 \end{lemma}
\begin{proof}
It was established in \cite{BDLM2} that
\begin{equation}
\bar s_\Lambda(x,\theta) = (-1)^{\binom{m}{2}+\binom{N-m}{2}} 
\sum_{\sigma \in S_N /(S_m \times S_{m^c})}
{\mathcal K}_{\sigma}  \theta_1 \cdots \theta_m 
A_m \frac{1}{\Delta_{m^c}(x)} A_{m^c}
x^{\gamma}   K_{\Lambda^s,\Lambda^a}(x_N,\dots,x_1)
\end{equation}
where $x^\gamma=x_N^{N-m-1} x_{N-1}^{N-m-2} \cdots x_{m+2}$,
and where $A_{m^c}$ is the antisymmetrization operator (see the proof the previous lemma) 
acting on the variables
$x_{m+1},\dots,x_N$ (and similarly for $\Delta_{m^c}(x)$).
  In this case,
it is convenient to make the change of variables $y_i=x_{N+1-i}$ for $i=1,\dots,N$ to obtain
\begin{equation} \label{eqsbarinkey}
\bar s_\Lambda(x,\theta) = (-1)^{\binom{m}{2}}
\sum_{\sigma \in S_N /(S_m \times S_{m^c})}
{\mathcal K}_{\sigma}  \theta_1 \cdots \theta_m 
A_{(N-m)^c}' \frac{1}{\Delta_{N-m}(y)} A_{N-m}' y^\delta   K_{\Lambda^s,\Lambda^a}(y)
\end{equation}
where $y^\delta=y_1^{N-m-1} y_2^{N-m-2} \cdots y_{N-m-1}$, and the $'$ indicate that the operator now acts on the $y$ variables 
(note that ${\mathcal K}_{\sigma}$ still acts on the $x$ variables).  
Observe that changing $\Delta_{m^c}(x)$
to $\Delta_{N-m}(y)$ introduces a factor of $(-1)^{\binom{N-m}{2}}$.  
Using $A_{(N-m)^c}' = \Delta_{(N-m)^c}(y) \partial_{\omega_{(N-m)^c}}'$ and
$ A_{N-m}' y^\delta = \Delta_{N-m}(y) \pi_{\omega_{N-m}}'$,
where again the $'$ mean that that the operators are acting on the $y$ variables, we get
\begin{equation}\label{eqn:key1}
\frac{\bar s_\Lambda(x,\theta) \big|_{\theta_1\cdots \theta_m}}{\Delta_m(x)} =
\partial_{\omega_{(N-m)^c}}' K_{(\Lambda^s)^R,\Lambda^a}(y)
\end{equation}
since the effect of $\pi_{\omega_{N-m}}'$ on  $K_{\Lambda^s,\Lambda^a}(y)$ is to reorder the entries of $\Lambda^s$ in weakly increasing order.
Observe that there is no sign anymore in the previous equation due to the fact
that $\Delta_m(x)$ and $\Delta_{(N-m)^c}(y)$ differ by a factor 
$(-1)^{\binom{m}{2}}$.
\end{proof}

To prove the Pieri rules associated to the multiplication by a fermionic quantity 
(a superpolynomial of fermionic degree one) we need to find the image of
$\tilde e_\ell \, s_\Lambda$, $\tilde p_\ell \, s_\Lambda$ and $\tilde e_\ell \, \bar s_\Lambda$
in the
identification between $\Lambda_N^{m+1}$ and $\mathbb Q [x_1,\dots,x_N]^{S_{m+1} \times S_{(m+1)^c}}$.

\begin{lemma} \label{lemets}
Let $\Lambda$ be a superpartition of fermionic degree $m$.
In the identification \eqref{eqidenti} between $\Lambda_N^{m+1}$ and $\mathbb Q [x_1,\dots,x_N]^{S_{m+1} \times S_{(m+1)^c}}$
we have
\begin{equation}
\tilde e_\ell \, s_\Lambda \quad \longleftrightarrow \quad  
(-1)^{\binom{m+1}{2}} \partial_{\omega_{m+1}} e_\ell^{(m+1)} \pi_{\omega_{m^c}} \hat K_{(\Lambda^a)^R,\Lambda^s}(x)
\end{equation}
where $e_\ell^{(m+1)}$ stands for the elementary symmetric function $e_\ell$ without the variable $x_{m+1}$ (or equivalently, at $x_{m+1}=0$).
\end{lemma}
\begin{proof}
Using 
$
\tilde e_\ell = \sum_{i=1}^N  \theta_i e_\ell^{(i)}
$, we get
\begin{align} \label{eqetsym}
\tilde e_\ell \sum_{\sigma \in S_N /(S_m \times S_{m^c})} 
{\mathcal K}_{\sigma} \theta_1 \cdots \theta_m  A_{m}  \Big |_{\theta_1\cdots \theta_{m+1}}
 &=
\left( \theta_{m+1} e_\ell^{(m+1)}+\sum_{i=1}^{m}  \theta_i e_\ell^{(i)} {\mathcal K}_{i m+1}  \right) \theta_1 \cdots \theta_m A_{m} \Big |_{\theta_1\cdots \theta_{m+1}} \nonumber \\
& = (-1)^m \left(1+\sum_{i=1}^{m} {\mathcal K}_{i m+1} \right) \theta_1 \cdots \theta_{m+1} A_{m}  e_\ell^{(m+1)} \Big |_{\theta_1\cdots \theta_{m+1}} \nonumber \\
& = (-1)^m \left(1-\sum_{i=1}^{m} {\kappa}_{i m+1} \right) \theta_1 \cdots \theta_{m+1} A_{m}  e_\ell^{(m+1)} \Big |_{\theta_1\cdots \theta_{m+1}} \nonumber \\
&=(-1)^m \left(1- \sum_{i=1}^{m}  {\kappa}_{i m+1} \right) A_m e_\ell^{(m+1)}    
\nonumber \\
&=(-1)^m A_{m+1} e_\ell^{(m+1)} 
\end{align}
Note that in the third equality we used the fact that interchanging $\theta_i$ and $\theta_{m+1}$ introduces a sign (${\mathcal K}_{\sigma}$ permutes both $x$'s and $\theta$'s while $\kappa_{\sigma}$ only permutes $x$'s).  From \eqref{eqsdev},  
$\pi_{\omega_{m^c}}= \sum_{v \, \in S_{m^c}} \hat \pi_v$ and 
$A_{m+1} ={\Delta_{m+1}(x)} \partial_{\omega_{m+1}}$, we then deduce that
\begin{equation}
\frac{\tilde e_\ell s_\Lambda(x) \big|_{\theta_1\cdots \theta_{m+1}}}{\Delta_{m+1}(x)}
= (-1)^{m+\binom{m}{2}} \partial_{\omega_{m+1}} e_\ell^{(m+1)} \pi_{\omega_{m^c}} \hat K_{(\Lambda^a)^R,\Lambda^s}(x)
\end{equation}
\end{proof}

\begin{lemma} \label{lemetsb}
Let $\Lambda$ be a superpartition of fermionic degree $m$.
In the identification \eqref{eqidenti} between $\Lambda_N^{m+1}$ and $\mathbb Q [x_1,\dots,x_N]^{S_{m+1} \times S_{(m+1)^c}}$
we have
\begin{equation}
\tilde e_\ell \, \bar s_\Lambda \quad \longleftrightarrow \quad  
 (-1)^m \partial_{\omega_{(N-m-1)^c}}' e_\ell^{(N-m)}(y) K_{(\Lambda^s)^R,\Lambda^a}(y)
\end{equation}
where again $y$ stands for the variables 
$y_1=x_N,y_2=x_{N-1},\dots,y_N=x_1$, and 
the $'$ indicates that the divided differences act on the y variables.
 \end{lemma}
\begin{proof}
As was established in the proof of the previous lemma, we have
\begin{equation}
\tilde e_\ell \sum_{\sigma \in S_N /(S_m \times S_{m^c})} 
{\mathcal K}_{\sigma} \theta_1 \cdots \theta_m A_{m} \Big |_{\theta_1\cdots \theta_{m+1}}
=(-1)^m A_{m+1} e_\ell^{(m+1)} 
\end{equation}
 from which we deduce from \eqref{eqsbarinkey}, following the argument that led to \eqref{eqn:key1},  that
\begin{equation}
\frac{\tilde e_\ell \bar s_\Lambda(x) \big|_{\theta_1\cdots \theta_{m+1}}}{\Delta_{m+1}(x)}
= (-1)^m \partial_{\omega_{(N-m-1)^c}}' e_\ell^{(N-m)}(y) K_{(\Lambda^s)^R,\Lambda^a}(y)
\end{equation}
\end{proof}

\begin{lemma} \label{lempts}
Let $\Lambda$ be a superpartition of fermionic degree $m$.
In the identification \eqref{eqidenti} between $\Lambda_N^{m+1}$ and $\mathbb Q [x_1,\dots,x_N]^{S_{m+1} \times S_{(m+1)^c}}$
we have
\begin{equation}
\tilde p_\ell \, s_\Lambda \quad \longleftrightarrow \quad  
(-1)^{\binom{m+1}{2}} \partial_{\omega_{m+1}}  x_{m+1}^\ell \pi_{\omega_{m^c}} \hat K_{(\Lambda^a)^R,\Lambda^s}(x)
\end{equation}
 \end{lemma}
\begin{proof}
Recall that $\tilde p_\ell=\sum_{i=1}^N  \theta_i x_i^\ell$.
As in \eqref{eqetsym}, we have
\begin{equation}
\left(\sum_{i=1}^N  \theta_i x_i^\ell\right) \sum_{\sigma \in S_N /(S_m \times S_{m^c})} 
{\mathcal K}_{\sigma} \theta_1 \cdots \theta_m A_{m} \Big |_{\theta_1\cdots \theta_{m+1}} =(-1)^m A_{m+1}  \, x_{m+1}^\ell 
\end{equation}
We then deduce, as in the proof of Lemma~\ref{lemets}, that
\begin{equation}
\frac{\tilde p_\ell \, s_{\Lambda}(x,\theta) \big|_{\theta_1\cdots \theta_{m+1}}}{\Delta_{m+1}(x)}
= (-1)^{m+\binom{m}{2}} \partial_{\omega_{m+1}}  x_{m+1}^\ell \pi_{\omega_{m^c}} \hat K_{(\Lambda^a)^R,\Lambda^s}(x)
\end{equation}
\end{proof}

\section{Pieri rules} \label{secpieri}
Pieri rules for the multiplication of a Schur functions in superspace
$s_\Lambda$ or $\bar s_\Lambda$ by $e_\ell$ and $\tilde e_\ell$ will be
established in this section using Key polynomials.  As we will see, these Pieri rules
are essentially the transposed of
those corresponding to the multiplication of 
$s_\Lambda^*$ or $\bar s_\Lambda^*$ by $h_\ell$ and $\tilde h_\ell$.

\subsection{Pieri rules for $s_\Lambda^*$ and $\bar s_\Lambda$}  \label{subsec41}
In order to use Lemmas~\ref{lemsb} and \ref{lemetsb}, it will prove convenient
to express the Key polynomial appearing in these lemmas
as a sequence of operators $\pi_j$'s acting on a monomial.   
Letting $\pi_{(b,a)}=\pi_b\pi_{b-1}\cdots\pi_a$, for $b\geq a$ and $\pi_{(b,a)}=1$ if $b<a$, we get (see Lemma~\ref{lemreorder} in the Appendix for extra details)
\begin{equation} \label{eqlong}
K_{(\Lambda^s)^R,\Lambda^a}(x)=\pi_{\omega_{N-m}}\pi_{(N-m,\alpha_1)}\pi_{(N-m+1,\alpha_2)}\cdots\pi_{(N-1,\alpha_m)}x^{\Lambda^*}
\end{equation} where $\alpha_i$ is the row of the $i$-th circle (starting from the top) in $\Lambda$.
To simplify the notation, we set, for $\alpha_1<\alpha_2<\cdots <\alpha_m$, 
\begin{equation}
\mathcal{R}_{N,[\alpha_1,\dots,\alpha_m]}= \pi_{\omega_{N-m}}\pi_{(N-m,\alpha_1)}\pi_{(N-m+1,\alpha_2)}\cdots\pi_{(N-1,\alpha_m)}.
\end{equation}
from which we have
 \begin{equation}\label{id:KR}
 K_{(\Lambda^s)^R,\Lambda^a}(x)=\mathcal{R}_{N,[\alpha_1,\dots,\alpha_m]}x^{\Lambda^*}.
 \end{equation}

\begin{theorem}\label{thm:esbar} Let $\Lambda$ be a superpartition of fermionic degree $m$.  
Then, for $\ell \geq 1$, we have respectively
\begin{equation} \label{eqesbar} 
 s_{\Lambda}^* \, h_\ell= \sum_{\Omega}  s_{\Omega}^* \qquad \quad {\rm and }\qquad \quad  
\bar s_{\Lambda} \,  e_\ell = \sum_{\Omega} \bar s_{\Omega}
\end{equation}
where the sum is over all superpartitions $\Omega$ of fermionic degree $m$ 
such that
\begin{enumerate}
\item $\Omega^*/\Lambda^*$ is a horizontal (resp. vertical) $\ell$-strip. 
\item  The $i$-th circle, starting from below, of $\Omega$ is either in the same
row (resp. column) as the $i$-th circle of $\Lambda$ if $\Omega^*/\Lambda^*$ does not contain a cell in that row (resp. column) or 
one row below that (resp. one column to the right) of the $i$-th circle of $\Lambda$ if $\Omega^*/\Lambda^*$ contains a cell in the row (resp. column) of the
$i$-th circle of $\Lambda$.  In the latter case, we say that the circle was moved. 
\end{enumerate}
\end{theorem}

We illustrate the rules by giving the expansion of  $s_{(4,1,0;2)}^*\, h_3$. 

\medskip

{\tiny

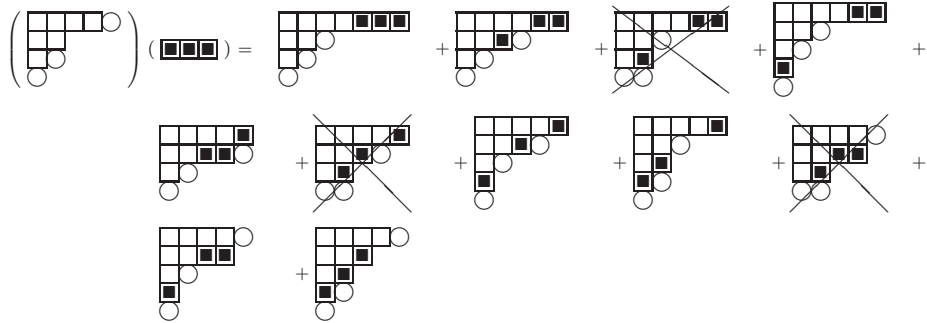
\begin{tabular}{rrrrrr} 
 $\left(~{\tableau[scY]{&&&&\bl\tcercle{}\\&\\&\bl\tcercle{}\\  \bl\tcercle{}}}~\right)
   \left( ~{\tableau[scY]{\blacksquare& \blacksquare& \blacksquare}}~\right)=$&
  $ {\tableau[scY]{&&&&\blacksquare&\blacksquare&\blacksquare\\&&\bl\tcercle{}\\&\bl\tcercle{}\\  \bl\tcercle{}}}$&+
  $ {\tableau[scY]{&&&&\blacksquare&\blacksquare\\&&\blacksquare&\bl\tcercle{}\\&\bl\tcercle{}\\  \bl\tcercle{}}}$&+
  $ \xcancel {\tableau[scY]{&&&&\blacksquare&\blacksquare\\&&\bl\tcercle{}\\&\blacksquare\\ \bl\tcercle{}& \bl\tcercle{}}}$&+
  $ {\tableau[scY]{&&&&\blacksquare&\blacksquare\\&&\bl\tcercle{}\\&\bl\tcercle{}\\ \blacksquare\\ \bl\tcercle{}}}$&+\\
  \\ 
  $ {\tableau[scY]{&&&&\blacksquare\\&&\blacksquare&\blacksquare&\bl\tcercle{}\\&\bl\tcercle{}\\  \bl\tcercle{}}}$&+
  $ \xcancel {\tableau[scY]{&&&&\blacksquare\\&&\blacksquare&\bl\tcercle{}\\&\blacksquare\\ \bl\tcercle{}& \bl\tcercle{}}}$&+
   $ {\tableau[scY]{&&&&\blacksquare\\&&\blacksquare&\bl\tcercle{}\\&\bl\tcercle{}\\ \blacksquare\\  \bl\tcercle{}}}$&+
   $ {\tableau[scY]{&&&&\blacksquare\\&&\bl\tcercle{}\\&\blacksquare\\ \blacksquare &\bl\tcercle{}\\  \bl\tcercle{}}}$&+
$   \xcancel{\tableau[scY]{&&&&\bl\tcercle{} \\ &&\blacksquare&\blacksquare \\ &\blacksquare \\ \bl\tcercle{} &\bl\tcercle{} \\  }}$&+\\
\\
${\tableau[scY]{&&&&\bl\tcercle{} \\ &&\blacksquare&\blacksquare \\ &\bl\tcercle{} \\ \blacksquare \\ \bl\tcercle{} \\  }}$&+
  $ {\tableau[scY]{&&&&\bl\tcercle{}\\&&\blacksquare\\&\blacksquare\\ \blacksquare &\bl\tcercle{}\\  \bl\tcercle{}}}$
  
\end{tabular}}
$$s_{(4,1,0;2)}^*\, h_3=s_{(2,1,0;7)}^*+ s_{(3,1,0;6)}^*+ s_{(2,1,0;6,1)}^*+ s_{(4,1,0;5)}^*+ s_{(3,1,0;5,1)}^*+ s_{(2,1,0;5,2)}^*+
 s_{(4,1,0;4,1)}^*+s_{(4,1,0;3,2)}^*$$

To generate all $\Omega$'s described in Theorem~\ref{thm:esbar}, draw all possible horizontal strips on $\Lambda^*$.
 For each partition obtained this way, start from the bottom row and proceed row by row.
 If a new square occupies the place of a circle in $\Lambda^{\circledast}$,
 move the circle to the next row and slide it to the first available column.
  If there already is a circle occupying the row then the resulting diagram
is not a superpartition and should be discarded. Note that this happens twice in our example.

\begin{proof}  Applying the involution $\omega$ on $\bar s_\Lambda \, e_\ell$, we obtain 
from Proposition~\ref{propdual} and \eqref{hedual} that
\begin{equation}
\bar s_\Lambda \, e_\ell = \sum_\Omega  \bar s_{\Omega} \quad \iff \quad 
s_{\Lambda'}^* \, h_\ell = \sum_\Omega  s_{\Omega'}^* 
\end{equation}
The superpartitions $\Omega$ appearing in the expansion of
$\bar s_\Lambda \, e_\ell$ are thus simply the transposes of those appearing in the expansion of
$s_\Lambda^* \, h_\ell$.  It thus suffices to prove  the rule for the multiplication of $\bar s_\Lambda$ by $e_\ell$.
 From Lemma~\ref{lemsb} and the commutativity between $\bar s_\Lambda$ and $e_\ell$, 
this is equivalent to proving the statement
\begin{equation}\label{eqn:thmkeyeqn1}
 e_\ell \, \partial_{\omega_{(N-m)^c}} K_{(\Lambda^s)^R,\Lambda^a}(x)= \sum_{\Omega} \partial_{\omega_{(N-m)^c}} K_{(\Omega^s)^R,\Omega^a}(x)
\end{equation}
where the sum is over all superpartitions $\Omega$ described above.
Note that for simplicity we changed $y$ by $x$ and removed 
the $'$ on the divided differences.

Using (\ref{id:KR}) and the fact that $e_{\ell}$  
commutes with all the operators,  it is easily seen that we can rewrite the 
left hand side of (\ref{eqn:thmkeyeqn1}) as
\begin{equation}
e_\ell \, \partial_{\omega_{(N-m)^c}} K_{(\Lambda^s)^R,\Lambda^a}(x)=\partial_{\omega_{(N-m)^c}}\mathcal{R}_{N,[\alpha_1,\dots,\alpha_m]}x^{\Lambda^*} e_\ell
\end{equation}

Suppose that ${\Lambda^*}$ is a partition whose entries are of
exactly $k$ distinct sizes, that is, 
$\Lambda^*=(a_1^{\ell_1},\dots,a_k^{\ell_k})$ with $a_1>a_2>\cdots > a_k$,
and let $I_1, I_2, \dots, I_k$ be intervals such that the entries $a_1,a_2,\dots,a_k$ occupy the rows 
$I_1, I_2, \dots, I_k$ respectively in $\Lambda^*$. For example, if the partition is  $\Lambda^*=(5,5,5,4,3,3,3,3,1)$ then $I_1=[1,3]$, $I_2=[4]$, $I_3=[5,8]$ and $I_4=[9]$.

For  $I$ equal to the interval $[a,b]$ and $i\leq |I|$, we
let $x^{(i,I)}=x_a x_{a+1} \cdots x_{a+i-1}$ and
\begin{equation}
\boldsymbol{\pi}_{i,I} =      \pi_{[b-i,b-1]}   \cdots \pi_{[a+1,a+i]} \pi_{[a,a+i-1]}
\end{equation}
where $\pi_{[c,d]}=\pi_c \pi_{c+1} \cdots \pi_{d}=\pi_{(d,c)}^{-1}$ for $c\leq d$ and $\pi_{[c,d]}=1$ if $c>d$
(note that $\boldsymbol{\pi}_{i,I}=1$ if $i=|I|$).
As shown in Lemma~\ref{lemapelx} of the Appendix, we have the expansion
\begin{equation} \label{eqxe}
x^{\Lambda^*} e_\ell=
\sum_{i_1+i_2+\dots+i_k=\ell}
\boldsymbol{\pi}_{i_1,I_1}  \boldsymbol{\pi}_{i_2,I_2} \cdots \boldsymbol{\pi}_{i_k,I_k} 
x^{\Lambda^*} x^{(i_1,I_1)}x^{(i_2,I_2)}\cdots x^{(i_k,I_k)}
\end{equation}
The left hand side side of (\ref{eqn:thmkeyeqn1}) thus becomes
\begin{equation} \label{eqpiiR}
\sum_{i_1+i_2+\dots+i_k=\ell}\partial_{\omega_{(N-m)^c}}\mathcal{R}_{N,[\alpha_1,\dots,\alpha_m]}
\boldsymbol{\pi}_{i_1,I_1}  \boldsymbol{\pi}_{i_2,I_2} \cdots \boldsymbol{\pi}_{i_k,I_k} 
x^{\Lambda^*} x^{(i_1,I_1)}x^{(i_2,I_2)}\cdots x^{(i_k,I_k)}
\end{equation}

It is shown in the appendix that   $\mathcal{R}_{N,[\alpha_1,\dots,\alpha_m]} \pi_j
=\mathcal{R}_{N,[\alpha_1,\dots,\alpha_m]}$ whenever $j\neq \alpha_i-1$ for all $i$.
By definition, 
the circles in $\Lambda^{\circledast}$  can only occur at the beginning of
the intervals $I_1,\dots,I_k$.  
Therefore, a value $j = \alpha_i-1$ would correspond to the end of one such
interval. Now, the $\pi_r$'s that occur in  
a given ${\boldsymbol \pi}_{s,I}$ are such that $r$ belongs to $[a,b-1]$ (supposing that $I=[a,b]$), which implies that they cannot coincide with 
any $\pi_{\alpha_i-1}$.  Hence,
every $\boldsymbol{\pi}_{i_j,I_j}$
in  \eqref{eqpiiR} acts as the identity on $\mathcal{R}_{N,[\alpha_1,\dots,\alpha_m]}$, 
and we have
\begin{equation}
e_\ell \,\partial_{\omega_{(N-m)^c}}  K_{(\Lambda^s)^R,\Lambda^a}(x)=\sum_{i_1+i_2+\dots+i_k=\ell}\partial_{\omega_{(N-m)^c}}\mathcal{R}_{N,[\alpha_1,\dots,\alpha_m]}x^{\Lambda^*} x^{(i_1,I_1)}x^{(i_2,I_2)}\dots x^{(i_k,I_k)}
\end{equation}
Notice that this summation sums over all vertical $\ell$-strips and so we get
\begin{equation} \label{eqverti}
e_\ell \, \partial_{\omega_{(N-m)^c}}  K_{(\Lambda^s)^R,\Lambda^a}(x)=
\sum_{\substack{\mu\\ \mu/{\Lambda^*} \text{ vertical $\ell$-strip}}}
\partial_{\omega_{(N-m)^c}}\mathcal{R}_{N,[\alpha_1,\dots,\alpha_m]}x^\mu
\end{equation}
We will now show that \eqref{eqn:thmkeyeqn1} holds 
(and thus also the theorem)
 by showing
that 
\begin{equation}
\sum_{\substack{\mu\\ \mu/{\Lambda^*} \text{ vertical $\ell$-strip}}}\partial_{\omega_{(N-m)^c}}\mathcal{R}_{N,[\alpha_1,\dots,\alpha_m]}x^\mu=\sum_{\Omega}\partial_{\omega_{(N-m)^c}} K_{(\Omega^s)^R,\Omega^a}(x)
\end{equation}
where the sum runs over the  $\Omega$'s described in the statement of the
theorem.  
First of all, part (1) of Theorem \ref{thm:esbar} is clearly satisfied since by construction we have that $\Omega^*=\mu$ and $\mu/\Lambda^*$ is a vertical $\ell$-strip.  Second, for a given vertical strip, there can be at most 
one $\Omega$ in the statement of the theorem.  It thus suffices to show 
that for a given vertical strip $\mu/\Lambda^*$, we have that
$\partial_{\omega_{(N-m)^c}}\mathcal{R}_{N,[\alpha_1,\dots,\alpha_m]}x^\mu$ is either equal 
to zero (if there is no corresponding $\Omega$)
or to $\partial_{\omega_{(N-m)^c}} K_{(\Omega^s)^R,\Omega^a}(x)$ for the
right value of $\Omega$ (see part (2) of Theorem \ref{thm:esbar}).
We will in fact see that if $\partial_{\omega_{(N-m)^c}}\mathcal{R}_{N,[\alpha_1,\dots,\alpha_m]}x^\mu$ is not equal to zero then
\begin{equation}
\mathcal{R}_{N,[\alpha_1,\dots,\alpha_m]}x^\mu
= \mathcal{R}_{N,[\alpha_1',\dots,\alpha_m']}x^\mu =
K_{(\Omega^s)^R,\Omega^a}(x).
\end{equation}
where $\Omega$ is the superpartition obtained from $\mu$
by adding circles in rows $\alpha_1',\dots, \alpha_m'$.

We simply need to focus on what happens to $\alpha_i$ in 
$\partial_{\omega_{(N-m)^c}}\mathcal{R}_{N,[\alpha_1,\dots,\alpha_m]}x^\mu$.
Note that circles can only occupy rows that are addable corners of
$\Lambda^*$.  The possible cases are:  
\begin{enumerate}
\item
If there is no cell in row $\alpha_i$ of $\mu/\Lambda^*$  
then $\alpha_i'=\alpha_i$ and there is a circle in row 
$\alpha_i$ of $\Omega$ (hence the $i$-th
circle is in the same position in $\Lambda$ and in $\Omega$).
\item
If there is a cell in row $\alpha_i$ of $\mu/\Lambda^*$ and row $\alpha_i$ is still an addable corner of $\mu$
then there will be a circle in row $\alpha_i$ of $\Omega$ (hence the $i$-th
circle of $\Omega$ is one column to the right of the $i$-th circle
of $\Lambda$ since $\mu/\Lambda^*$ is vertical).
\item
If there is a cell in row $\alpha_i$ of $\mu/\Lambda^*$ and row $\alpha_i$ is not an addable corner of $\mu$
then there cannot be a circle in row $\alpha_i$ of $\Omega$. But this is 
okay since in this case there exists a row, call it $\alpha_i'$ that has
the same length as row $\alpha_i$ and is an addable corner, and such that 
by Corollary~\ref{coroalphap} of the Appendix
$$\mathcal{R}_{N,[\alpha_1,\dots,\alpha_i,\dots,\alpha_m]}x^\mu=
\mathcal{R}_{N,[\alpha_1,\dots,\alpha_i',\dots,\alpha_m]}x^\mu$$
Since $\mu/\Lambda^*$ is vertical, we thus have that again the $i$-th
circle of $\Omega$ is one column to the right of the $i$-th circle
of $\Lambda$.

If, however, there is a circle in row $\alpha'_i$, that is $\alpha_{i-1}=\alpha'_i$, then we have
$$\mathcal{R}_{N,[\alpha_1,\dots,\alpha_{i-1},\alpha_i,\dots,\alpha_m]}x^\mu=
\mathcal{R}_{N,[\alpha_1,\dots,\alpha_i',\alpha_i',\dots,\alpha_m]}x^\mu=0$$
according to Lemma \ref{lemequal} in the Appendix.

\end{enumerate}
Therefore, part (2) of Theorem \ref{thm:esbar} is satisfied, which completes
the proof of the theorem.

\end{proof}

\medskip

\begin{theorem}\label{pinball2} Let $\Lambda$ be a superpartition of fermionic degree $m$.  
Then, for $\ell \geq 0$, we have respectively
\begin{equation}  \label{eqtwo}
 s_{\Lambda}^* \, \tilde h_\ell= \sum_{\Omega}  (-1)^{\#(\Omega,\Lambda)} 
s_{\Omega}^* \qquad \quad {\rm and }\qquad \quad \bar s_{\Lambda} \, \tilde e_\ell = \sum_{\Omega} (-1)^{\#(\Omega,\Lambda)} \bar s_{\Omega}
\end{equation}
where the sum is over all superpartitions $\Omega$ of fermionic degree $m+1$ 
such that
\begin{enumerate}
\item 
$\Omega^*/\Lambda^*$ is a horizontal (resp. vertical) $\ell$-strip. 
\item There exists a unique circle of $\Omega$ (the new circle), let's say in column $c$ (resp. row $r$), such that
\begin{itemize}
\item column $c$ (resp. row $r$) does not contain any cell
of $\Omega^*/\Lambda^*$. 
\item there is a cell of $\Omega^*/\Lambda^*$ in 
every column (resp. row) strictly to the left of column $c$ 
(resp. strictly above row $r$). 
\end{itemize}
\item If $\tilde \Omega$ is $\Omega$ without its new circle,
then the $i$-th circle, starting from below, of $\tilde \Omega$ is either in the same
row (resp. column) as the $i$-th circle of $\Lambda$ if $\Omega^*/\Lambda^*$ does not contain a cell in that row (resp. column) or 
one row below that (resp. one column to the right) of the $i$-th circle of $\Lambda$ if $\Omega^*/\Lambda^*$ contains a cell in the row (resp. column) of the
$i$-th circle of $\Lambda$. In the latter case, we say that the circle was moved. 
\end{enumerate}
and where ${\#}(\Omega, \Lambda)$ is the number of circles in $\Omega$
below the new circle.
\end{theorem}

We illustrate this time the rules by giving the expansion of  
$s_{(4,1;3)}^*\, \tilde h_3$. 

\medskip

{\tiny

\begin{tabular}{rrrrrrr} 
 $\left(~{\tableau[scY]{&&&&\bl\tcercle{}\\&&\\&\bl\tcercle{}  }}~\right)
   \left( ~{\tableau[scY]{\blacksquare& \blacksquare& \blacksquare &\bl\tcercle{$\bullet$}}}~\right)=$&
  $ {\tableau[scY]{&&&&\blacksquare&\blacksquare&\blacksquare\\&&&\bl\tcercle{}\\&\bl\tcercle{}\\ \bl\tcercle{$\bullet$}  }}$&+
  $  {\tableau[scY]{&&&&\blacksquare&\blacksquare\\&&&\blacksquare&\bl\tcercle{}\\&\bl\tcercle{}\\ \bl\tcercle{$\bullet$}  }}$&
$-~{\tableau[scY]{&&&&\blacksquare\\&&& \bl\tcercle{}\\&\blacksquare&\bl\tcercle{$\bullet$} \\ \blacksquare&\bl\tcercle{} }}$&
 $ -~ {\tableau[scY]{&&&&\bl\tcercle{}\\&&&\blacksquare \\&\blacksquare&\bl\tcercle{$\bullet$} \\ \blacksquare&\bl\tcercle{} }}$&
   $-~ {\tableau[scY]{&&&&\bl\tcercle{}\\&&&\bl\tcercle{$\bullet$} \\&\blacksquare &\blacksquare\\ \blacksquare&\bl\tcercle{} }}$&
  
\end{tabular}}
$$s_{(4,1;3)}^* \, \tilde h_3= s_{(3,1,0;7)}^*+  s_{(4,1,0;6)}^* - s_{(3,2,1;5)}^* - s_{(4,2,1;4)}^*- s_{(4,3,1;3)}^*$$

This rule is very similar to the previous one but in this case keep in mind that
 every column to the left of the new circle must have a new box. Also, multiply by $(-1)$ for every circle below the new circle.

\begin{proof}
 Applying the involution $\omega$, we obtain
from Proposition~\ref{propdual} and \eqref{hedual} that
\begin{equation}
\bar s_\Lambda \, \tilde e_\ell = \sum_\Omega (-1)^{\#(\Omega,\Lambda)} \bar s_{\Omega} \quad \iff \quad 
s_{\Lambda'}^* \, \tilde h_\ell = (-1)^m \sum_\Omega (-1)^{\#(\Omega,\Lambda)} s_{\Omega'}^* 
\end{equation}
The superpartitions $\Omega$ appearing in the expansion of
$\bar s_\Lambda \, \tilde e_\ell$ are thus simply the transposes of those appearing in the expansion of
$s_\Lambda^* \, \tilde h_\ell$.  Observe that $(-1)^{\#(\Omega,\Lambda)+m}=(-1)^{\#(\Omega',\Lambda')}$, which means
that the sign associated to $\Omega$ is the same in both formulas of \eqref{eqtwo}. Hence, it suffices to prove  the rule for the multiplication of $\bar s_\Lambda$ by $\tilde e_\ell$.
Using $\tilde e_\ell \, \bar s_{\Lambda}= (-1)^m \bar s_{\Lambda} \, \tilde e_\ell$, this is equivalent to proving that 
\begin{equation}
\tilde e_\ell \, \bar s_{\Lambda} =  \sum_{\Omega} (-1)^{\#(\Omega,\Lambda)'} \bar s_{\Omega}
\end{equation}
where $\#(\Omega,\Lambda)'$ is equal to the number of circles in 
$\Omega$ {\it above} the
new circle. From Lemmas~\ref{lemsb} and \ref{lemetsb}, this corresponds in the bisymmetric world  $\mathbb Q [x_1,\dots,x_N]^{S_{m+1} \times S_{(m+1)^c}}$ to proving
that \begin{equation}\label{eqn:thmkeyeqn2}
(-1)^m\partial_{\omega_{(N-m-1)^c}} e^{(N-m)}_\ell K_{(\Lambda^s)^R,\Lambda^a}(x)= 
\sum_{\Omega} (-1)^{\#(\Omega,\Lambda)'} \partial_{\omega_{(N-m-1)^c}} K_{(\Omega^s)^R,\Omega^a}(x)
\end{equation}
where the sum is over all superpartitions $\Omega$ described in the statement of the theorem.
From (\ref{id:KR}), we have
 \begin{equation}
 (-1)^m\partial_{\omega_{(N-m-1)^c}} e^{(N-m)}_\ell K_{(\Lambda^s)^R,\Lambda^a}(x)=
 (-1)^m\partial_{\omega_{(N-m-1)^c}} e^{(N-m)}_\ell\mathcal{R}_{N,[\alpha_1,\dots,\alpha_m]}x^{\Lambda^*}
\end{equation}
Using Corollary~\ref{corobigexpan} of the Appendix, we get 
\begin{align}
& {\partial_{\omega_{(N-m-1)^c}}} e^{(N-m)}_\ell\mathcal{R}_{N,[\alpha_1,\dots,\alpha_m]}= 
\nonumber \\
& \qquad \qquad { \partial_{\omega_{(N-m-1)^c}}} \sum_{r\not \in \{\alpha_1,\dots,\alpha_m \}}
(-1)^{\text{pos}(r)}\mathcal{R}_{N,[\alpha_1,\ldots,r,\ldots, \alpha_m]}x_1\dots x_{r-1}e_{\ell-r+1}^{(1,2,\dots,r)}\\
\end{align}
where $e_{\ell}^{(1,2,\dots,r)}$ is $e_\ell$ with $x_1=x_2=\dots=x_r=0$, and where
$\text{pos}(r)$ is the position of $r$ starting from $0$ in the increasing sequence
 $[\alpha_1,\ldots,r,\ldots, \alpha_m]$, that is, $\text{pos}(r)$ is the number of entries of $[\alpha_1,\ldots,r,\ldots, \alpha_m]$ that
 are less than $r$. For example, if we have the increasing sequence $[2,3,9,13,19,22]$ then $\text{pos}(13)=3$ because there are 
 $3$ numbers less than $13$ in the sequence.
 Therefore, we have
\begin{align}
&(-1)^m\partial_{\omega_{(N-m-1)^c}} e^{(N-m)}_\ell K_{(\Lambda^s)^R,\Lambda^a}(x)= \nonumber \\
& \qquad \qquad (-1)^m\partial_{\omega_{(N-m-1)^c}}\sum_{r \not \in \{\alpha_1,\ldots, \alpha_m \}}
(-1)^{\text{pos}(r)}\mathcal{R}_{N,[\alpha_1,\ldots,r,\ldots, \alpha_m]} \, 
e_{\ell-r+1}^{(1,2,\dots,r)}x^{\Lambda^* +1^{r-1}}
\end{align}
where $\Lambda^*+1^{r-1}$ is the partition obtained by adding a new cell in the first $r-1$ rows of $\Lambda^*$.
Using an argument similar to the one that led to \eqref{eqverti} 
in the proof of 
Theorem~\ref{thm:esbar},
we can show that
\begin{equation}
\mathcal{R}_{N,[\alpha_1,\ldots,r,\ldots, \alpha_m]} \, e_{\ell-r+1}^{(1,2,\dots,r)}x^{\Lambda^*+1^{r-1}}=
\sum_{\mu} \mathcal{R}_{N,[\alpha_1,\ldots,r,\ldots, \alpha_m]} \, x^{\mu+1^{r-1}}
\end{equation}
where the sum is over all $\mu$'s such that $\mu/\Lambda^*$
is a vertical $(\ell-r+1)$-strip whose cells
are all in rows strictly below row $r$.  From the previous two equations, the right hand 
side of equation (\ref{eqn:thmkeyeqn2}) is seen to be equal to
\begin{equation}
\sum_{r \not \in \{\alpha_1,\dots,\alpha_m \}}
\sum_{\mu}
\partial_{\omega_{(N-m-1)^c}}(-1)^{m+\text{pos}(r)}\mathcal{R}_{N,[\alpha_1,\ldots,r,\ldots,\alpha_m]} \, x^{\mu+1^{r-1}}
\end{equation}
where the sum runs over the $\mu$'s described in the previous equation.
 Now, as in the proof of Theorem~\ref{thm:esbar}, we have that if 
$\partial_{\omega_{(N-m-1)^c}} \mathcal{R}_{N,[\alpha_1,\ldots,r,\dots, \alpha_m]} \, x^{\mu+1^{r-1}} \neq 0$ then 
\begin{equation}
\mathcal{R}_{N,[\alpha_1,\dots,r,\ldots, \alpha_m] }\, x^{\mu+1^{r-1}}
= \mathcal{R}_{N,[\alpha_1',\dots,r,\dots, \alpha_m'] }\, x^{\mu+1^{r-1}} =
K_{(\Omega^s)^R,\Omega^a}(x)
\end{equation}
where $\Omega$ is the superpartition obtained from $\mu+1^{r-1}$
by adding circles in certain rows $\alpha_1',\dots,r,\dots,  \alpha_m'$.
Observe that we can always add a circle in row $r$ since there is an addable corner in row $r$ of $\mu+1^{r-1}$.
This circle corresponds to the new circle in row $r$ mentioned in part (2) of  
Theorem~\ref{pinball2}.  Notice that in $(\mu+1^{r-1})/\Lambda^*$
there is, by definition, no cell in row $r$ and  a new cell in every
row above row $r$, as required in 
part (2) of  the theorem.  The sign is also correct given that
 $(-1)^{m+({\text{pos}(r)})}$ is the same as $(-1)^{\#(\Omega,\Lambda)'}$ (which
corresponds to the number of circles above row $r$).  Furthermore, the sum is over all
vertical strips that contain cells in every row above row $r$ and none
in row $r$ as required by part (1) and (2) of the theorem.  Now,
if  $\alpha_i < r$ then row $\alpha_i$ is still an addable corner in row $\alpha_i$ of
$\mu+1^{r-1}$.  The corresponding circle is thus moved one column to the right
given the new cell in its row.  Finally, for the $\alpha_i$ below row $r$,
we proceed exactly as in the proof of the previous theorem to show that
part (3) needs to be satisfied.

Therefore, the sum over $\Omega$ in the right hand side of (\ref{eqn:thmkeyeqn2}) is indeed 
as 
in the statement of Theorem~\ref{pinball2}, which concludes the proof.
\end{proof}

The following corollary shows that a product of fermionic symmetric functions $\tilde e_\ell$ or $\tilde h_\ell$
corresponds to a single
Schur function in superspace.
\begin{corollary}  \label{coroesb}
We have
$e_{(\Lambda^a; \emptyset)} = (-1)^{\binom{m}{2}}\bar s_{(\Lambda^a; \emptyset)'}$ and $h_{(\Lambda^a;\emptyset)}=s_{(\Lambda^a; \emptyset)}^*$. 
In particular, $\tilde e_\ell= \bar s_{(0;1^\ell)}$ and $\tilde h_\ell= s_{(\ell;\emptyset)}^*$. 
\end{corollary}
\begin{proof}
Let $\lambda=\Lambda^a$ and $\Lambda=(\lambda;\emptyset)$.  We first show that $h_\Lambda=s_\Lambda^*$.
We proceed by induction on $m$.  If $m=1$, the result holds since
\begin{equation}
\tilde h_{\lambda_1}= s_{(\emptyset;\emptyset)}^*\,  \tilde h_{\lambda_1} = s_{(\lambda_1;\emptyset)}^* 
\end{equation}
from Theorem~\ref{pinball2} (starting from the empty superpartition, the only superpartition $\Omega$ satisfying the conditions of the theorem is obviously the one consisting of a row with $\lambda_1$ squares followed by a circle).  

Now, let $\hat \lambda=(\lambda_1,\lambda_2,\dots,\lambda_{m-1})$ 
be the partition $\lambda$ without its
last part.  By induction, we simply need to show that 
\begin{equation}
 {s}_{(\hat \lambda;\emptyset)}^*\, \tilde h_{\lambda_m} ={s}_{\Lambda}^* 
\end{equation}
By Theorem \ref{pinball2}, this corresponds to showing that in 
$$ {s}_{(\hat \lambda; \emptyset)}^* \, \tilde h_{\lambda_m}=\sum_{\Omega}  (-1)^{\#(\Omega,\hat{\Lambda})}s_{\Omega}^*$$
the sum only contains the term ${s}_{\Lambda}^*$. Observe that this is precisely the superpartition 
obtained by placing the new circle in column $\lambda_{m}+1$ and having a new square in every column to the left of the new circle.

The new circle cannot be in a column $c>\lambda_m+1$ because there are not enough new squares to fill every column strictly to the left of 
column $c$. So the new circle must be in row $m$ seeing as $\lambda_m<\lambda_{m-1}$ (this is possible given
that there is no circle in row $m$ of $\hat \Lambda=(\hat \lambda;\emptyset)$).  

Suppose by contradiction that the new
circle is in column $c\le\lambda_m$.  Then there are new squares in every column strictly to the left of cell $c$. Since
$c\le\lambda_m$, there are some extra new squares that need to be placed to the right of cell $c$ in a horizontal strip. For each possible horizontal strip 
to the right of column $c$, the leftmost new square will replace an existing circle (since every row of $\hat \Lambda$ ends with a circle) and move it down.  But the resulting diagram cannot be a superpartition since
the highest such horizontal strip will move a circle down in a row already containing a circle given that the new circle is
in row $m$ and $\hat \Lambda$ has circles in every row. Therefore, the only resulting superpartition is $\Lambda$, and so we have
$$ {s}_{(\hat \lambda; \emptyset)}^*\, \tilde h_{\lambda_m}=  (-1)^{\#(\Lambda,\hat{\Lambda})}s_{\Lambda}^*$$
The result then follows since  $(-1)^{\#(\Lambda,\hat{\Lambda})}=1$  given that
the new circle is in the last row.

Applying the homomorphism $\omega$ on both sides of  $h_{\Lambda} = s_{\Lambda}^*$ we immediately get from
 Proposition~\ref{propdual} and \eqref{hedual} 
that $e_{\Lambda} = (-1)^{\binom{m}{2}}\bar s_{\Lambda'}$.
\end{proof}

\subsection{Pieri rules for $\bar s_\Lambda^*$ and $s_\Lambda$} \label{secpieri2}
This time, we express the Key polynomial appearing in \eqref{eqn:keyhat} as a sequence of operators
$\hat \pi_j$'s acting on a monomial.   
Letting $\hat \pi_{[a,b]}=\hat \pi_a \hat \pi_{a+1}\cdots \hat \pi_b$, for $a \leq b$ and $\hat \pi_{[a,b]}=1$ if $a>b$, we get
\begin{equation} \label{eq428}
\hat{K}_{(\Lambda^a)^R,\Lambda^s}(x)=\hat{\pi}_{\omega_m}\hat{\pi}_{[m, {\alpha_m-1}]}\cdots\hat{\pi}_{[1,{\alpha_1-1}]}x^{\Lambda^*}
\end{equation} where $\alpha_i$ is again the row of the $i$-th circle (starting from the top) in $\Lambda$.  Note that
it is understood that $\hat \pi_{[i,\alpha_i-1]}=1$ if $\alpha_i=i$.  We set, for $\alpha_1<\dots<\alpha_m$, 
\begin{equation}
 \mathcal{P}_{N,[\alpha_1,\dots,\alpha_m]}= \partial_{\omega_m} \pi_{\omega_{m^c}}
\hat \pi_{\omega_m}
\hat{\pi}_{[m, {\alpha_m-1}]}\cdots\hat{\pi}_{[1,{\alpha_1-1}]}
\end{equation}
which implies
 \begin{equation}\label{id:KP}
 \partial_{\omega_{m}} \pi_{\omega_{m^c}} \hat{K}_{(\Lambda^a)^R,\Lambda^s}(x)=\mathcal{P}_{N,[\alpha_1,\dots,\alpha_m]}x^{\Lambda^*}
 \end{equation}

\begin{theorem} \label{thm:es}
Let $\Lambda$ be a superpartition of fermionic degree $m$.  Then, for $\ell \geq 1$, we have respectively
\begin{equation} \label{eons}
   \bar s_{\Lambda}^* \, h_\ell = \sum_{\Omega} \bar s_{\Omega}^* \qquad {\rm and} \qquad  s_{\Lambda} \, e_\ell = \sum_{\Omega}   s_{\Omega}
\end{equation}
where the sum is over all superpartitions $\Omega$ of fermionic degree $m$ 
such that
\begin{itemize}
\item  $\Omega^*/\Lambda^*$ is a horizontal (resp. vertical) $\ell$-strip 
\item 
the $i$-th circles, starting from the first row and going down, 
of $\Omega$ and
$\Lambda$ are either in the same row or the same column.  In the
case where the $i$-th circles are in the same row $r$ (resp. column $c$), 
the circle in row $r$ (resp. column $c$ ) of $\Omega$ cannot be located passed row $r-1$ (resp. column $c-1$) of $\Lambda$
(if $r=1$ or $c=1$ the condition does not apply).
\end{itemize}
\end{theorem}

We illustrate the rules by giving the expansion of
$ \bar s_{(4,1;5,4)}^* \, h_3$.

\medskip

{\Tiny

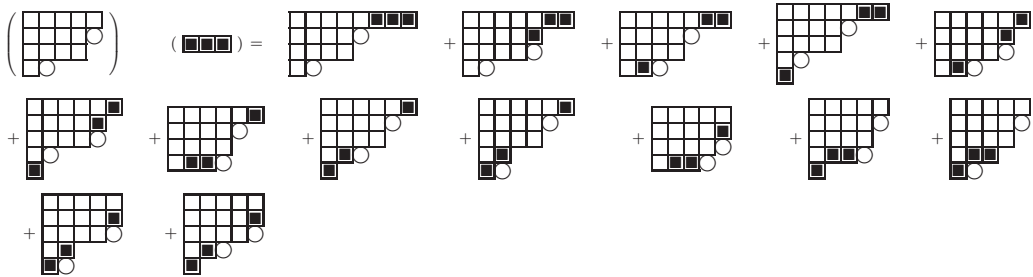
\begin{tabular}{rrrrrrrr} 
 $\left(~{\tableau[scY]{&&&&\\&&&&\bl\tcercle{}\\&&& \\ &\bl\tcercle{}   }}~\right)$ &
   $\left( ~{\tableau[scY]{\blacksquare& \blacksquare& \blacksquare}}~\right)=$&
   
  ${\tableau[scY]{&&&& & \blacksquare & \blacksquare & \blacksquare \\&&&&\bl\tcercle{}\\&&& \\ &\bl\tcercle{}   }}$&
    $+~ {\tableau[scY]{&&&& & \blacksquare & \blacksquare  \\&&&& \blacksquare \\&&&&  \bl\tcercle{} \\ &\bl\tcercle{}   }}$&
    $+~ {\tableau[scY]{&&&& & \blacksquare & \blacksquare  \\&&&&  \bl\tcercle{} \\&&& \\ &  \blacksquare &\bl\tcercle{}   }}$&
      $+~ {\tableau[scY]{&&&& & \blacksquare & \blacksquare  \\&&&&  \bl\tcercle{} \\&&& \\  &\bl\tcercle{}   \\  \blacksquare}}$&
    $+~ {\tableau[scY]{&&&& & \blacksquare  \\&&&& \blacksquare \\&&&&  \bl\tcercle{} \\  & \blacksquare &\bl\tcercle{}   }}$&
\\
\\
  $+~ {\tableau[scY]{&&&& & \blacksquare  \\&&&& \blacksquare \\&&&&  \bl\tcercle{} \\   &\bl\tcercle{}  \\  \blacksquare }}$&
  $+~ {\tableau[scY]{&&&& & \blacksquare  \\&&& & \bl\tcercle{} \\&&& \\   &    \blacksquare & \blacksquare & \bl\tcercle{}  }}$&
  $+~ {\tableau[scY]{&&&& & \blacksquare  \\&&& &  \bl\tcercle{}  \\&&&\\  & \blacksquare &\bl\tcercle{} \\ \blacksquare  }}$&
 $+~ {\tableau[scY]{&&&& & \blacksquare  \\&&& &  \bl\tcercle{}  \\&&&\\  & \blacksquare   \\ \blacksquare  &\bl\tcercle{} }}$&
 $+~ {\tableau[scY]{&&&&  \\&&&& \blacksquare \\&&&&  \bl\tcercle{} \\ & \blacksquare & \blacksquare  &\bl\tcercle{}   }}$&
    $+~ {\tableau[scY]{&&&&  \\&&& &  \bl\tcercle{}  \\&&&\\  & \blacksquare  & \blacksquare &\bl\tcercle{} \\ \blacksquare  }}$&
   $+~ {\tableau[scY]{&&&&   \\&&& &  \bl\tcercle{}  \\&&&\\  & \blacksquare & \blacksquare  \\ \blacksquare  &\bl\tcercle{} }}$&\\
   \\
   $+~{\tableau[scY]{&&&& \\ &&&&\blacksquare \\ &&&&\bl\tcercle{} \\ &\blacksquare \\ \blacksquare&\bl\tcercle{} \\  }}$&
   $+~{\tableau[scY]{&&&& \\ &&&&\blacksquare \\ &&&&\bl\tcercle{} \\ &\blacksquare&\bl\tcercle{} \\ \blacksquare \\  }}$
  
\end{tabular}}

\begin{align}
 \bar s_{(4,1;5,4)}^* \, h_3=& \, \bar s_{(4,1;8,4)}^*+\bar s_{(4,1;7,5)}^*+ \bar s_{(4,2;7,4)}^*+  \bar s_{(4,1;7,4,1)}^* +\bar s_{(4,2;6,5)}^* +\bar  s_{(4,1;6,5,1)}^*   \nonumber \\
& \qquad +\bar s_{(4,3;6,4)}^* + \bar s_{(4,2;6,4,1)}^*+
 \bar s_{(4,1;6,4,2)}^* +\bar s_{(4,3;5,5)}^* +\bar s_{(4,3;5,4,1)}^*+\bar s_{(4,1;5,4,3)}^* \nonumber\\
 & \qquad +\bar s_{(4,1;5,5,2)}^* + \bar s_{(4,2;5,5,1)}^* \nonumber
 \end{align}
 
Notice here that the circle can be pushed down if there is room or pushed to the right. It cannot be pushed to the right farther than the
original row above it.

\begin{proof}
Applying the involution $\omega$ on $s_\Lambda \, e_\ell$, we deduce
from Proposition~\ref{propdual} that the superpartitions $\Omega$ appearing in the expansion of
$s_\Lambda \, e_\ell$ are simply the transposed of those appearing in the expansion of
$\bar s_\Lambda^* \, h_\ell$ (we used a similar argument in the proof of Theorem~\ref{thm:esbar}).  
It thus suffices to prove  the rule for the multiplication of $s_\Lambda$ by $e_\ell$.
{F}rom Lemma~\ref{lems} and the commutativity between $s_\Lambda$ and $e_\ell$, 
this is equivalent to proving the statement:
\begin{equation}\label{eqn:thmkeyhateqn2}
 e_\ell \, \partial_{\omega_{m}} \pi_{\omega_{m^c}} \hat{K}_{(\Lambda^a)^R,\Lambda^s}(x)= \sum_{\Omega} \partial_{\omega_{m}}\pi_{\omega_{m^c}} \hat{K}_{(\Omega^a)^R,\Omega^s}(x)
\end{equation}
where the sum is over the superpartitions $\Omega$ described in the statement
of the theorem.  Using (\ref{id:KP}) and the commutativity of $e_\ell$ with
the divided-difference operators,  
we can rewrite the left side of (\ref{eqn:thmkeyhateqn2}) as
\begin{equation}
e_\ell\, \partial_{\omega_{m}} \pi_{\omega_{m^c}}  \hat{K}_{(\Lambda^a)^R,\Lambda^s}(x)=\mathcal{P}_{N,[\alpha_1,\dots,\alpha_m]}x^{\Lambda^*}e_\ell
\end{equation}
Recall expansion \eqref{eqxe}
$$x^{\Lambda^*} e_\ell=
\sum_{i_1+i_2+\dots+i_k=\ell}
\boldsymbol{\pi}_{i_1,I_1}  \boldsymbol{\pi}_{i_2,I_2} \cdots \boldsymbol{\pi}_{i_k,I_k} 
x^{\Lambda^*} x^{(i_1,I_1)}x^{(i_2,I_2)}\cdots x^{(i_k,I_k)}$$
which gives
 \begin{equation}\label{eqn:eqnel}
e_\ell \, \partial_{\omega_{m}} \pi_{\omega_{m^c}} \hat{K}_{(\Lambda^a)^R,\Lambda^s}(x)= \sum_{i_1+i_2+\ldots+i_k=\ell} \mathcal P_{N,[\alpha_1,\ldots,\alpha_m]}
\left(\prod_{j=1}^k\boldsymbol{\pi}_{i_j,I_j}x^{(i_j,I_j)} \right) x^{\Lambda^*}
\end{equation}
Observe that the summation corresponds to all possible ways to add a 
vertical $\ell$-strip to $\Lambda^*$. Now, we apply the product of $\boldsymbol{\pi}_{i_j,I_j}$'s to $\mathcal{P}_{N,[\alpha_1,\dots,\alpha_m]}$ one at a time according to the following rules (see Lemma~\ref{lem:propP} in the Appendix): 
\begin{enumerate}
\item
$\mathcal P_{N,[\alpha_1,\ldots,\alpha_m]}\boldsymbol{\pi}_{i_j,I_j}=
\mathcal P_{N,[\alpha_1,\ldots,\alpha_i,\ldots,\alpha_m]}+\mathcal P_{N,[\alpha_1,\ldots,\alpha_i+i_j,\ldots,\alpha_m]}$
if the interval $I_j$ starts with $\alpha_i \in \{\alpha_1,\dots,\alpha_m \}$.\\
\item
$\mathcal P_{N,[\alpha_1,\ldots,\alpha_m]}\boldsymbol{\pi}_{i_j,I_j}= \mathcal P_{N,[\alpha_1,\ldots,\alpha_m]}$
if the interval $I_j$ starts with $\beta \notin \{\alpha_1,\dots,\alpha_m \}$.
\\
\end{enumerate}
Since the $\boldsymbol{\pi}_{i_j,I_j}$'s commute with each other, by the second rule above
we can get rid 
of the factors $\boldsymbol{\pi}_{i_j,I_j}$ such that
$I_j$ start with $\beta \notin \{\alpha_1,\dots,\alpha_m\}$.  Hence
\eqref{eqn:eqnel} becomes
 \begin{equation}\label{eqn:eqnel2}
e_\ell \, \partial_{\omega_{m}} \pi_{\omega_{m^c}} \hat{K}_{(\Lambda^a)^R,\Lambda^s}(x)= \sum_{i_1+i_2+\ldots+i_k=\ell} \mathcal P_{N,[\alpha_1,\ldots,\alpha_m]}
\left( \prod_{\substack{1\le j \le k\\ I_j \text{ starts with }\alpha_i}}\boldsymbol{\pi}_{i_j,I_j}
x^{(i_j,I_j)} \right) x^{\Lambda^*}
\end{equation}
Now we interpret the right side of equation (\ref{eqn:eqnel}) and show that it 
corresponds to the left hand side of \eqref{eqn:thmkeyhateqn2}.
First of all, the sum in (\ref{eqn:eqnel}) is over vertical strips so the first point in the characterization of the superpartitions $\Omega$'s 
in Theorem~\ref{thm:es} is satisfied.
If $I_j$ starts with $\alpha_i$ then
$x^{\Lambda^*}$ is multiplied by $x^{(i_j,I_j)}$, which means that cells
are added in rows $\alpha_i,\alpha_i+1,\dots,\alpha_i+i_j-1$ of $\Lambda^*$. 
The rule for 
$\mathcal{P}_{N,[\alpha_1,\ldots,\alpha_m]}\boldsymbol{\pi}_{i_j,I_j}$ given above
results in this case in two 
terms.  In the first one, the circle remains in row $\alpha_i$ while in the second the circle is moved from row $\alpha_i$ to $\alpha_i+i_j$ (thus sliding down in its column).  Observe that if 
$i_j = |I_j|$, then $\boldsymbol{\pi}_{i_j,I_j}=1$ and the 
second term does not 
occur, meaning that a circle cannot slide down a column past the end of its
block. 
Furthermore, if $\Omega^*_{\alpha_i}=\Omega^*_{\alpha_i-1}$ for some $\alpha_i$ then by Lemma~\ref{lem:propP} of the Appendix,
\begin{equation}
\mathcal{P}_{N,[\alpha_1,\ldots,\alpha_i,\ldots,\alpha_m]}x^{\Omega^*} =\mathcal{P}_{N,[\alpha_1,\ldots,\alpha_i,\ldots,\alpha_m]}\pi_{\alpha_i-1}x^{\Omega^*}=0
\end{equation}
In other words, if there is no addable corner in row $\Omega_{\alpha_i}$ then
the circle in that row cannot be pushed one cell to  its right.  Hence the resulting diagrams have to be superpartitions
obeying the conditions of the theorem and the proof is complete.
\end{proof}

\begin{theorem} \label{thm:ets}
Let $\Lambda$ be a superpartition of fermionic degree $m$.  Then, for $\ell \geq 0$, we have respectively
\begin{equation} \label{etons}
    \bar s_{\Lambda}^* \, \tilde h_\ell = \sum_{\Omega} (-1)^{\#(\Omega,\Lambda)} \bar s_{\Omega}^* \qquad {\rm and} \qquad  s_{\Lambda} \, \tilde e_\ell = \sum_{\Omega} (-1)^{\#(\Omega,\Lambda)} s_{\Omega}
\end{equation}
where the sum is over all superpartitions $\Omega$ of fermionic degree 
$m+1$ 
such that
\begin{itemize}
\item  
$\Omega^{\circledast}/\Lambda^\circledast$ is a horizontal (resp. vertical) $(\ell+1)$-strip
whose rightmost (resp. lowermost) cell (the new circle) belongs to $\Omega^{\circledast}/\Omega^{*}$. 
\item Let $\tilde \Omega$ be $\Omega$ without its new circle. Then
the $i$-th circles, starting from the first row and going down,
of $\tilde \Omega$ and
$\Lambda$ are either in the same row or the same column.  In the
case where the $i$-th circles are in the same row $r$ (resp. column $c$), the circle
in row $r$ (column $c$) of $\Omega$ cannot be located passed row $r-1$ (resp. column $c-1$) of $\Lambda$
(if $r=1$ or $c=1$ the condition does not apply).
\end{itemize}
and where ${\#}(\Omega, \Lambda)$ is again the number of circles in $\Omega$
below the new circle.
\end{theorem}

We illustrate this time the rules by giving the expansion of
$\bar s_{(4,1;5,4)}^* \, \tilde h_2$.

\medskip

{\Tiny

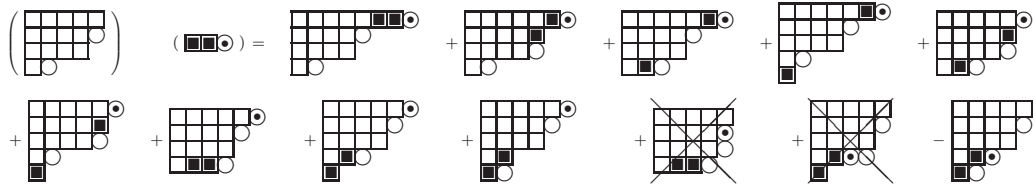
\begin{tabular}{rrrrrrrr} 
 $\left(~{\tableau[scY]{&&&&\\&&&&\bl\tcercle{}\\&&& \\ &\bl\tcercle{}   }}~\right)$&
   $\left( ~{\tableau[scY]{\blacksquare& \blacksquare& \bl\tcercle{$\bullet$} }}~\right)=$&
   
  ${\tableau[scY]{&&&& & \blacksquare & \blacksquare & \bl\tcercle{$\bullet$} \\&&&&\bl\tcercle{}\\&&& \\ &\bl\tcercle{}   }}$&
    $+~ {\tableau[scY]{&&&& & \blacksquare & \bl\tcercle{$\bullet$} \\&&&& \blacksquare \\&&&&  \bl\tcercle{} \\ &\bl\tcercle{}   }}$&
    $+~ {\tableau[scY]{&&&& & \blacksquare & \bl\tcercle{$\bullet$}  \\&&&&  \bl\tcercle{} \\&&& \\ &  \blacksquare &\bl\tcercle{}   }}$&
        $+~ {\tableau[scY]{&&&& & \blacksquare &  \bl\tcercle{$\bullet$}  \\&&&&  \bl\tcercle{} \\&&& \\  &\bl\tcercle{}   \\  \blacksquare}}$&
    $+~ {\tableau[scY]{&&&& &  \bl\tcercle{$\bullet$}  \\&&&& \blacksquare \\&&&&  \bl\tcercle{} \\  & \blacksquare &\bl\tcercle{}   }}$&
\\
\\
  $+~ {\tableau[scY]{&&&& &  \bl\tcercle{$\bullet$}  \\&&&& \blacksquare \\&&&&  \bl\tcercle{} \\   &\bl\tcercle{}  \\  \blacksquare }}$&
  $+~ {\tableau[scY]{&&&& &\bl\tcercle{$\bullet$} \\&&& & \bl\tcercle{} \\&&& \\   &    \blacksquare & \blacksquare & \bl\tcercle{}  }}$&
  $+~ {\tableau[scY]{&&&& &  \bl\tcercle{$\bullet$}  \\&&& &  \bl\tcercle{}  \\&&&\\  & \blacksquare &\bl\tcercle{} \\ \blacksquare  }}$&
 $+~ {\tableau[scY]{&&&& & \bl\tcercle{$\bullet$}  \\&&& &  \bl\tcercle{}  \\&&&\\  & \blacksquare   \\ \blacksquare  &\bl\tcercle{} }}$&
 $+~ \xcancel{\tableau[scY]{&&&&  \\&&&& \bl\tcercle{$\bullet$} \\&&&&  \bl\tcercle{} \\ & \blacksquare & \blacksquare  &\bl\tcercle{}   }}$&
    $+~ \xcancel{\tableau[scY]{&&&&  \\&&& &  \bl\tcercle{}  \\&&&\\   & \blacksquare & \bl\tcercle{$\bullet$} &\bl\tcercle{} \\ \blacksquare  }}$&
   $-~ {\tableau[scY]{&&&&   \\&&& &  \bl\tcercle{}  \\&&&\\  & \blacksquare &  \bl\tcercle{$\bullet$}   \\ \blacksquare  &\bl\tcercle{} }}$&
  
\end{tabular}}

\begin{align}
 \bar s_{(4,1;5,4)}^* \, \tilde h_3=& \,   \bar s_{(7,4,1;4)}^*+ \bar s_{(6,4,1;5)}^*+ \bar s_{(6,4,2;4)}^*+  
  \bar s_{(6,4,1;4,1)}^* +  \bar s_{(5,4,2;5)}^* +  \bar s_{(5,4,1;5,1)}^* \nonumber \\
& \qquad + \bar  s_{(5,4,3;4)}^* + \bar s_{(5,4,2;4,1)}^* +
  \bar s_{(5,4,1;4,2)}^* -  \bar s_{(4,2,1;5,4)}^* \nonumber
 \end{align}

To generate these superpartitions, start with the expansion 
of $\bar s_{(4,1;5,4)}^* \, h_3$ given
previously and replace the new box that is farthest to the right with the new circle. Discard every diagram that has two circles in the same row or column.

\begin{proof}  We first use an argument similar to the one given at the beginning of the proof of Theorem~\ref{pinball2}.
 Applying the involution $\omega$, we obtain
from Proposition~\ref{propdual} that
\begin{equation}
s_\Lambda \, \tilde e_\ell = \sum_\Omega (-1)^{\#(\Omega,\Lambda)} s_{\Omega} \quad \iff \quad 
\bar s_{\Lambda'}^* \, \tilde h_\ell = (-1)^m \sum_\Omega (-1)^{\#(\Omega,\Lambda)} \bar s_{\Omega'}^* 
\end{equation}
The superpartitions $\Omega$ appearing in the expansion of
$s_\Lambda \, \tilde e_\ell$ are thus simply the transposed of those appearing in the expansion of
$\bar s_\Lambda^* \, \tilde h_\ell$.   Observe that $(-1)^{\#(\Omega,\Lambda)+m}=(-1)^{\#(\Omega',\Lambda')}$, which means
that the sign associated to $\Omega$ is the same in both formulas of \eqref{etons}.
It thus suffices to prove  the rule for the multiplication of $s_\Lambda$ by $\tilde e_\ell$.
Using $\tilde e_\ell \, s_{\Lambda}= (-1)^m s_{\Lambda} \, \tilde e_\ell$, this is equivalent to proving that 
\begin{equation}
\tilde e_\ell \, s_{\Lambda} =  \sum_{\Omega} (-1)^{\#(\Omega,\Lambda)'} s_{\Omega}
\end{equation}
where we recall that $\#(\Omega,\Lambda)'$ is equal to the number of circles in 
$\Omega$ {\it above} the
new circle. From Lemmas~\ref{lems} and \ref{lemets}, this corresponds in the bisymmetric world  $\mathbb Q [x_1,\dots,x_N]^{S_{m+1} \times S_{(m+1)^c}}$ to proving
\begin{equation}\label{eqn:etils}
 \partial_{\omega_{m+1}} e_\ell^{(m+1)}\pi_{\omega_{m^c}} \hat K_{(\Lambda^a)^R
,\Lambda^s}(x)=
\sum_{\Omega}(-1)^{\#(\Omega,\Lambda)'}\partial_{\omega_{m+1}}\pi_{\omega_{(m+1)^c}} \hat K_{(\Omega^a)^R,\Omega^s}(x)
\end{equation}
Using  $\pi_{\omega_{m^c}}=\pi_{\omega_{(m+1)^c}}\pi_{[m+1,N-1]}$ and commuting
$e_\ell^{(m+1)}$ with $\pi_{\omega_{(m+1)^c}}$, the left hand side 
of (\ref{eqn:etils}) becomes
\begin{equation}\label{eqn:etils2} 
\partial_{\omega_{m+1}} \pi_{\omega_{(m+1)^c}}e_\ell^{(m+1)}\pi_{[m+1,N-1]} \hat K_{(\Lambda^a)^R,\Lambda^s}(x)
\end{equation}
Lemma~\ref{lem:elm+1} of the Appendix states that
\begin{equation} 
\partial_{\omega_{m+1}}\pi_{\omega_{(m+1)^c}}e_\ell^{(m+1)}\pi_{[m+1,N-1]}= \partial_{\omega_{m+1}}\pi_{\omega_{(m+1)^c}}\sum_{i=1}^{N-m}\hat{\pi}_{[m+1,m+i-1]}e_\ell^{\{1,\dots, m+i-1\}}
\end{equation}
where $e_\ell^{\{1,\dots, m+i-1\}}$ corresponds to $e_\ell(x_1,\dots,x_{m+i-1})$.  We then use
\eqref{eq428} to transform (\ref{eqn:etils2}) into
\begin{equation}
\partial_{\omega_{m+1}} \pi_{\omega_{(m+1)^c}}
\sum_{i=1}^{N-m}\hat{\pi}_{[{m+1},{m+i-1}]}e_\ell^{\{1,\dots, m+i-1\}}
\hat{\pi}_{\omega_m}\hat{\pi}_{[m,{\alpha_m-1}]}\ldots\hat{\pi}_{[1,{\alpha_1-1}]}x^{\Lambda^*}
\end{equation}
As shown in Lemma~\ref{lemfinal} of the Appendix, this expression can be further reduced to
\begin{equation}
\sum_{r\not \in \{\alpha_1,\dots,\alpha_m \}}(-1)^{{\rm pos}(r)}\mathcal P_{N,[\alpha_1,\dots,r,\dots,\alpha_m]}
e_\ell^{\{1,\dots,r-1\}}x^{\Lambda^*}
\end{equation}
where $\text{pos}(r)$ is again the number of entries of $[\alpha_1,\ldots,r,\ldots, \alpha_m]$ that
 are less than $r$.  As in \eqref{eqxe}, we use the identity
$$e_\ell^{\{1,\dots,r-1\}}x^{\Lambda^*}=\sum_{i_1+i_2+\ldots+i_k=\ell}
\left( \prod_{j=1}^k\boldsymbol{\pi}_{i_j,I_j} x^{(i_j,I_j)} \right) x^{\Lambda^*}$$
where the $I_j$'s are this time the blocks of the subpartition $(\Lambda_1^*,\dots,\Lambda_{r-1}^*)$ given that $e_\ell^{\{1,\dots,r-1\}}$ only involves the first $r-1$ variables.  Using a similar argument to the one that led to (\ref{eqn:eqnel2}), we then get
 \begin{eqnarray}\label{eqn:eqnel3}
\sum_{r\not \in \{\alpha_1,\dots,\alpha_m \}}& \displaystyle{\sum_{i_1+i_2+\ldots+i_k=\ell}(-1)^{{\rm pos}(r)}\mathcal P_{N,[\alpha_1,\ldots,r,\ldots,\alpha_m]}
\left( \prod_{\substack{1\le j \le k\\ I_j \text{ starts with }\alpha_i}} \boldsymbol{\pi}_{i_j,I_j}  x^{(i_j,I_j)}\right)}
 x^{\Lambda^*}\nonumber\\
 &=
\displaystyle{ \sum_{r\not \in \{\alpha_1,\dots,\alpha_m \}}\sum_{\mu}(-1)^{{\rm pos}(r)}
\mathcal P_{N,[\alpha_1,\ldots,r,\ldots,\alpha_m]}
\hspace{-1.1em}\prod_{\substack{1\le j \le k\\ I_j \text{ starts with }\alpha_i}}\hspace{-1.1em}\boldsymbol{\pi}_{i_j,I_j}\quad x^\mu}
\end{eqnarray}
where the sum is over all $\mu$'s such that $\mu/\Lambda^*$
is a vertical $\ell$-strip whose cells
are all in rows strictly less than $r$.

The right hand side of the previous equation is very reminiscent of 
(\ref{eqn:eqnel2}), whence its interpretation in the language of Key polynomials will be very similar. The first sum goes through all possibilities
of where the new circle could be. Note that if there is no addable corner in row $r$ of $\mu$ then $\mu_{r-1}=\mu_r$, which implies that  $x^\mu=\pi_{r-1}x^{\mu}$.  Given that by Lemma~\ref{lem:propP} {\it (4)} and {\it (5)} of the Appendix there will still be a circle in row $r$ after  the
operators $\boldsymbol{\pi}_{i_j,I_j}$'s have been applied to
$\mathcal P_{N,[\alpha_1,\ldots,r,\ldots,\alpha_m]}$ (recall that $\boldsymbol{\pi}_{i_j,I_j}$ only acts on the rows above row $r$), the ensuing action
of $\pi_{r-1}$ will annihilate it by Lemma~\ref{lem:propP} {\it (2)} of the Appendix.  
Thus 
the new circle in row $r$ occupies
an addable corner of $\mu$.  We also have that the new cells
are added only in rows strictly less than $r$, that is, strictly above the new
circle. Finally,
 ${\rm pos}(r)$ is exactly ${\#(\Omega,\Lambda)'}$ and so we have 
\begin{equation}
\partial_{\omega_{m+1}} e_\ell^{(m+1)}\pi_{\omega_{m^c}} \hat K_{(\Lambda^a)^R
,\Lambda^s}(x)=
\sum_{\Omega}(-1)^{\#(\Omega,\Lambda)'}\partial_{\omega_{m+1}}\pi_{\omega_{(m+1)^c}} \hat K_{(\Omega^a)^R,\Omega^s}(x)
\end{equation}
where the sum over $\Omega$ runs over the set of superpartitions
stated in the theorem.
\end{proof}

The following corollary shows that a product of fermionic symmetric functions $\tilde e_\ell$ or $\tilde h_\ell$
corresponds to a single
Schur function in superspace (although distinct from that appearing in Corollary~\ref{coroesb}).
\begin{corollary}  \label{coroepieri}
We have
$e_{(\Lambda^a;  \emptyset)} = (-1)^{\binom{m}{2}}s_{(\Lambda^a; \emptyset)'}$, and $h_{(\Lambda^a;  \emptyset)}=\bar s_{(\Lambda^a;  \emptyset)}^*$. 
In particular, $\tilde e_\ell = s_{(0;1^\ell)}$ and $\tilde h_\ell= \bar s_{(\ell;\emptyset)}^*$.

\end{corollary}
\begin{proof}
Let $\lambda=\Lambda^a$ and $\Lambda=(\lambda;\emptyset)$.  We first show that $h_\Lambda=\bar s_\Lambda^*$.
We proceed by induction on $m$.  If $m=1$, the result holds since
\begin{equation}
\tilde h_{\lambda_1}=  \bar s_{(\emptyset;\emptyset)}^* \,  \tilde h_{\lambda_1}  =\bar  s_{(\lambda_1;\emptyset)}^* 
\end{equation}
from Theorem~\ref{thm:ets} (starting from the empty superpartition, the only superpartition $\Omega$ satisfying the conditions of the theorem is obviously the one consisting of a row with $\lambda_1$ squares followed by a circle).  

Now, let $\hat \lambda=(\lambda_2,\dots,\lambda_{m})$ 
be the partition $\lambda$ without its
first part.  By induction, we simply need to show that 
\begin{equation}
 \tilde h_{\lambda_1} \,  \bar {s}_{(\hat \lambda;\emptyset)}^*  = \bar {s}_{\Lambda}^* 
\end{equation}
By Theorem \ref{thm:ets} and the fact that $  \bar {s}_{(\hat \lambda;\emptyset)}^*= (-1)^{\binom{m}{2}} \bar {s}_{(\hat \lambda;\emptyset)}^* \,\tilde h_{\lambda_1} $, this corresponds to showing that in 
$$ \tilde h_{\lambda_1} \, \bar {s}_{(\hat \lambda; \emptyset)}^* =\sum_{\Omega}  (-1)^{\#(\Omega,\hat{\Lambda})'} \bar s_{\Omega}^*$$
the sum only contains the term ${s}_{\Lambda}^*$ (with positive sign), where we recall that $\#(\Omega,\hat{\Lambda})'$ is the number of circles above the new circle.
Observe that $\Lambda$ is precisely the superpartition obtained by 
placing the new circle in column $\lambda_1+1$ and having a new square in every column to the left of the new circle.

The new circle cannot be in a column $c<\lambda_1+1$ because by the first conditions of Theorem~\ref{thm:ets}, all of the new squares must 
be to the left of the new circle. So the new circle must be in the first row. This means that the circle in the first row must be moved to the second row.  By the conditions of Theorem \ref{thm:ets}, this implies that
the second row must be completed. Now the circle in the second row must be moved to the third row and 
the third row must thus also be completed.
Since there is a circle in every row, repeating the argument again and again we see that all rows must be completed
 (including row $m$). Therefore, the new circle must be in column $\lambda_1+1$ (otherwise it would be impossible 
to complete all the rows) and there is a 
new square in every column to the left of it.  As previously commented, this corresponds to the superpartition $\Lambda$.
Hence,
\begin{equation}
\tilde h_{\lambda_1} \bar s_{(\hat{\lambda}; \emptyset)}^* = (-1)^{\#(\Lambda,\hat{\Lambda})} \bar s_{\Lambda}^*
\end{equation}
and the result holds since $\#(\Lambda,\hat{\Lambda})'=0$ (the new circle is in the first row).

Applying the homomorphism $\omega$ on both sides of  $h_{\Lambda} = \bar s_{\Lambda}^*$ we immediately get, from
 Proposition~\ref{propdual} and \eqref{hedual}, 
that $e_{\Lambda} = (-1)^{\binom{m}{2}} s_{\Lambda'}$.

\end{proof}

\section{Kostka coefficients in superspace}

Define the Kostka coefficients in superspace $K_{\Omega \Lambda}$ and $\bar K_{\Omega \Lambda}$  to be respectively such that
\begin{equation}\label{h2snew}
h_\Lambda = \sum_{\Omega} K_{\Omega \Lambda} \, \bar s_{\Omega}^*
\qquad \text{and} \qquad h_\Lambda = \sum_{\Omega} \bar K_{\Omega \Lambda} \, s_{\Omega}^*
\end{equation}
As expected, the Kostka coefficients in superspace give the monomial expansion of the Schur functions in superspace.
\begin{proposition} We have 
\begin{equation} \label{s2m}
s_\Lambda= \sum_{\Omega \leq \Lambda} \bar K_{\Lambda \Omega} \, m_{\Omega} \qquad {\rm and} \qquad
\bar s_\Lambda= \sum_{\Omega \leq \Lambda} K_{\Lambda \Omega} \, m_{\Omega} 
\end{equation}
\end{proposition}
\begin{proof}
The triangularity in both equations follows from the fact that the Schur functions in superspace
are special cases of Macdonald 
polynomials in superspace, which are triangular by construction.

Suppose that $s_\Lambda= \sum_{\Omega} A_{\Lambda \Omega} \, m_{\Omega}$.
On the one hand, by duality of the bases $s_\Lambda$ and $s_\Lambda^*$, we have 
 $\langle \! \langle h_{\Lambda}, s_{\Gamma} \rangle \! \rangle = \bar K_{\Gamma \Lambda}$ by the second formula of \eqref{h2snew}.  
On the other hand, by the duality \eqref{hmdual} between the bases $m_\Lambda$ and $h_\Lambda$, we have 
\begin{equation}
\langle \! \langle h_{\Lambda}, s_{\Gamma} \rangle \! \rangle =
 \langle \! \langle h_{\Lambda}, \sum_{\Omega} A_{\Gamma \Omega} \, m_{\Omega} \rangle \! \rangle=  A_{\Gamma \Lambda}
\end{equation}  
Therefore $ A_{\Lambda  \Omega}=\bar K_{\Lambda \Omega}$ and the first equality holds.

Similarly, suppose that $\bar s_\Lambda= \sum_{\Omega} B_{\Lambda \Omega} \, m_{\Omega}$.
As before, by duality of the bases $\bar s_\Lambda$ and $\bar s_\Lambda^*$,
we get that $\langle \! \langle h_{\Lambda}, \bar s_{\Gamma} 
\rangle \! \rangle = K_{\Gamma \Lambda}$ by the first formula of \eqref{h2snew}, while 
by duality of the bases $m_\Lambda$ and $h_\Lambda$ we get 
and
$\langle \! \langle h_{\Lambda}, \bar s_{\Gamma} \rangle \! \rangle = 
B_{\Gamma \Lambda}$.
Hence $ B_{\Lambda  \Omega}=K_{\Lambda \Omega}$ and the second equality also holds.
\end{proof}
The  Kostka coefficients  in superspace turn out to be nonnegative integers,
 which is relatively  surprising given that 
signs show up in the Pieri rules.   
\begin{proposition} The Kostka coefficients in superspace
$K_{\Omega \Lambda}$ and $\bar K_{\Omega \Lambda}$ are nonnegative integers.  
\end{proposition}
\begin{proof}
We first prove the nonnegativity of $K_{\Omega \Lambda}$.  From Corollary~\ref{coroepieri}, 
we have that if $\Lambda=(\Lambda^a;\emptyset)$, then
$h_\Lambda=\bar s_{\Lambda}^*$ and the proposition holds in that case.  
For $\Lambda$ generic, we thus have
$h_\Lambda=h_{(\Lambda^a;\emptyset)} \, h_{(\emptyset;\Lambda^s)}=\bar s_{(\Lambda^a;\emptyset)}^* h_{(\emptyset;\Lambda^s)}$.
The expansion of $h_\Lambda$ in terms of Schur functions in superspace $s_\Omega$ of \eqref{h2snew}
is thus obtained by using repeatedly,
starting from  $\bar s_{(\Lambda^a;\emptyset)}^*$, the Pieri rule for the multiplication by $h_\ell$ of Theorem~\ref{thm:es}.
But since that Pieri rule does not involve signs, the positivity of $K_{\Omega  \Lambda}$ is immediate.

The positivity of $\bar K_{\Omega \Lambda}$ follows similarly using Corollary~\ref{coroesb} and 
the Pieri rule for the multiplication by $h_\ell$ of Theorem~\ref{thm:esbar}.
\end{proof}
A combinatorial interpretation in terms of tableaux for the coefficients
 $\bar K_{\Lambda \Omega}$ and $K_{\Lambda \Omega}$  
will be given in Corollary~\ref{coroschur} and Corollary~\ref{coroschurb} respectively.
As will be discussed in Remark~\ref{remarktableaux}, a 
different  combinatorial interpretation for $\bar K_{\Lambda \Omega}$ and $K_{\Lambda \Omega}$  
is conjectured in \cite{BM}.

\section{Extra Pieri rules}
In order to establish the duality that will appear in 
Section~\ref{SecDual}, we need to derive another Pieri rule 
whose terms are exactly those that appear in the multiplication of $h_\ell$ by $\bar s_\Lambda^*$
obtained in Theorem~\ref{thm:es}.

\begin{theorem}\label{thm:hs}
Let $\Lambda$ be a superpartition of fermionic degree $m$.  For $\ell \geq 1$,
we have
\begin{equation}
s_{\Lambda} \, h_\ell= \sum_{\Omega} s_{\Omega}
\end{equation}
where the sum is over all superpartitions $\Omega$ of fermionic degree $m$ 
such that
\begin{itemize}
\item  $\Omega^*/\Lambda^*$ is a horizontal $\ell$-strip 
\item 
the $i$-th circles, starting from the first row and going down, 
of $\Omega$ and
$\Lambda$ are either in the same row or the same column.  In the
case where the $i$-th circles are in the same row $r$, 
the circle in row $r$ of $\Omega$ cannot be located passed row $r-1$ 
of $\Lambda$
(if $r=1$ the condition does not apply).
\end{itemize}
\end{theorem}
\begin{proof}
Using Lemma~\ref{lems} and the commutativity of $h_\ell$ and $s_\Lambda$, it is equivalent to prove that
\begin{equation}\label{eqn:thmkeyhateqn}
h_\ell \, \partial_{\omega_{m}} \pi_{\omega_{m^c}}  \hat{K}_{(\Lambda^a)^R,\Lambda^s}(x)= \sum_{\Omega} \partial_{\omega_{m}}\pi_{\omega_{m^c}} \hat{K}_{(\Omega^a)^R,\Omega^s}(x)
\end{equation}
where the sum is over all superpartitions $\Omega$ described in the statement of the theorem.  With the help of
(\ref{id:KP}) and the fact that $h_\ell$ commutes with all the operators, 
we can rewrite the left hand side of (\ref{eqn:thmkeyhateqn}) as
\begin{equation}
h_\ell \, \partial_{\omega_{m}} \pi_{\omega_{m^c}} \hat{K}_{(\Lambda^a)^R,\Lambda^s}(x)=\mathcal{P}_{N,[\alpha_1,\dots,\alpha_m]}x^{\Lambda^*}h_\ell.
\end{equation}
As shown in Lemma~\ref{lemKh} of the Appendix, we have
 $$x^{\Lambda^*} h_{\ell}=
 \sum_{\substack{\mu/\Lambda^* \\ \text{horizontal $\ell$-strip}}}\left(\prod^{(\rm {down})}_{\substack{i\\ \mu_{i+1}=\Lambda^*_i}}\pi_i \right)x^\mu$$
where the superscript ``down'' means that the $\pi_i$'s in the product are ordered from the largest to 
the smallest index $i$.   Hence,
\begin{equation}\label{eqn:KhatP}
 h_\ell \, \partial_{\omega_{m}} \pi_{\omega_{m^c}} \hat{K}_{(\Lambda^a)^R,\Lambda^s}(x)=
 \sum_{\substack{\mu/\Lambda^* \\ \text{horizontal $\ell$-strip}}}\mathcal{P}_{N,[\alpha_1,\dots,\alpha_m]} \left( \prod^{(\rm{down})}_{\substack{i\\ \mu_{i+1}=\Lambda^*_i}}\pi_i \right)  x^\mu
\end{equation}
We say that row $i+1$ of $\Lambda^*$ has been completed whenever, as in the product above, $\mu_{i+1}=\Lambda_i^*$. Note that
this also include the case where $\mu_{i+1}=\Lambda_{i+1}^*=\Lambda_i^*$ where no real completion occurred.

We need to interpret the right hand side of equation (\ref{eqn:KhatP}) and show that it sums over superpartitions satisfying the two conditions given in the statement of  Theorem \ref{thm:hs}.  The fist condition is satisfied because the sum in (\ref{eqn:KhatP}) is over horizontal $\ell$-strips.  To show that the second condition is satisfied is more complex.  In order to do so, we
need to understand the behavior or the $\pi_i$'s associated to completed rows. This relies essentially 
on the following rules (see Lemma~\ref{lem:propP} of the Appendix):
\begin{enumerate}
\item
$\mathcal P_{N,[\alpha_1,\ldots,\alpha_m]}{\pi}_{\alpha_i}=
\mathcal P_{N,[\alpha_1,\ldots,\alpha_m]}+\mathcal{P}_{N,[\alpha_1,\ldots,\alpha_i+1,\ldots,\alpha_m]}$
\item
$\mathcal P_{N,[\alpha_1,\ldots,\alpha_m]}{\pi}_{\alpha_i-1}=0$
\item
$ \mathcal P_{N,[\alpha_1,\ldots,\alpha_m]}{\pi}_{\beta}= 
\mathcal P_{N,[\alpha_1,\ldots,\alpha_m]}$
if $\beta \notin \{\alpha_i,\alpha_i-1\}$ for $i=1,\ldots,m$.
\end{enumerate}

Since the product $\prod^{(\rm{down})}_{\substack{i : \mu_{i+1}=\Lambda^*_i}} {\pi}_i$ is decreasing, we start from row $\rho=N$ 
with the greatest value and follow the steps.
 \begin{enumerate}
\item
If row $\rho$ has not been completed then 
 $\pi_{\rho-1}$ does not appear in the product $\prod^{(\rm{down})}_{\substack{i : \mu_{i+1}=\Lambda^*_i}}\pi_i$
and so $\mathcal P_{N,[\alpha_1,\ldots,\alpha_m]}$ is unaffected and we move on to row $\rho-1$.
\item
If row $\rho$ has been completed but there is no circle in both rows
$\rho$ and $\rho-1$ then $\pi_{\rho-1}$ belongs to the product $\prod^{(\rm{down})}_{\substack{i : \mu_{i+1}=\Lambda^*_i}}\pi_i$ but 
$\rho-1\notin \{\alpha_i,\alpha_i-1\}$ for $i=1,\ldots,m$.   Hence $\pi_{\rho-1}$ acts as the identity on
 $\mathcal P_{N,[\alpha_1,\ldots,\alpha_m]}$ and we move on to row $\rho-1$.
\item
If row $\rho$ has been completed and there is a circle in row $\rho$ then  $\rho=\alpha_i$ for some $i$ and
$\pi_{\alpha_i-1}$ is in the product $\prod^{(\rm{down})}_{\substack{i : \mu_{i+1}=\Lambda^*_i}}\pi_i$.
Therefore in this case $\mathcal P_{N,[\alpha_1,\ldots,\alpha_m]}\pi_{\alpha_i-1}=0$ and we stop the process.
\item
If row $\rho$ has been completed and there is no circle in row $\rho$ while there is
a circle in row $\rho-1$ then $\rho-1=\alpha_i$ for some $i$ and
$\pi_{\alpha_i}$ is in the product $\prod^{(\rm{down})}_{\substack{i : \mu_{i+1}=\Lambda^*_i}}\pi_i$.  There are two possible cases:
\begin{itemize}  
\item If $\mu_{\alpha_i}=\mu_{\alpha_i+1}$, then $\mu_{\alpha_i}=\Lambda_{\alpha_i}^*<
\Lambda_{\alpha_i-1}^*\leq \mu_{\alpha_i-1}$ since there is a circle in row $\alpha_i$.  Thus row $\alpha_i$ is not a completed row, which means that $\pi_{\alpha_i-1}$ is not in the product $\prod^{(\rm{down})}_{\substack{i : \mu_{i+1}=\Lambda^*_i}}\pi_i$.
Therefore, $\pi_{\alpha_i}$ commutes with all the $\pi_i$'s to its right and then
acts as the identity on $x^\mu$ 
(recall that $\mu_{\alpha_i}=\mu_{\alpha_i+1}$).   We can thus consider that
$\mathcal P_{N,[\alpha_1,\ldots,\alpha_m]}\pi_{\alpha_i}=P_{N,[\alpha_1,\ldots,\alpha_m]}$ and we
move on to row $\rho-1$. 
\item Otherwise, $\mathcal P_{N,[\alpha_1,\ldots,\alpha_m]}\pi_{\alpha_i}=\mathcal P_{N,[\alpha_1,\ldots,\alpha_m]}+\mathcal P_{N,[\alpha_1,\ldots,\alpha_i+1,\ldots,\alpha_m]}$
and we split the process into the one for $\mathcal P_{N,[\alpha_1,\ldots,\alpha_m]}$ and that for which $\mathcal P_{N,[\alpha_1,\ldots,\alpha_i,\ldots, \alpha_m]}$ is replaced by $\mathcal P_{N,[\alpha_1,\ldots,\alpha_i+1,\ldots,\alpha_m]}$ (the two cases correspond to leaving the circle in its
row $\alpha_i$ and sliding it down to row $\alpha_{i}+1$ respectively).  We then move on to row $\rho-1$. 
\end{itemize}
\end{enumerate}

We thus have, by construction, that
\begin{equation}
 \mathcal{P}_{N,[\alpha_1,\dots,\alpha_m]} \left( \prod^{(\rm{down})}_{\substack{i\\ \mu_{i+1}=\Lambda^*_i}}\pi_i \right)  x^\mu = \sum_{\alpha_1',\dots,\alpha_m'}
\mathcal{P}_{N,[\alpha_1',\dots,\alpha_m']}   x^\mu =\sum_\Omega
\partial_{\omega_{m}} \pi_{\omega_{m^c}} \hat{K}_{(\Omega^a)^R,\Omega^s}(x)
\end{equation}
where the sum is over all superpartitions $\Omega$ (obtained from the
partition $\mu=\Omega^*$ by adding circles in rows $\alpha_1',\dots,\alpha_m'$) such that
\begin{enumerate}[1.]
\item
If a square was added to a row with a circle then that circle may move to the right if there is room
or down if there is room (that is, if the row below has been completed)
\item
A square cannot push a circle to the right past the length of the original length of the row above.
\end{enumerate}
For a given $\Omega^*=\mu$, these are precisely the superpartitions $\Omega$ described in the statement of the
theorem.  Since the sum in  \eqref{eqn:KhatP} is over all $\mu$'s such that $\mu/\Lambda^*$ is a horizontal $\ell$-strip,
the other condition in the statement of the theorem is satisfied and the proof is complete.
\end{proof}
The Pieri rule corresponding to the multiplication of $s_\Lambda$ by $\tilde p_\ell$ will be established 
in full generality in Corollary~\ref{coronewpieri}.  For our purposes, it is sufficient at this point to prove the very special
case when the superpartition $\Lambda=(\Lambda^a;\Lambda^s)$ is such that $\Lambda^s$ is the empty partition.
\begin{lemma} Let $\lambda=(\lambda_1,\dots,\lambda_m,\lambda_{m+1})$ be a partition and let
$\hat \lambda=(\lambda_1,\dots,\lambda_m)$  be $\lambda$ without its
last part.  We have
\begin{equation}
 s_{(\hat \lambda; \emptyset)} \, \tilde p_{\lambda_{m+1}}
= s_{(\lambda; \emptyset)} 
\end{equation}
\end{lemma}

\begin{proof}
We will instead show that $ \tilde p_{\lambda_{m+1}} \, s_{(\hat \lambda; \emptyset)} 
= (-1)^m s_{(\lambda; \emptyset)}$.
Let $\ell=\lambda_{m+1}$ and note that $\hat \lambda$ is a partition of length $m$.  
Using Lemma~\ref{lems} and Lemma~\ref{lempts}, we need to show that
\begin{equation}\label{eqn:ptils}
\partial_{\omega_{m+1}} x_{m+1}^\ell  \pi_{\omega_{m^c}} \hat K_{(\hat \lambda)^R, 0^{N-m}}(x)=(-1)^m\partial_{\omega_{m+1}} \pi_{\omega_{(m+1)^c}}\hat K_{\lambda^R,0^{N-m-1}}(x)
\end{equation}
We start with the left hand side and use $\pi_{\omega_{m^c}}=\pi_{\omega_{(m+1)^c}}\pi_{[m+1,N-1]}$ to get
\begin{equation}
\partial_{\omega_{m+1}}  x_{m+1}^\ell \pi_{\omega_{m^c}} \hat K_{(\hat \lambda^R; 0^{N-m})}(x)
=\partial_{\omega_{m+1}}  x_{m+1}^\ell\pi_{\omega_{(m+1)^c}}\pi_{[m+1,N-1]} \hat K_{(\hat \lambda^R; 0^{N-m})}(x)
\end{equation}
By definition, 
$\hat K_{\hat \lambda^R,0^{N-m}}(x)=\hat\pi_{\omega_m}x^{{\hat \lambda}}$, which means that
$\pi_{m+1},\dots,\pi_{N-1}$ act on the right as the identity 
(since rows $m+1,\dots,N$ of $\hat \lambda$ are all zero).  After commuting $x_{m+1}^\ell$ with 
$\pi_{\omega_{(m+1)^c}}  \hat\pi_{\omega_m}$ and combining 
$x_{m+1}^\ell x^{{\hat \lambda}}= x^{{\lambda}}$,  we then get
\begin{equation}
\partial_{\omega_{m+1}} x_{m+1}^\ell  \pi_{\omega_{m^c}} \hat K_{(\hat \lambda)^R, 0^{N-m}}(x)=
 \partial_{\omega_{m+1}} \pi_{\omega_{(m+1)^c}} \hat\pi_{\omega_m} x^{{ \lambda}}
\end{equation}
Owing to the fact that $\partial_{\omega_{m+1}}=(-1)^m\partial_{\omega_{m+1}}\hat{\pi}_{[1,m]}$,
we have
\begin{eqnarray}
\partial_{\omega_{m+1}} \pi_{\omega_{(m+1)^c}} \hat\pi_{\omega_m} x^{{ \lambda}}
&=&
(-1)^m\partial_{\omega_{m+1}}\hat{\pi}_{[1,m]} \pi_{\omega_{(m+1)^c}} \hat\pi_{\omega_m}  x^{{\lambda}}\nonumber\\
 &=&  (-1)^m\partial_{\omega_{m+1}} \pi_{\omega_{(m+1)^c}} \hat{\pi}_{[1,m]}\hat\pi_{\omega_m}  x^{{ \lambda}}\nonumber\\
  &=& (-1)^m\partial_{\omega_{m+1}} \pi_{\omega_{(m+1)^c}} \hat\pi_{\omega_{m+1}} x^{{\lambda}} \nonumber \\
&=& (-1)^m\partial_{\omega_{m+1}} \pi_{\omega_{(m+1)^c}} \hat K_{\lambda^R,0^{N-m-1}}(x)
\end{eqnarray}
where we used the fact that $\hat\pi_{\omega_{m+1}} x^{{\lambda}}= \hat K_{\lambda^R,0^{N-m-1}}(x)$ since
$\lambda$ is a partition of length $m+1$. Therefore \eqref{eqn:ptils} holds from the previous two equations and the claim follows.
\end{proof}
The following corollary is immediate.
\begin{corollary} \label{corop} If $\Lambda=(\Lambda^a;\emptyset)$ then
$
p_{(\Lambda^a;\emptyset)}=\tilde p_{\Lambda^a_{1}}\cdots \tilde p_{\Lambda^a_{m}}= s_{(\Lambda^a; \emptyset)}=s_{\Lambda}$. 
In particular, $\tilde p_\ell = s_{(\ell;\emptyset)}$.
\end{corollary}

\section{Duality} \label{SecDual}
We will now see that there is a natural duality that relates $s_\Lambda$ and $s_{\Lambda'}$.  Unexpectedly, for reasons we will see later in this section, no such simple duality exists in the case of $\bar s_\Lambda$.

Applying $\omega$ on the first formula of \eqref{h2snew}, we obtain from \eqref{hedual} and Proposition~\ref{propdual} that
\begin{equation}\label{e2snew}
e_\Lambda = (-1)^{\binom{m}{2}} \sum_{\Omega} K_{\Omega \Lambda} \, s_{\Omega'}
\end{equation}
Now, let $H_\Lambda=p_{(\Lambda^a;\emptyset)} h_{(\emptyset;\Lambda^s)}$, that is, $H_\Lambda$ is $h_\Lambda$ with the superspace generators $\tilde h_r$ replaced by $\tilde p_r$.  
\begin{proposition} We have
\begin{equation}\label{H2s}
H_\Lambda = \sum_{\Omega} K_{\Omega \Lambda} \, s_{\Omega}
\end{equation}
where we stress that the coefficients $K_{\Omega \Lambda}$ are exactly those that appear in \eqref{e2snew}. 
\end{proposition}
\begin{proof}
 From Corollary~\ref{coroepieri} and \ref{corop}, we have that if $\Lambda=(\Lambda^a;\emptyset)$, then
$e_\Lambda=(-1)^{\binom{m}{2}} s_{\Lambda'}$ and $H_\Lambda= p_{(\Lambda^a;\emptyset)}
= s_{\Lambda}$.  Which means that the proposition
holds in that case with $K_{\Omega \Lambda}=\delta_{\Omega \Lambda}$.  For $\Lambda$ generic, we thus have
$e_\Lambda=e_{(\Lambda^a;\emptyset)} \, e_{(\emptyset;\Lambda^s)}=(-1)^{\binom{m}{2}} s_{(\Lambda^a;\emptyset)'} e_{(\emptyset;\Lambda^s)}$,
and similarly, $H_\Lambda= p_{(\Lambda^a;\emptyset)} \, h_{(\emptyset;\Lambda^s)}=s_{(\Lambda^a;\emptyset)} h_{(\emptyset;\Lambda^s)}$.
The expansion of $e_\Lambda$ in terms of Schur functions in superspace $s_\Omega$ is thus obtained by using repeatedly,
starting from  $(-1)^{\binom{m}{2}} s_{(\Lambda^a;\emptyset)'}$, the Pieri rule for the multiplication by $e_\ell$ of Theorem~\ref{thm:es}. 
Similarly, the expansion of $H_\Lambda$ in terms of Schur functions in superspace $s_\Omega$ is obtained by using repeatedly,
starting from  $s_{(\Lambda^a;\emptyset)}$, the Pieri rule for the multiplication by $h_\ell$ of Theorem~\ref{thm:hs}. 
But since those Pieri rules are transposed of each other and do not involve any sign, 
the proposition follows immediately given that we start on the transposed superpartitions $(\Lambda^a;\emptyset)'$ and $(\Lambda^a;\emptyset)$ and with signs that differ by a factor of $(-1)^{\binom{m}{2}}$.
\end{proof}

The previous proposition suggests a natural duality between the $H_\Lambda$ and $e_\Lambda$ bases.  In effect,  let $\varphi$ be the homomorphism defined by 
\begin{equation} \label{defphi}
\varphi(\tilde p_r) = \tilde e_r  \quad {\rm and}  \quad \varphi(h_r)= e_r
\end{equation}
that is, such that $\varphi(H_\Lambda)=e_\Lambda$.  By  \eqref{e2snew} and \eqref{H2s}, 
the following result is essentially immediate:
\begin{corollary}  \label{corosdual} We have
\begin{equation}
\varphi(s_\Lambda)=(-1)^{\binom{m}{2}}s_{\Lambda'}
\end{equation}
As a consequence, the homomorphism $\varphi$ is an involution, that is, $\varphi \circ \varphi$ is the identity.
\end{corollary}
\begin{proof} Applying the homomorphism $\varphi$ on both sides of \eqref{H2s}, we have from  \eqref{e2snew} 
\begin{equation}
e_\Lambda = \sum_\Omega K_{\Omega \Lambda} \,  \varphi(s_\Omega) = (-1)^{\binom{m}{2}} 
\sum_\Omega K_{\Omega \Lambda} \, s_{\Omega'}  
\end{equation}
Therefore $\varphi(s_\Lambda)=(-1)^{\binom{m}{2}}s_{\Lambda'}$ since the
matrix $\{ K_{\Omega \Lambda}\}_{\Omega, \Lambda}$ is invertible (it is, up to a sign, the change of basis matrix between the bases $e_\Lambda$ and $s_{\Lambda'}$). The involutivity of $\varphi$ is then immediate.
\end{proof}

Using Corollary~\ref{corosdual}, the Pieri rules for the multiplication of $s_\Lambda$ by 
$\tilde p_\ell$ and $h_\ell$ are seen, from Theorems \ref{thm:es} and 
\ref{thm:ets}, to be identical to those for the multiplication of $\bar s_\Lambda^*$ by 
$\tilde h_\ell$ and $h_\ell$ respectively.  The first formula in the corollary is the content of Theorem~\ref{thm:hs}.
But the point here is to understand how it follows from the duality (which unfortunately we were not able to prove without
first proving Theorem~\ref{thm:hs}).
We should note that these Pieri rules were conjectured to hold in \cite{BM}.
\begin{corollary}  \label{coronewpieri} We have, for $k \geq 1$ and $\ell \geq 0$,
\begin{equation}
  s_\Lambda \, h_k  = \sum_{\Omega} s_\Omega \qquad {\rm and} \qquad
 s_\Lambda\,  \tilde p_\ell = \sum_{\Omega} (-1)^{\#(\Omega,\Lambda)} s_\Omega
\end{equation}
where $\#(\Omega,\Lambda)$ is as usual the number of circles of $\Omega$ below the new circle and where 
the sums run over the superpartitions $\Omega$ obeying the conditions of
Theorem~\ref{thm:es} (in the $\bar s_\Lambda^* h_\ell$ case) and 
Theorem~\ref{thm:ets} (in the $\bar s_\Lambda^* \, \tilde h_\ell$ case) respectively.
\end{corollary}
\begin{proof} 
We prove the second formula, as it involves a non-trivial sign.  
The first formula follows from the same argument (without the sign issue)
using Corollary~\ref{corosdual} and Theorem~\ref{thm:es}.

Applying $\varphi$ on both sides of the second equation in \eqref{etons}, we obtain from Corollary~\ref{corosdual}
\begin{equation}
 s_{\Lambda'}\, \tilde p_\ell  = \sum_{\Omega} (-1)^{\#(\Omega,\Lambda)+m} s_{\Omega'}
\end{equation}
where the extra $m$ in the sign comes from the difference between ${\binom{m}{2}}$ and ${\binom{m+1}{2}}$, and where the sum is over the $\Omega$'s described in
Theorem~\ref{thm:ets} (in the $s_\Lambda \, \tilde e_\ell$ case).
 Hence
\begin{equation}
s_{\Lambda} \, \tilde p_\ell  = \sum_{\Omega'} (-1)^{\#(\Omega',\Lambda')+m} s_{\Omega}
=  \sum_{\Omega'} (-1)^{\#(\Omega,\Lambda)} s_{\Omega}
\end{equation}
since ${\#(\Omega',\Lambda')}+{\#(\Omega,\Lambda)}=m$.  The result then follows
by transposing the conditions on $\Omega$ (which means considering 
$\bar s_\Lambda^* \, \tilde h_\ell$ instead of $s_\Lambda \, \tilde e_\ell$).
\end{proof}

It is immediate from Proposition~\ref{propdual}  that the linear map $\omega\circ \varphi^+ \circ \omega$ sends
$\bar s_\Lambda$ to $\bar s_{\Lambda'}$ (up to a sign), where
$\varphi^+$ is the adjoint of $\varphi$ with respect to the scalar product \eqref{scal1}.  But $\varphi^+$ turns out not
to be a homomorphism since otherwise  the map $\omega\circ \varphi^+ \circ \omega$ would also be a homomorphism, contradicting the fact
that
the Littlewood-Richardson coefficients 
associated to the products of Schur functions in superspace $\bar s_\Lambda$
do not have a symmetry under conjugation (see Remark~\ref{remark}). 
As such the duality between $\bar s_\Lambda$ and $\bar s_{\Lambda'}$ is less natural
(for instance, it does not lead to any analog of Corollary~\ref{coronewpieri}).

\begin{remark} \label{remarkpieri}
In Conjecture 4.1 of \cite{BM}, Pieri rules for the multiplication of $s_{\Lambda}$ by 
$s_{(\emptyset;r)}$, $s_{(r;\emptyset)}$, $s_{(\emptyset;1^r)}$ and $s_{(0;1^r)}$ are stated.  Given that $s_{(\emptyset;r)}=s_r=h_r$, 
$s_{(r;\emptyset)}=\tilde p_r$ (see Corollary~\ref{corop}),  $s_{(\emptyset;1^r)}=e_r$ and  $s_{(0;1^r)}= m_{(0;1^r)}=\tilde e_r$ (by the triangularity in \eqref{s2m}), these Pieri rules are proven in Corollary~\ref{coronewpieri} 
($h_r$ and $\tilde p_r$), Theorem~\ref{thm:es} ($e_r$) and Theorem~\ref{thm:ets} ($\tilde e_r$).
\end{remark}

\section{Tableaux}
In this section, we show that the Schur functions in superspace are generating series of certain types of tableaux.  

We will refer to $\{\bar 0, \bar 1, \bar 2,\bar 3,\dots \}$ as the set of {\it fermionic} nonnegative integers.  In this spirit,
we will also refer to the set of nonnegative
integers  $\{0, 1, 2,3,\dots \}$ as the set of {\it bosonic} nonnegative integers.  For  
$\alpha \in  \mathbb \{0,\bar 0,1,\bar 1,2,\bar 2, \dots \}$, we will say that ${\rm type}(\alpha)$ is bosonic or fermionic
depending on whether the corresponding integer is fermionic or bosonic.  Finally, define
\begin{equation}
|\alpha|= 
\left\{
\begin{array}{ll}
a & \text{if~} \alpha=\bar a \text{~is fermionic} \\
a & \text{if~} \alpha=  a \text{~is bosonic}
\end{array} \right.
\end{equation}

\subsection{$s$-tableaux}  
By \eqref{h2snew}, 
the tableaux needed to represent the Schur function
in superspace $s_\Lambda$ are those stemming from the Pieri rules associated to the multiplication 
of $s_\Lambda^*$ by $h_r$ or $\tilde h_r$ given in Theorem~\ref{thm:esbar} and \ref{pinball2}. 
We say that the sequence $\Omega=\Lambda_{(0)},\Lambda_{(1)},\dots, \Lambda_{(n)}=\Lambda$ is an $s$-tableau 
of shape $\Lambda/\Omega$ and weight $(\alpha_1,\dots,\alpha_n)$, where 
$\alpha_i\in \mathbb \{0,\bar 0, 1,\bar 1,2,\bar 2,\dots \}$, if $\Omega=\Lambda_{(i)}$ and $\Lambda=\Lambda_{(i-1)}$ 
obey the conditions of Theorem~\ref{thm:esbar} with $\ell=\alpha_i$ whenever $\alpha_i$ is bosonic or
the conditions of Theorem~\ref{pinball2} with $\ell=|\alpha_i|$ whenever $\alpha_i$ is fermionic.  
An $s$-tableau can be represented by a diagram constructed recursively in the following way: 
\begin{enumerate}
\item  the cells of $\Lambda_{(i)}^*/\Lambda_{(i-1)}^*$, which form a horizontal strip,  are filled with the letter $i$.  In 
the fermionic case, the new circle is also filled with a letter $i$. 
\item the circles of $\Lambda_{(i-1)}$ that are moved a row below keep their fillings.
\end{enumerate}
The sign of an $s$-tableau $T$, which corresponds to the product of the signs appearing in the fermionic horizontal strips, 
can be extracted quite efficiently from an $s$-tableau.  Read the fillings of the circles from top to bottom
to obtain a word (without repetition):
the sign of the tableau $T$ is  then equal to $(-1)^{{\rm inv} (T)}$, 
where ${{\rm inv} (T)}$ is the number of inversions of the word.

Given a diagram of an $s$-tableau, we define the path of a given circle (filled let's say with letter $i$) in the following way.  Let $c$ be the leftmost column that does not contain a square (a cell of $\Omega^*$) 
filled with an $i$.  The path starts in the position of the smallest entry 
larger than $i$ (let's say $j$) in column  $c$.  The path then moves to the smallest entry (let's say $k$) larger than $j$ in the row below (if there are many such $k$'s the path goes through the leftmost such $k$).   We continue this way
until
we reach the row above that of the circle filled with an $i$.

It is important to realize that a tableau can be identified with its diagram given that the sequence  $\Omega=\Lambda_{(0)},\Lambda_{(1)},\dots, \Lambda_{(n)}=\Lambda$ can be recovered from the diagram.    We obtain the
diagram corresponding to $\Omega=\Lambda_{(0)},\Lambda_{(1)},\dots, \Lambda_{(n-1)}$ by removing the letters $n$ from the diagram 
(including, possibly, the circled one), and by moving the remaining circle one row above according to the following rule.  
A circle (filled let's say with letter $i$) in a given row $r$ 
is moved to row $r-1$ if there is an $n$ in row $r-1$ that 
belongs to its path.  Otherwise the circle
in row $r$ stays in its position. 

\smallskip

Consider the tableau $\scriptsize {\tableau[scY]{1&2&2&3& 4 \\ 2&4&6  \\4&5&\bl\tcercle{4} \\ 6&6 \\ \bl\tcercle{2} }}$ of weight $(1,\bar 3, 1, \bar 3, 1,3)$ and shape $(2,0;5,3,2)$ ($\Omega=\emptyset$ in the example). 
 The path for the $\tcercle{2}$ is  ${\scriptsize {\tableau[scY]{1&2&2&\bf{3} & 4 \\ 2&\bf{4}&6  \\4&\bf{5}&\bl\tcercle{4} \\ \bf{6} & 6 \\ \bl\tcercle{2} }}}$ which is
seen as follows:  the leftmost column without a non-circled 2 is column 4.  Since there is only one entry, a 3, in that column, the path starts there.   The smallest entry larger than 3 in the row below, the second one, is 4.  Then  the smallest entry larger than 4 in the third row is 5.  Finally, the smallest entry larger than 4 in row 4 is 6, and the path goes through the leftmost 6.  The path then stops since the circled 2 is in row 5.
The
path for $\tcercle{4}$ can similarly be seen to be  ${\scriptsize {\tableau[scY]{1&2&2&3& 4 \\ 2&4&\bf{6}  \\4&5&\bl\tcercle{4} \\ 6&6 \\ \bl\tcercle{2} }}}$.  
It is obvious from the example that the non-circled letters of an $s$-tableau form an ordinary tableau.  
However it is not obvious where the circled letters can be added, and the paths above
could only be constructed because the tableau was valid.

\smallskip

\noindent The sequence of superpartitions associated to that $s$-tableau can then be recovered by stripping successively the tableau of its largest letter and possibly moving circled letters one row above according 
to their paths:

  {\scriptsize
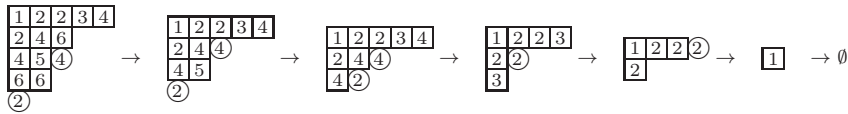
\begin{tabular}{lllllll} 
${\tableau[scY]{1&2&2&3& 4 \\ 2&4&6  \\4&5&\bl\tcercle{4} \\ 6&6 \\ \bl\tcercle{2} }}\rightarrow$&
${\tableau[scY]{1&2&2&3& 4 \\ 2&4 &\bl\tcercle{4}  \\4&5 \\ \bl\tcercle{2}\\  }}\rightarrow$&
${\tableau[scY]{1&2&2&3&4 \\ 2&4 &\bl\tcercle{4}  \\4&\bl\tcercle{2}  \\  }}\rightarrow$&
${\tableau[scY]{1&2&2&3 \\ 2& \bl\tcercle{2}  \\3 }}\rightarrow$&
${\tableau[scY]{1&2&2& \bl\tcercle{2}   \\ 2}}\rightarrow$&
${\tableau[scY]{1}}$ & $\rightarrow \emptyset$ 
\end{tabular}}

\medskip

Now, define the skew Schur function in superspace $s_{\Lambda/\Omega}$ as
\begin{equation}
s_{\Lambda/\Omega} = \sum_{T} (-1)^{{\rm inv}(T)} (x\theta)^T
\end{equation}
where the sum is over all $s$-tableaux of shape $\Lambda/\Omega$, and where
\begin{equation}
(x\theta)^T= \prod_i x_i^{|\alpha_i|} \prod_{j \, : \, {\rm type}(\alpha_j)={\rm fermionic}} \theta_j
\end{equation}
if $T$ is of weight $(\alpha_1,\dots,\alpha_n)$.  We stress that the product over anticommuting variables is ordered
from left to right over increasing indices.

\begin{proposition} \label{propsymfun}
 $s_{\Lambda/\Omega}$ is a symmetric function in superspace.  Moreover,
\begin{equation}
s_{\Lambda/\Omega}= \sum_{\Gamma} \bar K_{\Lambda/\Omega, \Gamma} \, m_\Gamma 
\end{equation}
where $\bar K_{\Lambda/\Omega, \Gamma}=\sum_T  (-1)^{{\rm inv}(T)}$, with the sum over 
 all $s$-tableaux $T$ of weight 
$$(\bar \Gamma_1,\dots,\bar \Gamma_m,\Gamma_{m+1},\dots,\Gamma_N)  \qquad \text{\rm for} \quad \Gamma=(\Gamma_1,\dots, \Gamma_m;\Gamma_{m+1},\dots,\Gamma_N)$$ 
\end{proposition}
\begin{proof}  By definition, the coefficient of $(x\theta)^\gamma$ in  $s_{\Lambda/\Omega}$ is equal to $\sum_T  (-1)^{{\rm inv}(T)}$,
where the sum is over all $s$-tableaux $T$ of weight $\gamma$.   From the construction of the $s$-tableaux from the
Pieri rules in Theorem~\ref{thm:esbar} and \ref{pinball2}, we have
\begin{equation} \label{honst}
s_{\Omega}^* \, h_\gamma=  \sum_{\Lambda} \bar K_{\Lambda/\Omega, \gamma} \, s_{\Lambda}^*
\end{equation}
where  $\bar K_{\Lambda/\Omega, \gamma}$ is equal to the signed sum over
all $s$-tableaux of weight $\gamma$.  By the commutation/anticommutation of the $h_i$ and $\tilde h_i$,
we have that if $\beta$ corresponds to $\gamma$ with two fermionic integers interchanged then
$\bar K_{\Lambda/\Omega, \beta}=-\bar K_{\Lambda/\Omega, \gamma}$, while if $\beta$ corresponds to $\gamma$ with two bosonic 
integers interchanged or with a bosonic integer and a fermionic integer interchanged then $\bar K_{\Lambda/\Omega, \beta}=\bar K_{\Lambda/\Omega, \gamma}$.  Thus the coefficients of $(x\theta)^\gamma$ and $(x\theta)^\beta$ in $s_{\Lambda/\Omega}$
are equal up to the right sign and
the first statement follows.  The other statement is then immediate since the coefficient of $(x\theta)^{(\bar \Gamma_1,\dots,\bar \Gamma_m,\Gamma_{m+1},\dots,\Gamma_N)}$ in $m_\Gamma$ is equal to 1.
\end{proof}
\begin{corollary} \label{coroschur}
We have that $s_{\Lambda/\emptyset}=s_{\Lambda}$.  Hence
\begin{equation}
s_{\Lambda} = \sum_{T} (-1)^{{\rm inv}(T)} (x\theta)^T
\end{equation}
where the sum is over all $s$-tableaux $T$ of shape $\Lambda$.
\end{corollary}
\begin{proof}  We have by \eqref{honst} and \eqref{h2snew} that $\bar K_{\Lambda/\emptyset, \Gamma}= \bar K_{\Lambda \Gamma}$.
The corollary then follows from \eqref{s2m}.
\end{proof}
We thus obtain the monomial expansion of  $s_{(3,1;2,1,1)}$ by listing every filling of the shape $(3,1;2,1,1)$ whose weight corresponds to a superpartition:

$$
\scriptsize{
{\tableau[scY]{1&1&1&\bl\tcercle{1}\\2&6\\3&\bl\tcercle{2} \\ 4\\5 }}\qquad
{\tableau[scY]{1&1&1&\bl\tcercle{1}\\2&5\\3&\bl\tcercle{2} \\ 4\\6 }}\qquad
{\tableau[scY]{1&1&1&\bl\tcercle{1}\\2&4\\3&\bl\tcercle{2} \\ 5\\6 }}\qquad
{\tableau[scY]{1&1&1&\bl\tcercle{1}\\2&3\\3&\bl\tcercle{2} \\ 4\\5 }}
}
$$
Therefore, $s_{(3,1;2,1,1)}=3\, m_{(3,1;1,1,1,1)}+m_{(3,1;2,1,1)}$ since there are 3 tableaux of weight
$(\bar 3,\bar 1, 1,1,1,1 )$ and one tableau of weight $(\bar 3, \bar 1,2,1,1)$.  We stress that
we don't have an easy criteria in general to determine whether a given filling is a valid tableau.  
However, in the example above, the rules for constructing tableaux immediately imply that we need to have
three non-circled 1's and one non-circled 2 (otherwise the circled 2 could never be in the second
column).  Then there are very few possibilities to fill the rest of the tableau with a weight corresponding to a superpartition. The case of weight
$(\bar 3,\bar 1, 1,1,1,1 )$ where a 3 is above the circled 2 is not allowed since again this would prevent the circled 2 from being in the
second column.

Note that when $\Omega \neq \emptyset$, the coefficient $\bar K_{\Lambda/\Omega, \Gamma}$ can be a negative integer.

\subsection{$\bar s$-tableaux} 
By \eqref{h2snew}, 
the tableaux needed to represent the Schur function
in superspace $\bar s_\Lambda$ are this time those stemming from the Pieri rules associated to the multiplication 
of $\bar s_\Lambda^*$ by $h_r$ or $\tilde h_r$ given in Theorem~\ref{thm:es} and \ref{thm:ets}.
We say that the sequence $\Omega=\Lambda_{(0)},\Lambda_{(1)},\dots, \Lambda_{(n)}=\Lambda$ is an $\bar s$-tableau 
of shape $\Lambda/\Omega$ and weight $(\alpha_1,\dots,\alpha_n)$, where 
$\alpha_i\in \mathbb \{0,\bar 0, 1,\bar 1,2,\bar 2,\dots \}$, if $\Omega=\Lambda_{(i)}$ and $\Lambda=\Lambda_{(i-1)}$ 
obey the conditions of Theorem~\ref{thm:es} with $\ell=\alpha_i$ whenever $\alpha_i$ is bosonic or
 the conditions of Theorem~\ref{thm:ets} 
with $\ell=|\alpha_i|$ whenever $\alpha_i$ is fermionic.  
An $\bar s$-tableau can be represented by a diagram constructed recursively in the following way: 
\begin{enumerate}
\item  the cells of $\Lambda_{(i)}^*/\Lambda_{(i-1)}^*$ are filled with the letter $i$.  In 
the fermionic case, the new circle is also filled with a letter $i$ 
\item the circles of $\Lambda_{(i-1)}$ that moved along a column or a row keep their fillings.
\end{enumerate}
As is the case for $s$-tableaux, the sign of an $\bar s$-tableau $T$ is equal to $(-1)^{{\rm inv} (T)}$, where ${{\rm inv} (T)}$ is 
the number of inversions of the word obtained by reading the filling of the circles from top to bottom.

It is important to realize that the sequence  $\Omega=\Lambda_{(0)},\Lambda_{(1)},\dots, \Lambda_{(n)}=\Lambda$ can be recovered from the diagram.  We obtain the
diagram corresponding to $\Omega=\Lambda_{(0)},\Lambda_{(1)},\dots, \Lambda_{(n-1)}$ by removing the letters $n$ from the diagram 
(including, possibly, the circled one), and by moving the circled letters one cell above if there is a letter $n$ above them
or to their left if there are letters $n$ to their left and none above them.  For instance, if one considers the $\bar s$-tableau $T$ below of weight $(3,\bar 1, \bar 2,1,\bar 3,5,\bar 0)$, the sequence of superpartitions associated to it can then be recovered by stripping successively the tableaux of their largest letter:

\medskip

{\scriptsize
\begin{tabular}{llllllll} 
$T={\tableau[scY]{1&1&1&5 & 6 & 6 & \bl \tcercle{5}  \\2&3&4 & 6 & \bl \tcercle{7} \\3&5& 6 & \bl \tcercle{3}\\5 & 6 & \bl\tcercle{2}\\}} \! \! \! \rightarrow$&
${\tableau[scY]{1&1&1&5 & 6 & 6 & \bl \tcercle{5}  \\2&3&4 & 6 \\3&5& 6 & \bl \tcercle{3}\\5 & 6 & \bl\tcercle{2}\\}}
\! \! \! \rightarrow$&
${\tableau[scY]{1&1&1&5 &\bl \tcercle{5}  \\2&3&4 & \bl \tcercle{3}\\3&5&\bl\tcercle{2} \\5 \\}}\!\! \rightarrow$&
${\tableau[scY]{1&1&1 & \bl \tcercle{3}  \\2&3 & 4 \\3&\bl\tcercle{2}\\ }} \! \rightarrow$&
${\tableau[scY]{1&1&1 &\bl \tcercle{3}  \\2&3  \\ 3 & \bl\tcercle{2}\\ }}\! \rightarrow$&
${\tableau[scY]{1&1 & 1 \\2 &\bl\tcercle{2} \\  }}\, \, \rightarrow$&
${\tableau[scY]{1 &1 &1 \\  }}$&
\end{tabular}}

We should stress that there is no immediate criteria to determine whether a given filling of a shape is a valid tableau.  It is only after checking that every removal of a letter corresponds to an application of a Pieri rule that we know that the filling is valid.  In the example above, removing the 6 corresponds to
an application of the Pieri rule $h_5$ since the the 6's form a horizontal strip and when moving 
the circled 2 and 3 above and the circled 5 to its left there are no collisions (two circles in the same row or
column).  Similarly, removing the 5's corresponds to the Pieri rule $\tilde h_3$ since the 5's form a horizontal strip with the circled one being the rightmost and when moving the circled 2 to its left and
the circled 3 above there are no collisions.

\smallskip

As was done in the previous subsection, define the Schur function in superspace $\bar s_{\Lambda/\Omega}$ as
\begin{equation}
\bar s_{\Lambda/\Omega} = \sum_{T} (-1)^{{\rm sign}(T)} (x\theta)^T
\end{equation}
where the sum is over all $\bar s$-tableaux of shape $\Lambda/\Omega$. 

The proof of the next proposition and its corollary are as in the previous subsection.
\begin{proposition}  $\bar s_{\Lambda/\Omega}$ is a symmetric function in superspace.  Moreover,
\begin{equation}
\bar s_{\Lambda/\Omega}= \sum_{\Gamma}  K_{\Lambda/\Omega, \Gamma} \, m_\Gamma 
\end{equation}
where $K_{\Lambda/\Omega, \Gamma}=\sum_T  (-1)^{{\rm sign}(T)}$, the sum over 
 all $\bar s$-tableaux $T$ of weight 
$$(\bar \Gamma_1,\dots,\bar \Gamma_m,\Gamma_{m+1},\dots,\Gamma_N)  \qquad \text{\rm for} \quad \Gamma=(\Gamma_1,\dots, \Gamma_m;\Gamma_{m+1},\dots,\Gamma_N)$$ 
\end{proposition}
\begin{corollary}  \label{coroschurb}
We have that $\bar s_{\Lambda/\emptyset}=\bar s_{\Lambda}$.  Hence
\begin{equation}
\bar s_{\Lambda} = \sum_{T} (-1)^{{\rm sign}(T)} (x\theta)^T
\end{equation}
where the sum is over all $\bar s$-tableaux in superspace of shape $\Lambda$. 
\end{corollary}
The monomial expansion of $\bar s_{(2,0;3)}$ is thus obtained by listing every filling of the shape $(2,0;3)$ whose weight is that of a superpartition
$$
\scriptsize{
{\tableau[scY]{1&4&6\\3&5&\bl\tcercle{1}\\ \bl\tcercle{2}  }}\,\, \, 
{\tableau[scY]{1&4&5\\3&6&\bl\tcercle{1}\\ \bl\tcercle{2}  }}\,\, \, 
{\tableau[scY]{1&3&6\\4&5&\bl\tcercle{1}\\ \bl\tcercle{2}  }}\qquad
{\tableau[scY]{1&3&4\\3&5&\bl\tcercle{1}\\ \bl\tcercle{2}  }} \,\, \, 
{\tableau[scY]{1&3&5\\3&4&\bl\tcercle{1}\\ \bl\tcercle{2}  }}\qquad
{\tableau[scY]{1&3&4\\3&4&\bl\tcercle{1}\\ \bl\tcercle{2}  }}\qquad
{\tableau[scY]{1&3&3\\3&4&\bl\tcercle{1}\\ \bl\tcercle{2}  }}\qquad
{\tableau[scY]{1&1&5\\3&4&\bl\tcercle{1}\\ \bl\tcercle{2}  }}\qquad
{\tableau[scY]{1&1&4\\3&3&\bl\tcercle{1}\\ \bl\tcercle{2}  }}\qquad
{\tableau[scY]{1&1&3\\3&3&\bl\tcercle{1}\\ \bl\tcercle{2}  }}
}
$$
Hence,
$$\bar s_{(2,0;3)}=3\, m_{(1,0;1,1,1,1)}+2\,m_{(1,0;2,1,1)}+m_{(1,0;2,2)}+m_{(1,0;3,1)}+m_{(2,0;1,1,1)}+m_{(2,0;2,1)}+m_{(2,0;3)}\, .$$ 
As was mentioned before, 
we don't have an easy criteria in general to determine whether a given filling is a valid tableau.  For instance, the tableau
$\scriptsize{
{\tableau[scY]{1&3&5\\4&6&\bl\tcercle{1}\\ \bl\tcercle{2}  }}}$ is not valid because the circled 1 and
the circled 2 collide at the moment of removing letter 4:
$$
\scriptsize{
{\tableau[scY]{1&3&5\\4&6&\bl\tcercle{1}\\ \bl\tcercle{2}  }}} \quad  \rightarrow \quad \scriptsize{
{\tableau[scY]{1&3&5\\4&\bl\tcercle{1}\\ \bl\tcercle{2}  }}} \quad \rightarrow \quad \scriptsize{
{\tableau[scY]{1&3\\4&\bl\tcercle{1}\\ \bl\tcercle{2}  }}} 
$$

We note that, as is the case for $\bar K_{\Lambda/\Omega, \Gamma}$,
the coefficient $K_{\Lambda/\Omega, \Gamma}$ can be a negative integer when $\Omega \neq \emptyset$.

\begin{remark} \label{remarktableaux} 
Combinatorial formulas, as
tableau generating series, for the monomial expansions of
$s_\Lambda$ and $\bar s_{\Lambda}$ are conjectured in \cite{BM}.  Their formulas essentially only 
concern the cases when the fermionic Pieri rules are applied last, that is, when the weights are of the type
$(a_1,...,a_\ell,\bar b_1,\dots,\bar b_m)$.  In those cases, their sums are cancellation free while 
ours are not (we have cancellation free sums when the fermionic Pieri rules are applied first).
We have not tried to prove their conjectures in this article, even though it would be interesting to
understand their combinatorial formula for $\bar s_\Lambda$ since it involves a subset of the $\bar s$-tableaux. 
 \end{remark}

\section{Further properties of skew Schur functions in superspace}
In this section, we first show that skew Schur functions in superspace can be obtained from Schur functions in superspace by taking their
adjoint with respect to the scalar product 
$\langle \! \langle \cdot \, , \cdot \rangle \! \rangle $.  We then connect the generalization to superspace of the
Littlewood-Richardson coefficients to skew Schur functions in superspace.  
These basic properties of skew Schur functions in superspace 
generalize well known properties in the classical case ($m=0$).

\begin{corollary} We have
\begin{equation}
\langle \! \langle s_{\Omega}^* \, f, s_{\Lambda} \rangle \! \rangle =
\langle \! \langle f , s_{\Lambda/\Omega} \rangle \! \rangle   \qquad {\rm and} \qquad
 \langle \! \langle   \bar s_{\Omega}^* \, f, \bar s_{\Lambda} \rangle \! \rangle =
\langle \! \langle f , \bar s_{\Lambda/\Omega} \rangle \! \rangle
\end{equation}
for all symmetric functions in superspace $f$.
\end{corollary}
\begin{proof}
Since the two cases are similar, we will only prove the first formula.  It suffices to consider that $f=h_{\Gamma}$ since the $h_\Lambda$'s form a basis of the space.  From Proposition~\ref{propsymfun} and the duality \eqref{hmdual} between the 
$h_\Lambda$ and $m_\Lambda$ bases, we have
\begin{equation}
\langle \! \langle h_\Gamma , s_{\Lambda/\Omega} \rangle \! \rangle =  \bar K_{\Lambda/\Omega, \Gamma} 
\end{equation}
On the other hand, using \eqref{honst}, we get from Proposition~\ref{propdual} that
\begin{equation}
\langle \! \langle  s_{\Omega}^* \, h_{\Gamma},
  s_{\Lambda} \rangle \! \rangle =  \langle \! \langle  \sum_{\Delta} \bar K_{\Delta/\Omega, \Gamma} \, s_{\Delta}^*,
  s_{\Lambda} \rangle \! \rangle =
\bar K_{\Lambda/\Omega, \Gamma} 
\end{equation}
and the result follows.
\end{proof}
Define $\bar c^\Lambda_{\Gamma \Omega}$ and $c^\Lambda_{\Gamma \Omega}$ to be respectively such that 
\begin{equation} 
\bar s_{\Gamma}\,  \bar s_{\Omega} = \sum_{\Lambda} \bar c^\Lambda_{\Gamma \Omega}\,  \bar s_{\Lambda}
 \qquad {\rm and} \qquad    
s_{\Gamma}\,  s_{\Omega} = \sum_{\Lambda} c^\Lambda_{\Gamma \Omega}\,  s_{\Lambda}    
\end{equation}
It is immediate from the (anti-)commutation relations between the Schur functions in superspace
 that if $\Gamma$ and $\Omega$ are respectively of fermionic degrees $a$ and $b$, then 
$\bar c^\Lambda_{\Gamma \Omega}=(-1)^{ab} \,  \bar c^\Lambda_{\Omega \Gamma}$ and 
$c^\Lambda_{\Gamma \Omega}=(-1)^{ab}  \, c^\Lambda_{\Omega \Gamma}$.  Even though $\bar c^\Lambda_{\Gamma \Omega}$ and 
$c^\Lambda_{\Gamma \Omega}$ are not always nonnegative from these relations, 
we can consider them as generalizations to superspace
of the Littlewood-Richardson coefficients. 

 We now extend to superspace the well-known connection between 
Littlewood-Richardson coefficients and skew Schur functions.
\begin{proposition} We have
\begin{equation}
s_{\Lambda/\Omega} =  \sum_{\Gamma}  \bar c^{\Lambda'}_{\Gamma' \Omega'} \, s_\Gamma
\qquad {\rm and} \qquad  
\bar s_{\Lambda/\Omega} = \sum_{\Gamma}  c^{\Lambda}_{\Omega \Gamma} \, \bar s_\Gamma
\end{equation}
Furthermore, 
 $c^\Lambda_{\Omega \Gamma}=c^{\Lambda'}_{\Gamma' \Omega'}$.

\end{proposition}
\begin{proof}
Suppose that $\Gamma$ and $\Omega$ are respectively of fermionic degrees $a$ and $b$. 
The symmetry $c^\Lambda_{\Omega \Gamma}=(-1)^{ab}  c^{\Lambda'}_{\Omega' \Gamma'}$ is from Corollary~\ref{corosdual}
an immediate consequence of applying the homomorphism $\phi$ on $s_{\Omega}\,  s_{\Gamma} = \sum_{\Lambda} c^\Lambda_{\Omega \Gamma}\,  s_{\Lambda}$.  In effect, applying the homomorphism 
gives
\begin{equation} \label{eqsign}
(-1)^{\binom{a}{2}+\binom{b}{2}} s_{\Omega'}\,  s_{\Gamma'} = (-1)^{\binom{a+b}{2}} \sum_{\Lambda} 
 c^\Lambda_{\Omega \Gamma}\,  s_{\Lambda'}   
\end{equation}
since $\Lambda$ is necessarily of fermionic degree $a+b$.  Hence
$s_{\Omega'}\,  s_{\Gamma'} = (-1)^{ab} \sum_{\Lambda} 
 c^\Lambda_{\Omega \Gamma}\,  s_{\Lambda'}$, which gives the symmetry
$c^\Lambda_{\Omega \Gamma}= (-1)^{ab} \, c^{\Lambda'}_{\Omega' \Gamma'}= c^{\Lambda'}_{\Gamma' \Omega'}$.

We now prove the first formula.   Using the previous corollary, we have
\begin{equation}
\langle \! \langle s_{\Omega}^*  s_{\Gamma}^*, s_{\Lambda} \rangle \! \rangle =
\langle \! \langle s_\Gamma^* , s_{\Lambda/\Omega} \rangle \! \rangle
\end{equation}
Therefore, if  $d^\Lambda_{\Omega \Gamma}$ is such that
\begin{equation} \label{eqss}
s_{\Omega}^*\,  s_{\Gamma}^* = \sum_{\Lambda} d^\Lambda_{\Omega \Gamma}\,  s_{\Lambda}^*    
\end{equation}
then we have by Proposition~\ref{propdual}
\begin{equation}
s_{\Lambda/\Omega} = \sum_{\Gamma} d^{\Lambda}_{\Omega \Gamma} \, s_\Gamma
\end{equation}
Applying $\omega$ on both sides of \eqref{eqss}, we get as in equation \eqref{eqsign},
\begin{equation}
(-1)^{\binom{a}{2}+\binom{b}{2}}\bar s_{\Omega'}\,  \bar s_{\Gamma'} = \sum_{\Lambda} 
(-1)^{\binom{a+b}{2}} d^\Lambda_{\Omega \Gamma}\,  \bar s_{\Lambda'}   
\end{equation}
Therefore
\begin{equation}
 d^\Lambda_{\Omega \Gamma} =  (-1)^{ab} \, 
c^{\Lambda'}_{\Omega' \Gamma'} =  \bar c^{\Lambda'}_{\Gamma' \Omega'} 
\end{equation}
and the first formula follows.

We can prove in a similar way that $\bar s_{\Lambda/\Omega} = \sum_{\Gamma}  c^{\Lambda'}_{\Gamma' \Omega'} \, \bar s_\Gamma$.
Using the symmetry $c^{\Lambda'}_{\Gamma' \Omega'}=c^{\Lambda}_{\Omega \Gamma}$, the second formula is then seen to hold.
\end{proof}

\begin{remark} \label{remark}
Somewhat surprisingly, the coefficient  $\bar c^\Lambda_{\Omega \Gamma}$ does not
have in general any symmetry under conjugation.
 For instance, it can be checked that
if $\Gamma=(1;)$, $\Omega=(0;)$ and $\Lambda=(1,0;)$ then 
$\bar c^{\Lambda}_{\Omega \Gamma}=1$ while
$\bar c^{\Lambda'}_{\Gamma' \Omega'}=0$.  
\end{remark}

\section{Large $m$ and $N$ limit}

Let $\Lambda$ be a superpartition of fermionic degree $m$.
Lemmas \ref{lems} and \ref{lemsb} state that
in the
identification \eqref{eqidenti}
between $\Lambda_N^m$ and $\mathbb Q [x_1,\dots,x_N]^{S_m \times S_{m^c}}$
we have
\begin{equation}
s_\Lambda \quad \longleftrightarrow \quad  (-1)^{\binom{m}{2}}
\partial_{\omega_m} \pi_{\omega_{m^c}} \hat K_{(\Lambda^a)^R,\Lambda^s}(x)
\end{equation}
and
\begin{equation}
\bar s_\Lambda \quad \longleftrightarrow \quad  
\partial_{\omega_{(N-m)^c}}' K_{(\Lambda^s)^R,\Lambda^a}(y)
\end{equation}
 where 
$y$ stands for the variables $y_1=x_N,y_2=x_{N-1},\dots,y_N=x_1$ and where
the $'$ indicates that the divided differences act on the $y$ variables.

Let $\mu=\Lambda^s$ and  $\lambda=\Lambda^a-\delta_m$, where $\delta_m=(m-1,m-2,\dots,0)$.
It was shown in \cite{BLM} that when $m \geq |\lambda|+|\mu|$  and $N-m \geq |\lambda|+|\mu|$ (the large $m$ and $N$ limit), the 
identification becomes
\begin{equation}
s_\Lambda \quad \longleftrightarrow \quad s_\lambda(x_1,\dots,x_m) s_\mu(x_{m+1},\dots,x_N)
\end{equation}
and
\begin{equation}
\bar s_\Lambda \quad \longleftrightarrow  \quad s_\lambda(x_1,\dots,x_N) s_\mu(x_{m+1},\dots,x_N)
\end{equation}
We thus have the following proposition which does not seem to have an easy proof in the Key world.
\begin{proposition}  Let $\mu=\Lambda^s$ and  $\lambda=\Lambda^a-\delta_m$, where $\delta_m=(m-1,m-2,\dots,0)$. 
If  $m \geq |\lambda|+|\mu|$  and $N-m \geq |\lambda|+|\mu|$ then
\begin{equation}
 (-1)^{\binom{m}{2}}
\partial_{\omega_m} \pi_{\omega_{m^c}} \hat K_{(\Lambda^a)^R,\Lambda^s}(x)=  s_\lambda(x_1,\dots,x_m) s_\mu(x_{m+1},\dots,x_N)
\end{equation}
and
\begin{equation}
\partial_{\omega_{m^c}} K_{(\Lambda^s)^R,\Lambda^a}(x_N,x_{N-1},\dots,x_1) =  s_\lambda(x_1,\dots,x_N) s_\mu(x_{m+1},\dots,x_N)
\end{equation}
\end{proposition}

\section{Cauchy formulas}
As is the case in symmetric function theory, the dualities of Proposition~\ref{propdual} translate into Cauchy type formulas.  
The one most relevant to this work is the following.  Given two bases $\{f_\Lambda\}_\Lambda$ and $\{g_\Lambda\}_\Lambda$ dual to each other with respect to the scalar product $\langle \! \langle \cdot \, , \cdot \rangle \! \rangle$, we have \cite{DLM1} 
\begin{equation}
\prod_{i,j}(1+x_i y_j + \theta_i \phi_j) =
\sum_\Lambda  (-1)^{\binom{m}{2}} f_\Lambda(x,\theta) \, \omega(g_{\Lambda})(y,\phi)
\end{equation}
where $m$ is the fermionic degree of $\Lambda$, and where the variables $y_1,y_2,\dots$ are ordinary variables while the variables $\phi_1,\phi_2,\dots$ are anticommuting (they also anticommute with the $\theta_i$'s).  Using Proposition~\ref{propdual}, we then get
\begin{equation}
\prod_{i,j}(1+x_i y_j + \theta_i \phi_j) = 
\sum_\Lambda s_\Lambda(x,\theta) \, \bar s_{\Lambda'}(y,\phi)
\end{equation}
The tableaux generating series of Propositions~\ref{coroschur} and ~\ref{coroschurb} suggest that there should exist a bijective proof of that formula using an extension to superspace of the dual Robinson-Schensted-Knuth algorithm \cite{Sta}.

\begin{acknow}
The authors would like to thank L.-F. Pr\'eville-Ratelle for his help in uncovering the Pieri rules of
Section~\ref{subsec41}.  The authors would also like to thank O. Blondeau-Fournier, P. Desrosiers and P. Mathieu for very fruitful discussions.
This work was supported by FONDECYT (Fondo Nacional de Desarrollo Cient\'{\i}fico y Tecnol\'ogico de Chile) postdoctoral grant \#{3130631} (M. J.) and regular grant \#{1130696} (L. L.).
\end{acknow}

\appendix

\section{}

In the appendix, we prove various technical results that were needed to prove the Pieri rules.  They range from 
the very elementary to the quite intricate (see Lemma~\ref{lemfinal} for instance).  Although the proofs of the
elementary ones may 
probably be found in the literature, we include them for completeness.
\begin{lemma} \label{lemreorder}
Using the notation of Section~\ref{secpieri}, we have
\begin{equation} 
K_{(\Lambda^s)^R,\Lambda^a}(x)=\pi_{\omega_{N-m}}\pi_{(N-m,\alpha_1)}\pi_{(N-m+1,\alpha_2)}\cdots\pi_{(N-1,\alpha_m)}x^{\Lambda^*}
\end{equation} where $\alpha_i$ is the row of the $i$-th circle (starting from the top) in $\Lambda$.
\end{lemma}
\begin{proof} By definition, we have $K_{\Lambda^*}(x)=x^{\Lambda^*}$.  It is then easy to see, again from the definition of Key polynomials, that $\pi_{(N-1,\alpha_m)} x^{\Lambda^*}=K_\gamma(x)$, where
$\gamma$ is the composition obtained by acting with
the cycle $(\alpha_m,N,N-1,\dots,\alpha_m+1)$ on the rows of $\Lambda^*$, that is, by sending row $\alpha_m$ to row
$N$, and shifting rows $\alpha_m+1$ to $N$ one row up.
  The action of $\pi_{(N-m,\alpha_1)}\pi_{(N-m+1,\alpha_2)}\cdots\pi_{(N-1,\alpha_m)}$
on $x^{\Lambda^*}$ thus produces $K_{\Lambda^s,\Lambda^a}(x)$, since $(\Lambda^s,\Lambda^a)$ corresponds to the composition obtained by sending row $\alpha_i$ to row
$N-m+i$, for $i=1,\dots,m$, and shifting the remaining rows up if necessary.  Finally, acting with $\pi_{\omega_{N-m}}$ on  $K_{\Lambda^s,\Lambda^a}(x)$
simply reorders the entries of $\Lambda^s$ in a non-decreasing way to produce $K_{(\Lambda^s)^R,\Lambda^a}(x)$.
\end{proof}

\begin{lemma} \label{lemapelx}
Using the notation of Section~\ref{secpieri}, we have
\begin{equation}
x^{\Lambda^*} e_\ell=
\sum_{i_1+i_2+\dots+i_k=\ell}
\boldsymbol{\pi}_{i_1,I_1}  \boldsymbol{\pi}_{i_2,I_2} \cdots \boldsymbol{\pi}_{i_k,I_k} 
x^{\Lambda^*} x^{(i_1,I_1)}x^{(i_2,I_2)}\cdots x^{(i_k,I_k)}
\end{equation}
\end{lemma}

\begin{proof}
By construction, $x^{\Lambda^*}$ commutes with $\boldsymbol{\pi}_{i_1,I_1}  \boldsymbol{\pi}_{i_2,I_2} \cdots \boldsymbol{\pi}_{i_k,I_k}$.
We thus only need to prove that 
\begin{equation}
e_\ell=
\sum_{i_1+i_2+\dots+i_k=\ell}
\boldsymbol{\pi}_{i_1,I_1}  \boldsymbol{\pi}_{i_2,I_2} \cdots \boldsymbol{\pi}_{i_k,I_k} 
x^{(i_1,I_1)}x^{(i_2,I_2)}\cdots x^{(i_k,I_k)}
\end{equation}
Let $X_{I}$ be the alphabet made out of the variables $x_{j}$ for $j\in I$.
It is well known (see for instance \cite{Mac}) that
\begin{equation}
e_\ell(x_1,\dots,x_N) = \sum_{i_1+i_2+\dots+i_k=\ell} e_{i_1}(X_{I_1}) \, e_{i_2}(X_{I_2}) \cdots  e_{i_k}(X_{I_k})  
\end{equation}
Since the $\boldsymbol{\pi}_{i_j,I_j}$'s commute among themselves, the 
result will then follow if we can show that $e(X_{I_j})=  \boldsymbol{\pi}_{i_j,I_j} \, x^{(i_j,I_j)}$ for all $j$.  
We only consider the case $j=1$, the remaining ones being similar. For simplicity we let $i_1=i$ and $I_1=\{1,\dots,r\}$.
By definition of Key polynomials, $K_{1^{i},0^{r-i}}= x^{1^i,0^{r-i}}$.  Hence, from \eqref{keyschur},  we have
\begin{equation}
s_{1^i}(x_1,\dots,x_r)= K_{0^{r-i},1^{i}}=
\pi_{[r-i,r-1]}\dots\pi_{[2,i+1]}\pi_{[1,i]} K_{1^{i},0^{r-i}}= \pi_{[r-i,r-1]}\dots\pi_{[2,i+1]}\pi_{[1,i]} x^{1^i,0^{r-i}} 
\end{equation}
and the result follows since it is well known that $e_i(x_1,\dots,x_{r})=s_{1^i}(x_1,\dots,x_{r})$.
\end{proof}

\begin{lemma}\label{lem:propR}
$\mathcal{R}_{N,[\alpha_1,\dots,\alpha_m]}$ obeys the following properties:
\begin{enumerate}
\item
 $\mathcal{R}_{N,[\alpha_1,\ldots,\alpha_m]}\pi_{\alpha_i-1}=
 \mathcal{R}_{N,[\alpha_1,\ldots,\alpha_{i-1},\alpha_i-1,\alpha_{i+1},\ldots,\alpha_m]}.$
 \item
 $\mathcal{R}_{N,[\alpha_1,\ldots,\alpha_m]}\pi_{\beta}=
 \mathcal{R}_{N,[\alpha_1,\ldots,\alpha_m]}$ if $\beta \not \in  \{\alpha_1-1,\dots,\alpha_m-1\}$.
\end{enumerate}
where we use the notation of Section~\ref{secpieri}.
\end{lemma}

\begin{proof} Let $v=N-m$.  Consider first Case {\it (1)}.  We have
\begin{eqnarray}
\mathcal{R}_{N,[\alpha_1,\ldots,\alpha_m]}\pi_{\alpha_i-1}
&=&\pi_{\omega_{v}}\pi_{(v,\alpha_1)}\pi_{(v+1,\alpha_2)}\cdots\pi_{(N-1,\alpha_m)}\pi_{\alpha_i-1}\nonumber\\
&=&\pi_{\omega_v}\pi_{(v,\alpha_1)}\pi_{(v+1,\alpha_2)}\cdots
\pi_{(v+i-1,\alpha_i)}\pi_{\alpha_i-1}\cdots\pi_{(N-1,\alpha_m)}\nonumber\\
&=&\pi_{\omega_v}\pi_{(v,\alpha_1)}\pi_{(v+1,\alpha_2)}\cdots
\pi_{(v+i-1,\alpha_i-1)}\cdots\pi_{(N-1,\alpha_m)}\nonumber\\
&=&\mathcal{R}_{N,[\alpha_1,\ldots,\alpha_i-1\ldots,\alpha_m]}\nonumber
\end{eqnarray}
Case {\it (2)} has three cases.  If $\beta<\alpha_1-1$ then
\begin{eqnarray}
\mathcal{R}_{N,[\alpha_1,\ldots,\alpha_m]}\pi_{\beta}&=&\pi_{\omega_v}\pi_{(v,\alpha_1)}\pi_{(v+1,\alpha_2)}\cdots\pi_{(N-1,\alpha_m)}\pi_{\beta}\nonumber\\
&=&\pi_{\omega_v}\pi_{\beta}\pi_{(v,\alpha_1)}\pi_{(v+1,\alpha_2)}\cdots\pi_{(N-1,\alpha_m)}\nonumber\\
&=&\pi_{\omega_v}\pi_{(v,\alpha_1)}\pi_{(v+1,\alpha_2)}\cdots\pi_{(N-1,\alpha_m)}\nonumber\\
&=&\mathcal{R}_{N,[\alpha_1,\ldots,\alpha_m]}\nonumber
\end{eqnarray}
If $\beta=\alpha_i$ for some $i$ then
\begin{eqnarray}
\mathcal{R}_{N,[\alpha_1,\ldots,\alpha_m]}\pi_{\alpha_i}&=&\pi_{\omega_v}\pi_{(v,\alpha_1)}\pi_{(v+1,\alpha_2)}\cdots\pi_{(N-1,\alpha_m)}\pi_{\alpha_i}\nonumber\\
&=&\pi_{\omega_v}\pi_{(v,\alpha_1)}\pi_{(v+1,\alpha_2)}\cdots
\pi_{(v+i-1,\alpha_{i})}\pi_{\alpha_i}\cdots\pi_{(N-1,\alpha_m)}\nonumber\\
&=&\pi_{\omega_v}\pi_{(v,\alpha_1)}\pi_{(v+1,\alpha_2)}\cdots
\pi_{(v+i-1,\alpha_{i})}\cdots\pi_{(N-1,\alpha_m)}\nonumber\\
&=&\mathcal{R}_{N,[\alpha_1,\ldots,\alpha_m]}\nonumber
\end{eqnarray}
Finally, if $\alpha_{i}<\beta<\alpha_{i+1}-1$ (the argument is the same for $\alpha_m<\beta$) then
\begin{eqnarray}
\mathcal{R}_{N,[\alpha_1,\ldots,\alpha_m]}\pi_{\beta}&=&\pi_{\omega_v}\pi_{(v,\alpha_1)}\pi_{(v+1,\alpha_2)}\cdots\pi_{(N-1,\alpha_m)}\pi_{\beta}\nonumber\\
&=&\pi_{\omega_v}\pi_{(v,\alpha_1)}\pi_{(v+1,\alpha_2)}\cdots
\pi_{(v+i-1,\alpha_{i})}\pi_{\beta}\cdots\pi_{(N-1,\alpha_m)}\nonumber\\
&=&\pi_{\omega_v}\pi_{(v,\alpha_1)}\pi_{(v+1,\alpha_2)}\cdots
\pi_{(v+i-1,\beta-1)}\pi_{\beta}\pi_{(\beta-2,\alpha_{i})}\cdots\pi_{(N-1,\alpha_m)}\nonumber\\
&=&\pi_{\omega_v}\pi_{(v,\alpha_1)}\pi_{(v+1,\alpha_2)}\cdots
\pi_{(v+i-1,\beta+1)}\pi_\beta\pi_{\beta-1}\pi_{\beta}\pi_{(\beta-2,\alpha_{i})}\cdots\pi_{(N-1,\alpha_m)}\nonumber\\
&=&\pi_{\omega_v}\pi_{(v,\alpha_1)}\pi_{(v+1,\alpha_2)}\cdots
\pi_{(v+i-1,\beta+1)}\pi_{\beta-1}\pi_{\beta}\pi_{\beta-1}\pi_{(\beta-2,\alpha_{i})}\cdots\pi_{(N-1,\alpha_m)}\nonumber\\
&=&\pi_{\omega_v}\pi_{(v,\alpha_1)}\pi_{(v+1,\alpha_2)}\cdots
\pi_{(v+i-2,\alpha_{i-1})}\pi_{\beta-1}\pi_{(v+i-1,\alpha_{i})}\cdots\pi_{(N-1,\alpha_m)}\nonumber
\end{eqnarray}
At this point, we repeat as in the second line until we get to the beginning of the product. 
\begin{eqnarray}
\pi_{\omega_v}\pi_{\beta-i}\pi_{(v,\alpha_1)}\pi_{(v+1,\alpha_2)}
\cdots\pi_{(N-1,\alpha_m)}&=&\pi_{\omega_v}\pi_{(v,\alpha_1)}\pi_{(v+1,\alpha_2)}
\cdots\pi_{(N-1,\alpha_m)}\nonumber\\
&=&\mathcal{R}_{N,[\alpha_1,\ldots,\alpha_m]}\nonumber
\end{eqnarray}
\end{proof}

\begin{lemma}\label{lem:Ralpha}
If $\mu_{\alpha_i}=\mu_{\alpha_i-1}$ then
\begin{equation}
\mathcal{R}_{N,[\alpha_1,\dots,\alpha_i,\dots,\alpha_m]}x^\mu=
\mathcal{R}_{N,[\alpha_1,\dots,\alpha_i-1,\dots,\alpha_m]}x^\mu.
\end{equation}
\end{lemma}
\begin{proof}
Since $\mu_{\alpha_i}=\mu_{\alpha_i-1}$, we have $x^\mu=\pi_{\alpha_i-1}x^\mu$.
By Lemma \ref{lem:propR}, we then have that
\begin{equation}
\mathcal{R}_{N,[\alpha_1,\dots,\alpha_i,\dots,\alpha_m]}x^\mu= \mathcal{R}_{N,[\alpha_1,\dots,\alpha_i,\dots,\alpha_m]}\pi_{\alpha_i-1}x^\mu 
=\mathcal{R}_{N,[\alpha_1,\dots,\alpha_i-1,\dots,\alpha_m]}x^\mu. \nonumber
\end{equation}
\end{proof}

\begin{corollary} \label{coroalphap}
If there is no addable corner in
row $\alpha_i$ of $\mu$ then there exists a row, call it $\alpha'_i$ such that $\mu_{\alpha_i}=\mu_{\alpha'_i}$ and there is an addable corner in row $\alpha'_i$ of $\mu$. Furthermore,
\begin{equation}
\mathcal{R}_{N,[\alpha_1,\dots,\alpha_i,\dots,\alpha_m]}x^\mu=
\mathcal{R}_{N,[\alpha_1,\dots,\alpha'_i,\dots,\alpha_m]}x^\mu.
\end{equation}
\end{corollary}

\begin{proof}
Follows by repeated applications of Lemma \ref{lem:Ralpha}.
\end{proof}

\begin{lemma}\label{lem:pin}  We have
\begin{equation}
 \pi_{(n-1,i)}\pi_{(n,i)}=\pi_n(\pi_{n-1}\pi_n)(\pi_{n-2}\pi_{n-1})\dots(\pi_{i}\pi_{i+1})
 \end{equation}
\end{lemma}
\begin{proof} Easy by induction.
\end{proof}

\begin{lemma} \label{lemequal}
If $\alpha_i=\alpha_{i+1}$ for some $i$ then
\begin{equation}
\partial_{\omega_{(N-m)^c}}\mathcal{R}_{N,[\alpha_1,\dots,\alpha_m]}=0
\end{equation}
\end{lemma}
\begin{proof} Let $v=N-m$.  We have
\begin{eqnarray}
\partial_{\omega_{v^c}}\mathcal{R}_{N,[\alpha_1,\dots,\alpha_i,\alpha_i,\dots,\alpha_m]}=&
\partial_{\omega_{v^c}}\pi_{\omega_v}\pi_{(v,\alpha_1)}\cdots\pi_{(v+i-1,\alpha_i)}\pi_{(v+i,\alpha_i)}
\cdots\pi_{(N-1,\alpha_m)} \nonumber 
\end{eqnarray}
Use Lemma \ref{lem:pin} to rewrite
$\pi_{(v+i-1,\alpha_i)}\pi_{(v+i,\alpha_i)}$ as $\pi_{v+i}$ to the left times a string of isobaric divided differences. 
 Then $\pi_{v+i}$ commutes with every isobaric divided difference to its left, and the result follows given that
$\partial_{\omega_{v^c}}\pi_{v+i}=0$.
\end{proof}

\begin{lemma}\label{lem:xpipi}
 For $i<v$,
 \begin{equation}
  x_v\pi_{\omega_{v-1}}\pi_{(v-1,i)}=\pi_{\omega_{v-1}}\pi_{(v-1,i+1)}(\pi_i-1)x_i.
 \end{equation}
\end{lemma}
\begin{proof}
We first commute $x_v$ and $\pi_{\omega_{v-1}}$.  Then we use the relation $x_{r+1}\, \pi_r= (\pi_r -1)x_r$ repeatedly to
get $x_v \pi_{(v-1,i)}= (\pi_{v-1}-1) (\pi_{v-2}-1)\dots (\pi_i-1) \, x_i$.  Since $\pi_{\omega_{v-1}} {(\pi_r-1)}=0$ for
all $r<v-1$, we have
\begin{align}
 x_v\pi_{\omega_{v-1}}\pi_{(v-1,i)}&= (\pi_{v-1}-1) (\pi_{v-2}-1)(\pi_{v-3}-1)\dots (\pi_i-1) \, x_i \nonumber \\ 
&= \pi_{v-1}(\pi_{v-2}-1)(\pi_{v-3}-1)\dots (\pi_i-1) \, x_i 
\nonumber \\ 
&= \pi_{v-1}\pi_{v-2}(\pi_{v-3}-1)\dots (\pi_i-1) \, x_i 
\end{align}
where we used the fact that $\pi_{v-1}$ and $(\pi_{v-3}-1)$ commute.  Continuing in this way, we get that
$ x_v\pi_{\omega_{v-1}}\pi_{(v-1,i)}= \pi_{\omega_{v-1}}\pi_{v-1}\pi_{v-2}\pi_{v-3}\cdots \pi_{i+1}(\pi_i-1)x_i$, and the lemma
follows.
\end{proof}

Recall that $e_{\ell}^{(v)}$ and $e_{\ell}^{(1,\dots,v)}$ were defined in Section~\ref{secpieri}.  
The following two lemmas are immediate.
\begin{lemma}\label{lem:elk1}
 \begin{equation}
  e_\ell^{(v)}=e_\ell-x_v e^{(v)}_{\ell-1}
 \end{equation}
\end{lemma}

\begin{lemma}\label{lem:el1k1}
 \begin{equation}
  e_\ell^{(1,\dots,v)}=e_\ell^{(1,\dots,v-1)}-x_ve_{\ell-1}^{(1,\dots,v)}
 \end{equation}
\end{lemma}

\begin{lemma}\label{lem:elpi1}
 \begin{equation}
  e_\ell^{(1,\dots,r)}\pi_r=\pi_r\left(e_\ell^{(1,\dots,r+1)}+x_r e_{\ell-1}^{(1,\dots,r+1)}\right)-x_re_{\ell-1}^{(1,\dots,r+1)}
 \end{equation}
\end{lemma}

\begin{proof}
 From Lemma \ref{lem:el1k1}, we have that
$$e_\ell^{(1,\dots,r)}=e_\ell^{(1,\dots,r+1)}+x_{r+1} e_{\ell-1}^{(1,\dots,r+1)}$$
and by the fact that $x_{r+1}\pi_r=(\pi_r-1)x_r$, we have
\begin{eqnarray}
e_\ell^{(1,\dots,r)}\pi_r&=&e_\ell^{(1,\dots,r+1)}\pi_r+x_{r+1} e_{\ell-1}^{(1,\dots,r+1)}\pi_r\nonumber\\
&=&\pi_re_\ell^{(1,\dots,r+1)}+x_{r+1}\pi_r e_{\ell-1}^{(1,\dots,r+1)}\nonumber\\
&=&\pi_re_\ell^{(1,\dots,r+1)}+((\pi_r-1)x_r) e_{\ell-1}^{(1,\dots,r+1)}\nonumber\\
&=&\pi_r\left(e_\ell^{(1,\dots,r+1)}+x_r e_{\ell-1}^{(1,\dots,r+1)}\right)-x_re_{\ell-1}^{(1,\dots,r+1)}\nonumber
\end{eqnarray}
\end{proof}

\begin{lemma}\label{lem:elvpi} We have
\begin{equation}
e^{(v)}_\ell \pi_{\omega_v}=\sum_{j=1}^v \pi_{\omega_{v-1}}\pi_{(v-1,j)}x_1x_2\dots x_{j-1}e_{\ell-j+1}^{(1,2,\dots,j)}
\end{equation}
where we note that $e_{0}^{(1,2,\dots,j)}=1$ and $e_n^{(1,2,\dots,j)}=0$ for $n<0$ .
 \end{lemma}

\begin{proof}
We proceed by induction on $\ell$. If $\ell=0$ then there is only the term $j=1$ in the sum:
\begin{equation}
e^{(v)}_0 \pi_{\omega_v}=\pi_{\omega_v} 
=\pi_{\omega_{v-1}}\pi_{(v-1,1)} 
\end{equation}
and the result is seen to hold.  Now, suppose that the result is true for
$k\le \ell$. 
By Lemma \ref{lem:elk1}, we have
\begin{equation}
e_{\ell+1}^{(v)}\pi_{\omega_{v}}
=(e_{\ell+1}-x_ve_{\ell}^{(v)})\pi_{\omega_{v}}
=e_{\ell+1}\pi_{\omega_{v}}-x_ve_{\ell}^{(v)}\pi_{\omega_{v}}
\end{equation}
Hence, by induction, 
\begin{eqnarray}
e_{\ell+1}^{(v)}\pi_{\omega_{v}} &=&\pi_{\omega_{v}}e_{\ell+1}-x_v\sum_{j=1}^v \pi_{\omega_{v-1}}\pi_{(v-1,j)}x_1x_2\dots x_{j-1}e_{\ell-j+1}^{(1,2,\dots,j)}\nonumber\\
&=&\pi_{\omega_{v}}e_{\ell+1}-x_v\pi_{\omega_{v-1}}x_1\dots x_{v-1}e_{\ell-v+1}^{(1,2,\dots,v)}-x_v\sum_{j=1}^{v-1} \pi_{\omega_{v-1}}\pi_{(v-1,j)}x_1x_2\dots x_{j-1}e_{\ell-j+1}^{(1,2,\dots,j)}\nonumber\\
&=&\pi_{\omega_{v}}e_{\ell+1}-\pi_{\omega_{v-1}}x_1\dots x_{v-1}x_ve_{\ell-v+1}^{(1,2,\dots,v)}-\sum_{j=1}^{v-1} x_v\pi_{\omega_{v-1}}\pi_{(v-1,j)}x_1x_2\dots x_{j-1}e_{\ell-j+1}^{(1,2,\dots,j)}\nonumber
\end{eqnarray}
Then by Lemma \ref{lem:xpipi}, we have
\begin{eqnarray}
e_{\ell+1}^{(v)}\pi_{\omega_{v}} &=&\pi_{\omega_{v}}e_{\ell+1}-\pi_{\omega_{v-1}}x_1\dots x_ve_{\ell-v+1}^{(1,2,\dots,v)}\nonumber\\
&&-\sum_{j=1}^{v-1} \pi_{\omega_{v-1}}\pi_{(v-1,j+1)}(\pi_j-1)x_jx_1x_2\dots x_{j-1}e_{\ell-j+1}^{(1,2,\dots,j)}\nonumber\\
&=&\pi_{\omega_{v}}e_{\ell+1}-\pi_{\omega_{v-1}}x_1\dots x_ve_{\ell-v+1}^{(1,2,\dots,v)}-\sum_{j=1}^{v-1} \pi_{\omega_{v-1}}\pi_{(v-1,j)}x_1x_2\dots x_{j}e_{\ell-j+1}^{(1,2,\dots,j)}\nonumber\\
&&+\sum_{j=1}^{v-1} \pi_{\omega_{v-1}}\pi_{(v-1,j+1)}x_1x_2\dots x_{j}e_{\ell-j+1}^{(1,2,\dots,j)}\nonumber
\end{eqnarray}
We combine the second term with the sum to get
\begin{eqnarray}
e_{\ell+1}^{(v)}\pi_{\omega_{v}} &=&\pi_{\omega_{v}}e_{\ell+1}-\sum_{j=1}^{v} \pi_{\omega_{v-1}}\pi_{(v-1,j)}x_1x_2\dots x_{j}e_{\ell-j+1}^{(1,2,\dots,j)}\nonumber\\
&&+\sum_{j=1}^{v-1} \pi_{\omega_{v-1}}\pi_{(v-1,j+1)}x_1x_2\dots x_{j}e_{\ell-j+1}^{(1,2,\dots,j)}\nonumber
\end{eqnarray}
We then take out the first term of the first sum and renumber the second sum to obtain
\begin{eqnarray}
e_{\ell+1}^{(v)}\pi_{\omega_{v}} &=&\pi_{\omega_{v}}e_{\ell+1}-\pi_{\omega_{v}}x_1e^{(1)}_{\ell}-\sum_{j=2}^{v} \pi_{\omega_{v-1}}\pi_{(v-1,j)}x_1x_2\dots x_{j}e_{\ell-j+1}^{(1,2,\dots,j)}\nonumber\\
&&+\sum_{j=2}^{v} \pi_{\omega_{v-1}}\pi_{(v-1,j)}x_1x_2\dots x_{j-1}e_{\ell-j+2}^{(1,2,\dots,j-1)}\nonumber
\end{eqnarray}
We now use Lemma \ref{lem:elk1} to combine the first two terms and combine the sums
$$
e_{\ell+1}^{(v)}\pi_{\omega_{v}} =\pi_{\omega_{v}}e^{(1)}_{\ell+1}+\sum_{j=2}^{v} \pi_{\omega_{v-1}}\pi_{(v-1,j)}x_1x_2\dots x_{j-1}
\left(e_{\ell-j+2}^{(1,2,\dots,j-1)} - x_{j}e_{\ell-j+1}^{(1,2,\dots,j)}\right)
$$
We use Lemma \ref{lem:el1k1} to combine the terms within the sum to get
$$
e_{\ell+1}^{(v)}\pi_{\omega_{v}}  = \pi_{\omega_{v}}e^{(1)}_{\ell+1}+\sum_{j=2}^{v} \pi_{\omega_{v-1}}\pi_{(v-1,j)}x_1x_2\dots x_{j-1}
e_{\ell-j+2}^{(1,2,\dots,j)}
$$
We finally combine to make one sum
$$
e_{\ell+1}^{(v)}\pi_{\omega_{v}} =\sum_{j=1}^{v} \pi_{\omega_{v-1}}\pi_{(v-1,j)}x_1x_2\dots x_{j-1}
e_{\ell-j+2}^{(1,2,\dots,j)}
$$
\end{proof}

\begin{lemma}\label{lem:prod1}
Suppose $\alpha_d< j+d$ and let $r=j+d$. Then we have that
\begin{align}
&\partial_{\omega_{(v-1)^c}}\pi_{\omega_{v-1}}\pi_{(v-1,j)}x_1x_2\dots x_{j-1}e_{\ell-j+1}^{(1,2,\dots,j)}
\prod_{i=1}^d\pi_{(v+i-1,\alpha_i)}=\nonumber\\
& \qquad (-1)^d\partial_{\omega_{(v-1)^c}}\pi_{\omega_{v-1}}
\left(\prod_{i=1}^{d}\pi_{(v+i-2,\alpha_i)}\right)\pi_{(v+d-1,r)}
x_1x_2\dots x_{r-1}e_{\ell-r+1}^{(1,2,\dots,r)}
\end{align}
\end{lemma}
\begin{proof}
We prove the lemma by induction on $d$ for a fixed $j$.  
The base case $d=1$ is easily seen to hold by applying the same logic as the inductive step.
Now assume the lemma is true for $k \leq d$, and suppose $\alpha_{d+1}< j+d+1$.  Then
\begin{align}
&\partial_{\omega_{(v-1)^c}}\pi_{\omega_{v-1}}\pi_{(v-1,j)}x_1x_2\dots x_{j-1}e_{\ell-j+1}^{(1,2,\dots,j)}
\prod_{i=1}^{d+1}\pi_{(v+i-1,\alpha_i)}=\nonumber\\
& \qquad (-1)^d \partial_{\omega_{(v-1)^c}}\pi_{\omega_{v-1}}
\left(\prod_{i=1}^{d}\pi_{(v+i-2,\alpha_i)}\right)\pi_{(v+d-1,r)}
x_1x_2\dots x_{r-1}e_{\ell-r+1}^{(1,2,\dots,r)}\pi_{(v+d,\alpha_{d+1})} \label{rhs}
\end{align}
Since $\alpha_{d+1}< r+1$, then $\pi_{(v+d,\alpha_{d+1})} $ cannot commute with $x_1x_2\dots x_{r-1}e_{\ell-r+1}^{(1,2,\dots,r)}$.

We use 
$\pi_{(v+d,\alpha_{d+1})}=\pi_{(v+d,r+1)}\pi_r\pi_{(r-1,\alpha_{d+1})}$ and then commute 
$\pi_{(v+d,r+1)}$ with $x_1x_2\dots x_{r-1}e_{\ell-r+1}^{(1,2,\dots,r)}$ to transform
the right hand side of \eqref{rhs} into
$$(-1)^d\partial_{\omega_{(v-1)^c}}\pi_{\omega_{v-1}}
\left(\prod_{i=1}^{d}\pi_{(v+i-2,\alpha_i)}\right)\pi_{(v+d-1,r)}\pi_{(v+d,r+1)}
x_1x_2\dots x_{r-1}e_{\ell-r+1}^{(1,2,\dots,r)}\pi_r\pi_{(r-1,\alpha_{d+1})}$$
Now, applying Lemma~\ref{lem:elpi1}, the previous expression becomes
\begin{align}
& (-1)^d\partial_{\omega_{(v-1)^c}}\pi_{\omega_{v-1}}
\left(\prod_{i=1}^{d}\pi_{(v+i-2,\alpha_i)}\right)\pi_{(v+d-1,r)}\pi_{(v+d,r+1)}\pi_r
\nonumber \\ 
& \qquad \qquad \qquad \qquad \times x_1x_2\dots x_{r-1}
\left(e_\ell^{(1,\dots,r+1)}+x_r e_{\ell-1}^{(1,\dots,r+1)}\right)
\pi_{(r-1,\alpha_{d+1})}\nonumber\\
& +(-1)^{d+1}\partial_{\omega_{(v-1)^c}}\pi_{\omega_{v-1}}
\left(\prod_{i=1}^{d}\pi_{(v+i-2,\alpha_i)}\right)\pi_{(v+d-1,r)}\pi_{(v+d,r+1)} 
\nonumber \\
& \qquad \qquad \qquad \qquad  \times x_1x_2\dots x_{r-1}
x_re_{\ell-1}^{(1,\dots,r+1)}\pi_{(r-1,\alpha_{d+1})}\nonumber
\end{align}
The first term of the expression is $0$ because by Lemma \ref{lem:pin}
one can extract $\pi_{v+d}$ from the left in
$\pi_{(v+d-1,r)}\pi_{(v+d,r+1)}\pi_r$.
$$\partial_{\omega_{(v-1)^c}}\pi_{\omega_{v-1}}
\left(\prod_{i=1}^{d}\pi_{(v+i-2,\alpha_i)}\right)\pi_{(v+d-1,r)}\pi_{(v+d,r+1)}\pi_r=0.$$
The second term is
$$(-1)^{d+1}\partial_{\omega_{(v-1)^c}}\pi_{\omega_{v-1}}
\left(\prod_{i=1}^{d}(-1)\pi_{(v+i-2,\alpha_i)}\right)\pi_{(v+d-1,r)}\pi_{(v+d,r+1)}\pi_{(r-1,\alpha_{d+1})}
x_1x_2\dots x_{r-1}
x_re_{\ell-1}^{(1,\dots,r+1)}$$
Using $\pi_{(v+d-1,r)}\pi_{(v+d,r+1)}\pi_{(r-1,\alpha_{d+1})}
=\pi_{(v+d-1,\alpha_{d+1})}\pi_{(v+d,r+1)}$, the left hand side of \eqref{rhs}
finally becomes
$$
(-1)^{d+1}\partial_{\omega_{(v-1)^c}}\pi_{\omega_{v-1}}
\left(\prod_{i=1}^{d+1}\pi_{(v+i-2,\alpha_i)}\right)\pi_{(v+d,r+1)}
x_1x_2\dots x_{r}
e_{\ell-1}^{(1,\dots,r+1)}
$$
and the lemma holds.
\end{proof}

\begin{corollary}\label{cor:delpi1}
For a given $j$, let $d$ be the maximum value such that
$\alpha_d< j+d$. We then have
\begin{align}\label{eqn:delpi1}
&\partial_{\omega_{(v-1)^c}}\pi_{\omega_{v-1}}\pi_{(v-1,j)}x_1x_2\dots x_{j-1}e_{\ell-j+1}^{(1,2,\dots,j)}
\prod_{i=1}^m\pi_{(v+i-1,\alpha_i)}=\nonumber\\
& \qquad (-1)^d \partial_{\omega_{(v-1)^c}}\pi_{\omega_{v-1}}
\left(\prod_{i=1}^{d}\pi_{(v+i-2,\alpha_i)}\right)\pi_{(v+d-1,r)}\left(\prod_{i=d+1}^m\pi_{(v+i-1,\alpha_i)}\right)
x_1x_2\dots x_{r-1}e_{\ell-r+1}^{(1,2,\dots,r)}\nonumber\\
\end{align}
where $r=j+d$.
\end{corollary}
\begin{proof}
We first break up the product $\displaystyle{\prod_{i=1}^m\pi_{(v+i-1,\alpha_i)}
=\prod_{i=1}^d\pi_{(v+i-1,\alpha_i)}\prod_{i=d+1}^m\pi_{(v+i-1,\alpha_i)}}.$

Using Lemma~\ref{lem:prod1} to multiply by the first factor, the left hand side of (\ref{eqn:delpi1}) becomes
$$(-1)^d
\partial_{\omega_{(v-1)^c}}\pi_{\omega_{v-1}}
\left(\prod_{i=1}^{d}\pi_{(v+i-2,\alpha_i)}\right)\pi_{(v+d-1,r)}
x_1x_2\dots x_{r-1}e_{\ell-r+1}^{(1,2,\dots,r)}\prod_{i=d+1}^m\pi_{(v+i-1,\alpha_i)}.$$
By hypothesis, $\alpha_{d+1}\ge j+d+1$ which implies that 
the rightmost product in the previous expression commutes with 
$x_1x_2\dots x_{r-1}e_{\ell-r+1}^{(1,2,\dots,r)}$.  The 
left hand side of  (\ref{eqn:delpi1}) is thus equal to its right hand side
and the result follows.
\end{proof}

\begin{corollary} \label{corobigexpan}
Let $v=N-m$. We have
\begin{equation}
\partial_{\omega_{(v-1)^c}}e^{(v)}_\ell\mathcal{R}_{N,[\alpha_1,\dots,\alpha_m]}=
\partial_{\omega_{(v-1)^c}}\sum_{r\not \in \{\alpha_1,\dots,\alpha_m \}}
(-1)^{{\rm pos}(r)}\mathcal{R}_{N,[\alpha_1,\ldots,r,\ldots,\alpha_m]}x_1\dots x_{r-1}e_{\ell-r+1}^{(1,2,\dots,r)}
\end{equation}
where we recall that ${\rm pos}(r)$ is the amount of elements of $\{\alpha_1,\ldots,\alpha_m\}$ that are less than $r$.
\end{corollary}
\begin{proof}
\begin{eqnarray*}
&&\partial_{\omega_{(v-1)^c}}e^{(v)}_\ell\mathcal{R}_{N,[\alpha_1,\dots,\alpha_m]}=
\partial_{\omega_{(v-1)^c}}e^{(v)}_\ell\pi_{\omega_v}\pi_{(v,\alpha_1)}\pi_{(v+1,\alpha_2)}\cdots\pi_{(N-1,\alpha_m)}\\
&=&\sum_{j=1}^v \partial_{\omega_{(v-1)^c}} \pi_{\omega_{v-1}}\pi_{(v-1,j)}x_1x_2\dots x_{j-1}e_{\ell-j+1}^{(1,2,\dots,j)}\pi_{(v,\alpha_1)}\pi_{(v+1,\alpha_2)}\cdots\pi_{(N-1,\alpha_m)}
\end{eqnarray*}
by Lemma \ref{lem:elvpi}.

For each $j$ in the sum define $d(j)$ to be the maximum value such that
$\alpha_{d(j)}< j+d(j)$. If no such $d(j)$ exists, then set $d(j)=0$. Then by Corollary~\ref{cor:delpi1} 
we have
\begin{eqnarray*}
&&\partial_{\omega_{(v-1)^c}}e^{(v)}_\ell\mathcal{R}_{N,[\alpha_1,\dots,\alpha_m]}=\sum_{j=1}^v (-1)^{d(j)}\partial_{\omega_{(v-1)^c}} \pi_{\omega_{v-1}}\times\\
&&\qquad
 \left(\prod_{i=1}^{d(j)}\pi_{(v+i-2,\alpha_i)}\right)\pi_{(v+d(j)-1,r(j))}\left(\prod_{i=d(j)+1}^m\pi_{(v+i-1,\alpha_i)}\right)
x_1x_2\dots x_{r(j)-1}e_{\ell-r(j)+1}^{(1,2,\dots,r(j))}
\end{eqnarray*}
where $r(j)=j+d(j)$. We can simplify this by using 
$$\mathcal{R}_{N,[\alpha_1,\ldots,r(j),\ldots,\alpha_m]}=
\pi_{\omega_{v-1}}\left(\prod_{i=1}^{d(j)}\pi_{(v+i-2,\alpha_i)}\right)\pi_{(v+d(j)-1,r(j))}
\left(\prod_{i=d(j)+1}^m\pi_{(v+i-1,\alpha_i)}\right)$$
And so now we have 
\begin{eqnarray*}
&&\partial_{\omega_{(v-1)^c}}e^{(v)}_\ell\mathcal{R}_{N,[\alpha_1,\dots,\alpha_m]}=\partial_{\omega_{(v-1)^c}}
\sum_{j=1}^v (-1)^{d(j)}\mathcal{R}_{N,[\alpha_1,\ldots,r(j),\ldots,\alpha_m]}
x_1x_2\dots x_{r(j)-1}e_{\ell-r(j)+1}^{(1,2,\dots,r(j))}
\end{eqnarray*}
 Note here that by construction, $\{r(1),\dots,r(v)\}$ is precisely the complement of $\{\alpha_1,\dots,\alpha_m\}$.
Therefore, we can change the summation to sum over the complement of $\{\alpha_1,\dots,\alpha_m\}$ and
furthermore, we can set $\text{pos}(r(j))=d(j)$, which happens to agree with the definition in Section \ref{secpieri} to 
be the amount of numbers in $\{\alpha_1,\dots,\alpha_m\}$ that are less than $r(j)$.

So we conclude that
$$\partial_{\omega_{(v-1)^c}}e^{(v)}_\ell\mathcal{R}_{N,[\alpha_1,\dots,\alpha_m]}=
\partial_{\omega_{(v-1)^c}}\sum_{r\not \in \{\alpha_1,\dots,\alpha_m \}}
(-1)^{{\rm pos}(r)}\mathcal{R}_{N,[\alpha_1,\ldots,r,\ldots,\alpha_m]}x_1\dots x_{r-1}e_{\ell-r+1}^{(1,2,\dots,r)}.$$
\end{proof}

\begin{lemma} \label{lemKh}
Let $\lambda$ be a partition. Using the notation of Section~\ref{secpieri2}, we have
\begin{equation}{\label {Kh}}
x^{\lambda}\, h_k=\sum_{\mu}\left(\prod_{\substack{i\\ \mu_{i+1}=\lambda_i}}^{({\rm down})}\pi_i \right) x^\mu
\end{equation}
where the sum is over all partitions $\mu$ such that $\mu/\lambda$ is a horizontal $k$-strip. 
\end{lemma}
\begin{proof} We proceed by induction on the number of variables $N$.
The result is easily checked when $N=1$.  Now suppose
(\ref{Kh}) is true when the number of variables is less than $N$.
We let $\lambda=(\lambda_1,\dots,\lambda_N)$, $\hat \lambda =(\lambda_1,\dots,\lambda_{N-1})$, $d=\lambda_{N-1}-\lambda_N$, and denote by $h_k^{[\ell]}$
the complete symmetric function $h_k(x_1,\dots,x_{\ell})$ in $\ell$ variables.
Using the simple identity $h_k^{[N]}=x_N^dh_{k-d}^{[N]}+\sum_{j=0}^{d-1}x_N^jh_{k-j}^{[N-1]}$, we have
\begin{align}
x^\lambda h^{[N]}_k&=x^\lambda\left[x_N^dh_{k-d}^{[N]}+\sum_{j=0}^{d-1}x_N^jh_{k-j}^{[N-1]}\right]\nonumber\\
&=\left(x^{\hat \lambda} x_N^{\lambda_N}\right)x_N^d\left(\pi_{N-1}h_{k-d}^{(N-1)}\right)+\sum_{j=0}^{d-1}\left(x^{\hat \lambda}x_N^{\lambda_N}\right)x_N^jh_{k-j}^{[N-1]} \\
&=\pi_{N-1} x_N^{\lambda_N+d}x^{\hat \lambda}  h_{k-d}^{[N-1]}+\sum_{j=0}^{d-1}x_N^{\lambda_N+j}x^{\hat \lambda} h_{k-j}^{[N-1]}
\end{align}
where we used the fact that $x^{\hat \lambda} x_N^{\lambda_N+d}$, being symmetric in $x_{N-1}$ and $x_N$, commutes with $\pi_{N-1}$.  Hence, by induction,
\begin{align}
x^\lambda h^{[N]}_k &=\pi_{N-1}x_N^{\lambda_N+d}\sum_{\substack{\nu/\hat \lambda \text{ is a }\\\text{horiz. } k-d\text{-strip}}}\left(\prod_{\substack{i < N-1\\ \nu_{i+1}=\lambda_i}}^{({\rm down})} \pi_i \right)x^\nu+\sum_{j=0}^{d-1}x_N^{\lambda_N+j}
\sum_{\substack{\nu/\hat \lambda \text{ is a }\\ \text{horiz. } k-j\text{-strip}}}\left(\prod_{\substack{i < N-1\\ \nu_{i+1}=\lambda_i}}^{({\rm down})}\pi_i \right)x^\nu\nonumber\\
&=\sum_{\substack{\nu/\hat \lambda \text{ is a }\\ \text{horiz. }k-d\text{-strip}}}\left(\prod_{\substack{i\\ \nu_{i+1}=\lambda_i}}^{({\rm down})}\pi_i \right)x^{\nu,\lambda_N+d}
+\sum_{j=0}^{d-1}
\sum_{\substack{\nu/\hat \lambda \text{ is a }\\ \text{horiz. } k-j\text{-strip}}}\left(\prod_{\substack{i\\ \nu_{i+1}=\lambda_i}}^{({\rm down})}\pi_i \right)x^{\nu,\lambda_N+j}\nonumber\\
&=\sum_{\substack{\mu/\lambda \text{ is a }\\ \text{horiz. } k\text{-strip}}}\left(\prod_{\substack{i\\ \mu_{i+1}=\lambda_i}}^{({\rm down})} \pi_i \right)x^\mu\nonumber
\end{align}
and the lemma follows.
\end{proof}

\begin{lemma}\label{lem:betazero}  If $\beta >i$, then
$\partial_{\omega_m} \pi_{\omega_{m^c}}
 \hat{\pi}_{\omega_m}\hat{\pi}_{[m,{\alpha_m-1}]}
 \cdots\hat{\pi}_{[i+1,{\alpha_{i+1}-1}]}\hat{\pi}_{[i,{\beta+1}]}\hat{\pi}_{\beta}=0$, where we use the notation of Section~\ref{secpieri2}
\end{lemma}
\begin{proof}
\begin{eqnarray*}
& & \partial_{\omega_m} \pi_{\omega_{m^c}}
 \hat{\pi}_{\omega_m}\hat{\pi}_{[m,{\alpha_m-1}]}
 \cdots\hat{\pi}_{[i+1,{\alpha_{i+1}-1}]}\hat{\pi}_{[i,{\beta+1}]}\hat{\pi}_{\beta} \\
& &
\qquad \qquad \qquad \qquad \qquad \qquad
 = \partial_{\omega_m} \pi_{\omega_{m^c}}
 \hat{\pi}_{\omega_m}\hat{\pi}_{[m,{\alpha_m-1}]}
 \cdots\hat{\pi}_{[i+1,{\alpha_{i+1}-1}]}\hat{\pi}_{[i,{\beta-1}]}\hat{\pi}_{\beta}\hat{\pi}_{\beta+1}\hat{\pi}_{\beta}\\
 & &
\qquad \qquad \qquad \qquad \qquad \qquad
 = \partial_{\omega_m} \pi_{\omega_{m^c}}
 \hat{\pi}_{\omega_m}\hat{\pi}_{[m,{\alpha_m-1}]}
 \cdots\hat{\pi}_{[i+1,{\alpha_{i+1}-1}]}\hat{\pi}_{[i,{\beta-1}]}\hat{\pi}_{\beta+1}\hat{\pi}_{\beta}\hat{\pi}_{\beta+1}\\
 & &
\qquad \qquad \qquad \qquad \qquad \qquad
= \partial_{\omega_m} \pi_{\omega_{m^c}}
 \hat{\pi}_{\omega_m}\hat{\pi}_{[m,{\alpha_m-1}]}
 \cdots\hat{\pi}_{[i+1,{\alpha_{i+1}-1}]}\hat{\pi}_{\beta+1}\hat{\pi}_{[i,{\beta+1}]}
\end{eqnarray*}
We repeat the above process through each $\hat{\pi}_{[j,{\alpha_j-1}]}$ to get
$$ \partial_{\omega_m} \pi_{\omega_{m^c}}
 \hat{\pi}_{\omega_m}\hat{\pi}_{\beta+m-i}\hat{\pi}_{[m,{\alpha_m-1}]}
 \cdots\hat{\pi}_{[i+1,{\alpha_{i+1}-1}]}\hat{\pi}_{[i,{\beta+1}]}$$
 The result then holds since, for $\beta >i$, 
$\hat{\pi}_{\beta+m-i}$ commutes with $\hat{\pi}_{\omega_m}$ and is such that $\pi_{\omega_{m^c}}\hat{\pi}_{\beta+m-i}=0$ (given that
$\pi_j \hat \pi_j=0$).
\end{proof}

\begin{lemma}\label{lem:propP}
$\mathcal{P}_{N,[\alpha_1,\dots,\alpha_m]}$ obeys the following properties (we use the notation of Section~\ref{secpieri2})
\begin{enumerate}
\item\label{prop:p1}
$\mathcal P_{N,[\alpha_1,\ldots,\alpha_m]}{\pi}_{\alpha_i}=
\mathcal P_{N,[\alpha_1,\ldots,\alpha_m]}+\mathcal P_{N,[\alpha_1,\ldots,\alpha_i+1,\ldots,\alpha_m]}$\\
\item\label{prop:pzero}
$\mathcal P_{N,[\alpha_1,\ldots,\alpha_m]}{\pi}_{\alpha_i-1}=0$\\
\item\label{prop:id}
$\mathcal P_{N,[\alpha_1,\ldots,\alpha_m]}{\pi}_{\beta}= \mathcal P_{N,[\alpha_1,\ldots,\alpha_m]}$
for $\beta \notin \{\alpha_i,\alpha_i-1\}$ for $i=1,\ldots,m$.\\
\item
$\mathcal P_{N,[\alpha_1,\ldots,\alpha_m]}\boldsymbol{\pi}_{i_j,I_j}=
\mathcal P_{N,[\alpha_1,\ldots,\alpha_i,\ldots,\alpha_m]}+\mathcal P_{N,[\alpha_1,\ldots,\alpha_i+i_j,\ldots,\alpha_m]}$
if $I_j$ starts with $\alpha_i$.\\
\item
$\mathcal P_{N,[\alpha_1,\ldots,\alpha_m]}\boldsymbol{\pi}_{i_j,I_j}=\mathcal P_{N,[\alpha_1,\ldots,\alpha_m]}$
if $I_j$ starts with  $\beta$ such that $\beta \not \in \{\alpha_1,\dots,\alpha_m\}$.
\end{enumerate}
\end{lemma}

\begin{proof}
To prove {\it (1)}, consider
\begin{eqnarray}
 \mathcal P_{N,[\alpha_1,\ldots,\alpha_m]}{\pi}_{\alpha_i}&=&
 \partial_{\omega_m} \pi_{\omega_{m^c}}
 \hat{\pi}_{\omega_m}\hat{\pi}_{[m,\alpha_m-1]}\cdots\hat{\pi}_{[1,{\alpha_1-1}]}{\pi}_{\alpha_i}\nonumber\\
 &=&
 \partial_{\omega_m} \pi_{\omega_{m^c}}
 \hat{\pi}_{\omega_m}\hat{\pi}_{[m,{\alpha_m-1}]}
 \cdots\hat{\pi}_{[i,{\alpha_i-1}]}{\pi}_{\alpha_i}\hat{\pi}_{[{i-1},{\alpha_{i-1}-1}]}\cdots\hat{\pi}_{[1,{\alpha_1-1}]}\nonumber\\
 &=&
  \partial_{\omega_m} \pi_{\omega_{m^c}}
 \hat{\pi}_{\omega_m}\hat{\pi}_{[m,{\alpha_m-1}]}
 \cdots\hat{\pi}_{[i,{\alpha_i-1}]}
 (1+\hat{\pi}_{\alpha_i})
 \hat{\pi}_{[{i-1},{\alpha_{i-1}-1}]}\cdots\hat{\pi}_{[1,{\alpha_1-1}]}\nonumber\\
&=&
\mathcal  P_{N,[\alpha_1,\ldots,\alpha_m]}+\mathcal P_{N,[\alpha_1,\ldots,\alpha_i+1,\ldots,\alpha_m]}\nonumber
 \end{eqnarray}

For {\it (2)}, we have
$$
 \mathcal P_{N,[\alpha_1,\ldots,\alpha_m]}{\pi}_{\alpha_i-1}=
 \partial_{\omega_m} \pi_{\omega_{m^c}}
 \hat{\pi}_{\omega_m}\hat{\pi}_{[m,{\alpha_m-1}]}
 \cdots\hat{\pi}_{[i,{\alpha_i-1}]}{\pi}_{\alpha_i-1}\hat{\pi}_{[{i-1},{\alpha_{i-1}-1}]}\cdots\hat{\pi}_{[1,{\alpha_1-1}]}\nonumber\\
 =0
$$
since $\hat{\pi}_{\alpha_i-1}{\pi}_{\alpha_i-1}=0$.

For {\it (3)}, suppose that $\alpha_i < \beta <\alpha_{i+1}-1$
(the cases $\beta<\alpha_1-1$ and $\beta > \alpha_m$ are similar).  We have
\begin{eqnarray}
\mathcal P_{N,[\alpha_1,\ldots,\alpha_m]}{\pi}_{\beta}
  &=&
 \partial_{\omega_m} \pi_{\omega_{m^c}}
 \hat{\pi}_{\omega_m}\hat{\pi}_{[m,{\alpha_m-1}]}
 \cdots\hat{\pi}_{[i,{\beta+1}]}{\pi}_{\beta}\hat{\pi}_{[{\beta+2},{\alpha_i-1}]}
 \hat{\pi}_{[{i-1},{\alpha_{i-1}-1}]}\cdots\hat{\pi}_{[1,{\alpha_1-1}]}\nonumber\\
  &=&
 \partial_{\omega_m} \pi_{\omega_{m^c}}
 \hat{\pi}_{\omega_m}\hat{\pi}_{[m,{\alpha_m-1}]}
 \cdots\hat{\pi}_{[i,{\beta+1}]}(\hat{\pi}_{\beta}+1)\hat{\pi}_{[{\beta+2},{\alpha_i-1}]}
 \hat{\pi}_{[{i-1},{\alpha_{i-1}-1}]}\cdots\hat{\pi}_{[1,{\alpha_1-1}]}\nonumber\\
 &=&
 \mathcal P_{N,[\alpha_1,\ldots,\alpha_m]}\nonumber
 \end{eqnarray}
 since  $\partial_{\omega_m} \pi_{\omega_{m^c}}
 \hat{\pi}_{\omega_m}\hat{\pi}_{[m,{\alpha_m-1}]}
 \cdots\hat{\pi}_{[i,{\beta+1}]}\hat{\pi}_{\beta}=0$ by Lemma \ref{lem:betazero} (observe that the lemma applies
since $i \leq \alpha_i<\beta$). 
 
 To prove {\it (4)}, let $I_j=[a,b]$ with $a=\alpha_i$.  By definition, we have
 \begin{equation} 
 \mathcal P_{N,[\alpha_1,\ldots,\alpha_m]}\boldsymbol{\pi}_{i_j,I_j}
 = \mathcal P_{N,[\alpha_1,\ldots,\alpha_m]}
  \pi_{[b-i_j,b-1]}   \cdots \pi_{[\alpha_i+1,\alpha_i+i_j]} \pi_{[\alpha_i,\alpha_i+i_j-1]}\nonumber
\end{equation}
Now, 
by construction of $I_j$, we have $b< \alpha_{i+1}$ since rows $\alpha_i$ and
$\alpha_{i+1}$ need to be distinct.
Hence,  $\pi_{[b-i_j,b-1]}   \cdots \pi_{[\alpha_i+1,\alpha_i+i_j]}$ acts as the identity on 
$\mathcal P_{N,[\alpha_1,\ldots,\alpha_m]}$ from  {\it (\ref{prop:id})}. We thus have
 \begin{eqnarray}
\mathcal P_{N,[\alpha_1,\ldots,\alpha_m]}\boldsymbol{\pi}_{i_j,I_j}
 &=& \mathcal P_{N,[\alpha_1,\ldots,\alpha_m]}
 \pi_{[\alpha_i,\alpha_i+i_j-1]}.\nonumber
\end{eqnarray}
By successive applications of {\it (\ref{prop:p1})} and {\it (\ref{prop:pzero})} we finally get
  \begin{eqnarray}
\mathcal P_{N,[\alpha_1,\ldots,\alpha_m]}\boldsymbol{\pi}_{i_j,I_j}
 &=& \mathcal P_{N,[\alpha_1,\ldots,\alpha_m]}+\mathcal P_{N,[\alpha_1,\ldots,\alpha_i+i_j\ldots,\alpha_m]}
 \nonumber
\end{eqnarray}
since $\alpha_i+i_j\leq b<\alpha_{i+1}$ by construction.

We finally prove {\it (5)}. 
Let $I_j=[a,b]$.  By construction, we have that
$\alpha_i < a\leq b < \alpha_{i+1}$ for some $i$ (the cases
$a\leq b < \alpha_{1}$ and $\alpha_m<a\leq b$ are similar).
By definition, we have
\begin{equation} \label{eqseqpi}
 \mathcal P_{N,[\alpha_1,\ldots,\alpha_m]}\boldsymbol{\pi}_{i_j,I_j}
 = \mathcal P_{N,[\alpha_1,\ldots,\alpha_m]}
  \pi_{[b-i_j,b-1]}   \cdots \pi_{[a+1,a+i_j]} \pi_{[a,a+i_j-1]}
\end{equation}
As can be seen in the previous equation, of the $\pi_j$'s acting from the right on 
$\mathcal P_{N,[\alpha_1,\ldots,\alpha_m]}$, the highest $j$ is such that
$j=b-1<\alpha_{i+1}-1$ while the lowest $j$ is such that $\alpha_{i}<a=j$.
We thus get from {\it (\ref{prop:id})} that
 all $\pi_j$'s in \eqref{eqseqpi} act as the identity
on $ \mathcal P_{N,[\alpha_1,\ldots,\alpha_m]}$, which gives 
$$
\mathcal P_{N,[\alpha_1,\ldots,\alpha_m]}\boldsymbol{\pi}_{i_j,I_j}
 = \mathcal P_{N,[\alpha_1,\ldots,\alpha_m]}\nonumber
$$
\end{proof}

Recall that for $S\subseteq \{1,2,\dots,N\}$, $e_\ell^{S}$ and $h_\ell^{S}$ respectively stand for
$e_\ell(x_1,\dots,x_N)$ and  $h_\ell(x_1,\dots,x_N)$ with $x_i=0$ for $i \notin S$, that is, to
$e_\ell$ and  $h_\ell$ in the variables $x_i$ for  $i \in S$.
The next two lemmas, being elementary, are stated without proof.
\begin{lemma}\label{eqn:elm}
$$e_\ell^{(m+1)}=\sum_{k=0}^\ell(-1)^{\ell-k}x_{m+1}^{\ell-k}e_{k}$$
\end{lemma}

\begin{lemma}\label{eqn:eli}
$$e_\ell^{\{1,\dots,i\}}=\sum_{k=0}^\ell(-1)^{\ell-k}h_{\ell-k}^{\{i+1,\dots, N\}}e_{k}$$
\end{lemma}

\begin{lemma}\label{eqn:hk}
$$h_{k}^{\{i+1,\dots, N\}}\pi_N=h_{k}^{\{i+1,\dots, N+1\}}+\hat{\pi}_Nh_{k}^{\{i+1,\dots, N-1,N+1\}}$$
\end{lemma}
\begin{proof}
\begin{eqnarray*}
h_{k}^{\{i+1,\dots, N\}}\pi_N&=&\displaystyle{\sum_{i=0}^kx^i_Nh_{k-i}^{\{i+1,\dots, N-1\}}\pi_N}\\
&=&\displaystyle{\sum_{i=0}^kx^i_N\pi_Nh_{k-i}^{\{i+1,\dots, N-1\}}}\\
&=&\displaystyle{\sum_{i=0}^k\left(h_i^{\{N,N+1\}}+\hat{\pi}_Nx_{N+1}^i\right)h_{k-i}^{\{i+1,\dots, N-1\}}}\\
&=&\displaystyle{\sum_{i=0}^k h_i^{\{N,N+1\}}h_{k-i}^{\{i+1,\dots, N-1\}}+\hat{\pi}_N\sum_{i=0}^kx_{N+1}^i h_{k-i}^{\{i+1,\dots, N-1\}}}\\
&=&h_{k}^{\{i+1,\dots, N+1\}}+\hat{\pi}_Nh_{k}^{\{i+1,\dots, N-1,N+1\}}
\end{eqnarray*}
\end{proof}

\begin{lemma}\label{eqn:xmk}
 $$\partial_{\omega_{m+1}}\pi_{\omega_{(m+1)^c}}x_{m+1}^{k}\pi_{[m+1,{N-1}]}= 
\partial_{\omega_{m+1}}\pi_{\omega_{(m+1)^c}}\sum_{i=m}^{N-1}\hat{\pi}_{[m+1,i]}h_{k}^{\{i+1,\dots, N\}}$$
\end{lemma}
\begin{proof} 
Proof by induction on $N$.
For the base case, let $N=m+1$ then we have
 $$\partial_{\omega_{m+1}}\pi_{\omega_{(m+1)^c}}x_{m+1}^{k}\pi_{[m+1,{m}]}= 
\partial_{\omega_{m+1}}\pi_{\omega_{(m+1)^c}}x_{m+1}^{k}$$ and $x_{m+1}^{k}=h_k^{\{m+1\}}.$

Now suppose it is true for $N$ and lets show that it is true for $N+1$. (Note here that
$\pi_{\omega_{(m+1)^c}}$ depends on $N$ so lets say $\pi^N_{\omega_{(m+1)^c}}$ 
is for $N$ variables. Also note that 
$\pi^{N+1}_{\omega_{(m+1)^c}}=\pi_{[m+2,N]}\pi^N_{\omega_{(m+1)^c}}$.)
So we have
\begin{eqnarray*}
\partial_{\omega_{m+1}}\pi^{N+1}_{\omega_{(m+1)^c}}x_{m+1}^{k}\pi_{[m+1,{N}]}&=
\partial_{\omega_{m+1}}\pi_{[m+2,N]}\pi^N_{\omega_{(m+1)^c}}x_{m+1}^{k}\pi_{[m+1,{N-1}]}\pi_N\\
&=
\pi_{[m+2,N]}\partial_{\omega_{m+1}}\pi^N_{\omega_{(m+1)^c}}x_{m+1}^{k}\pi_{[m+1,{N-1}]}\pi_N
\end{eqnarray*}
Now use the induction hypothesis to get
$$=\pi_{[m+2,N]}\partial_{\omega_{m+1}}\pi^N_{\omega_{(m+1)^c}}\sum_{i=m}^{N-1}\hat{\pi}_{[m+1,i]}h_{k}^{\{i+1,\dots, N\}}\pi_N.$$
And now use Lemma \ref{eqn:hk} to get
\begin{eqnarray*}
&=&\displaystyle{\partial_{\omega_{m+1}}\pi_{[m+2,N]}\pi^N_{\omega_{(m+1)^c}}\sum_{i=m}^{N-1}\hat{\pi}_{[m+1,i]}
\left(h_{k}^{\{i+1,\dots, N+1\}}+\hat{\pi}_Nh_{k}^{\{i+1,\dots, N-1,N+1\}}\right)}\\
&=&\displaystyle{\partial_{\omega_{m+1}}\pi^{N+1}_{\omega_{(m+1)^c}}
\left(\sum_{i=m}^{N-1}\hat{\pi}_{[m+1,i]}h_{k}^{\{i+1,\dots, N+1\}}+\sum_{i=m}^{N-1}\hat{\pi}_{[m+1,i]}\hat{\pi}_Nh_{k}^{\{i+1,\dots, N-1,N+1\}}\right)}.
\end{eqnarray*}
Notice that in the second sum, the $\pi_N$ commutes with $\hat{\pi}_{[m+1,i]}$ when $i\ne N-1$
 and $\pi^{N+1}_{\omega_{(m+1)^c}}\hat{\pi}_N=0$ so only the last summand survives and we have
 \begin{eqnarray*}
 &=&\displaystyle{\partial_{\omega_{m+1}}\pi^{N+1}_{\omega_{(m+1)^c}}
\left(\sum_{i=m}^{N-1}\hat{\pi}_{[m+1,i]}h_{k}^{\{i+1,\dots, N+1\}}+\hat{\pi}_{[m+1,N-1]}\hat{\pi}_Nh_{k}^{\{N+1\}}\right)}\\
&=&\displaystyle{\partial_{\omega_{m+1}}\pi^{N+1}_{\omega_{(m+1)^c}}
\sum_{i=m}^{N}\hat{\pi}_{[m+1,i]}h_{k}^{\{i+1,\dots, N+1\}}}.
\end{eqnarray*}
\end{proof}

\begin{lemma}\label{lem:elm+1}
\begin{equation}
 \partial_{\omega_{m+1}}\pi_{\omega_{(m+1)^c}}e_\ell^{(m+1)}\pi_{[m+1,N-1]}= \partial_{\omega_{m+1}}\pi_{\omega_{(m+1)^c}}\sum_{i=1}^{N-m}\hat{\pi}_{[m+1,m+i-1]}e_\ell^{\{1,\dots, m+i-1\}}
\end{equation}
\end{lemma}
\begin{proof}
For simplicity, assume that each expression begins with $ \partial_{\omega_{m+1}}\pi_{\omega_{(m+1)^c}}$. By Lemma \ref{eqn:elm}, we have
$$e_\ell^{(m+1)}\pi_{[m+1,N-1]}=\left(\sum_{k=0}^\ell(-1)^{\ell-k}x_{m+1}^{\ell-k}e_{k}\right)\pi_{[m+1,N-1]}$$
Commuting  $e_k$ and $\pi_{[m+1,N-1]}$ and using
Lemma \ref{eqn:xmk}, we obtain 
$$e_\ell^{(m+1)}\pi_{[m+1,N-1]}=\sum_{k=0}^\ell(-1)^{\ell-k}\left(\sum_{i=m}^{N-1}\hat{\pi}_{[m+1,i]}h_{\ell-k}^{\{i+1,\dots, N\}}\right)e_{k}$$
Switching the order of the sums, we get 
by Lemma \ref{eqn:eli} that
$$e_\ell^{(m+1)}\pi_{[m+1,N-1]}=\sum_{i=m}^{N-1}\hat{\pi}_{[m+1,i]}e_\ell^{\{1,\dots, i\}}$$
The lemma is then seen to hold after reindexing the sum and putting back
 $ \partial_{\omega_{m+1}}\pi_{\omega_{(m+1)^c}}$ on both sides of the equation.
\end{proof}

The following lemma is immediate.
\begin{lemma}\label{lem:el1k}
\begin{equation}
e_\ell^{\{1,\dots,k\}}=e_\ell^{\{1,\dots,k-1\}}+x_ke_{\ell-1}^{\{1,\dots,k-1\}}
\end{equation}
\end{lemma}

\begin{lemma}\label{lem:el1kp}
\begin{equation}
e_\ell^{\{1,\dots,k\}}\hat{\pi}_k=\hat{\pi}_ke_\ell^{\{1,\dots,k-1\}}+{\pi}_k\, x_{k+1}e_{\ell-1}^{\{1,\dots,k-1\}}
\end{equation}
\end{lemma}
\begin{proof}
The lemma is immediate from the previous lemma and the identity
$x_k \hat \pi_k = \pi_k x_{k+1}$.
\end{proof}

\begin{lemma}\label{lem:hatpiab} We have
\begin{equation}
\hat{\pi}_{[a,b]}\hat{\pi}_{[a-1,b-1]}\hat\pi_b=\hat\pi_{a-1}\hat{\pi}_{[a,b]}\hat{\pi}_{[a-1,b-1]}
\end{equation}
\end{lemma}
\begin{proof}
Easy by induction.
\end{proof}

\begin{lemma}\label{lem:hatpim}
If $h+1\leq \alpha\leq \beta$ and $h \leq m$, then
\begin{equation}
\partial_{\omega_{m+1}}\hat{\pi}_{[h+1,\alpha]}\hat{\pi}_{[h,\beta]}=
-\partial_{\omega_{m+1}}\hat{\pi}_{[h+1,\beta]}\hat{\pi}_{[h,\alpha-1]}
\end{equation}
\end{lemma}
\begin{proof} We have from the previous lemma
$$
\hat{\pi}_{[h+1,\alpha]}\hat{\pi}_{[h,\beta]}=
\hat{\pi}_{[h+1,\alpha]}\hat{\pi}_{[h,\alpha-1]} \hat \pi_{\alpha} \hat \pi_{\alpha+1} \cdots \hat \pi_\beta
=  \hat \pi_{h} \hat{\pi}_{[h+1,\alpha]}\hat{\pi}_{[h,\alpha-1]}  \hat \pi_{\alpha+1} \cdots \hat \pi_\beta
$$
Hence, $\hat{\pi}_{[h+1,\alpha]}\hat{\pi}_{[h,\beta]}=  \hat \pi_{h} \hat{\pi}_{[h+1,\beta]}\hat{\pi}_{[h,\alpha-1]}$ and
the lemma follows since $\partial_{\omega_{m+1}} \hat \pi_h= - \partial_{\omega_{m+1}}$.
\end{proof}

\begin{lemma}\label{lem:pihatprod}
Suppose that $\alpha_{h}\ge h+j$.  Then, for $1\leq h\leq m$, we have
\begin{eqnarray}
&&
\partial_{\omega_{m+1}} \pi_{\omega_{(m+1)^c}}\hat{\pi}_{\omega_m}
\hat{\pi}_{[{m+1},{m+j-1}]}e_\ell^{\{1,\dots, m+j-1\}}
\left(\prod_{i=m}^h\hat{\pi}_{[i,{\alpha_{i}-1}]}\right) \\
&&\quad\quad=
(-1)^{m+1-h} \partial_{\omega_{m+1}} \pi_{\omega_{(m+1)^c}}\hat{\pi}_{\omega_m}
\left(\prod_{i=m}^h \hat \pi_{[i+1,\alpha_{i}-1]}\right)\hat\pi_{[h,h+j-2]}
e_\ell^{\{1,\dots,h+j-2\}} \nonumber
\end{eqnarray}
where the products are decreasing.

\end{lemma}

\begin{proof}
We prove the lemma by induction on $h$ for a fixed $j$ in the decreasing direction.  We will not show the base case $h=m$ since it can
 be proven using the same ideas as in the general case.  Therefore, 
suppose the lemma holds for $h$ and we will show that it holds for $h-1$. Assume $\alpha_{h-1}\ge h-1+j$.
We then have
\begin{align} \label{eqimportant}
\partial_{\omega_{m+1}} & \pi_{\omega_{(m+1)^c}}\hat{\pi}_{\omega_m}
\hat{\pi}_{[{m+1},{m+j-1}]}e_\ell^{\{1,\dots, m+j-1\}}
\left(\prod_{i=m}^{h-1}\hat{\pi}_{[i,{\alpha_{i}-1}]}\right) \\
&=
\partial_{\omega_{m+1}} \pi_{\omega_{(m+1)^c}}\hat{\pi}_{\omega_m}
\hat{\pi}_{[{m+1},{m+j-1}]}e_\ell^{\{1,\dots, m+j-1\}}
\left(\prod_{i=m}^{h}\hat{\pi}_{[i,{\alpha_{i}-1}]}\right)\hat{\pi}_{[h-1,{\alpha_{h-1}-1}]}\nonumber \\
&=
(-1)^{m+1-h} \partial_{\omega_{m+1}} \pi_{\omega_{(m+1)^c}}\hat{\pi}_{\omega_m}
\left(\prod_{i=m}^h \hat \pi_{[i+1,\alpha_{i}-1]}\right)\hat\pi_{[h,h+j-2]}e_\ell^{\{1,\dots,h+j-2\}}
\hat{\pi}_{[h-1,{\alpha_{h-1}-1}]}\nonumber 
\end{align}
Since $\alpha_{h-1}\ge h-1+j$, $\hat{\pi}_{[h-1,{\alpha_{h-1}-1}]}$ does not commute with $e_\ell^{\{1,\dots,h+j-2\}}$.
Using $\hat{\pi}_{[h-1,{\alpha_{h-1}-1}]}=\hat{\pi}_{[h-1,h+j-3]}\hat{\pi}_{h+j-2}\hat{\pi}_{[h+j-1,{\alpha_{h-1}-1}]}$ and
Lemma~\ref{lem:el1kp}, the right hand side of the previous equation becomes
\begin{eqnarray}
&&(-1)^{m+1-h} \partial_{\omega_{m+1}} \pi_{\omega_{(m+1)^c}}\hat{\pi}_{\omega_m}
\left(\prod_{i=m}^h \hat \pi_{[i+1,\alpha_{i}-1]}\right)\hat\pi_{[h,h+j-2]}\hat{\pi}_{[h-1,h+j-3]}\times\\
&&\quad\quad\quad\quad
\left(\hat{\pi}_{h+j-2}e_\ell^{\{1,\dots,h+j-3\}}+\pi_{h+j-2}x_{h+j-1}e_{\ell-1}^{\{1,\dots,h+j-3\}}\right)
\hat{\pi}_{[h+j-1,{\alpha_{h-1}-1}]} \nonumber
\end{eqnarray}
The second term of the sum is zero from Lemma \ref{lem:hatpiab} (note that $\pi_{a}=\hat\pi_{a}+1$ satisfies
the usual braid relation with $\hat \pi_b$) and the fact that
$\partial_{\omega_{m+1}} \pi_{m}=0$.
The previous expression is thus equal to
\begin{eqnarray}
&&(-1)^{m+1-h} \partial_{\omega_{m+1}} \pi_{\omega_{(m+1)^c}}\hat{\pi}_{\omega_m}
\left(\prod_{i=m}^h \hat \pi_{[i+1,\alpha_{i}-1]}\right)\hat\pi_{[h,h+j-2]}\hat{\pi}_{[h-1,h+j-3]}\times\\
&&\quad\quad\quad\quad
\hat{\pi}_{h+j-2}e_\ell^{\{1,\dots,h+j-3\}}
\hat{\pi}_{[h+j-1,{\alpha_{h-1}-1}]} \nonumber
\end{eqnarray}
Now, $e_\ell^{\{1,\dots,h+j-3\}}$ commutes and the expression becomes
\begin{equation}
(-1)^{m+1-h} \partial_{\omega_{m+1}} \pi_{\omega_{(m+1)^c}}\hat{\pi}_{\omega_m}
\left(\prod_{i=m}^h \hat \pi_{[i+1,\alpha_{i}-1]}\right)\,
\left(\hat\pi_{[h,h+j-2]}\hat{\pi}_{[h-1,\alpha_{h-1}-1]}\right)
e_\ell^{\{1,\dots,(h-1)+j-2\}} \nonumber
\end{equation}
Using Lemma~\ref{lem:hatpim}, we finally get for the right hand side of \eqref{eqimportant}
\begin{eqnarray}
&&(-1)^{m+1-h} \partial_{\omega_{m+1}} \pi_{\omega_{(m+1)^c}}\hat{\pi}_{\omega_m}
\left(\prod_{i=m}^h \hat \pi_{[i+1,\alpha_{i}-1]}\right)\times\\
&&\quad\quad\quad\quad
\left((-1)\hat\pi_{[h,\alpha_{h-1}-1]}\hat{\pi}_{[h-1,(h-1)+j-2]}\right)
e_\ell^{\{1,\dots,(h-1)+j-2\}} \nonumber\\
&&(-1)^{m+1-(h-1)} \partial_{\omega_{m+1}} \pi_{\omega_{(m+1)^c}}\hat{\pi}_{\omega_m}
\left(\prod_{i=m}^{h-1} \hat \pi_{[i+1,\alpha_{i}-1]}\right)
\hat{\pi}_{[h-1,(h-1)+j-2]}
e_\ell^{\{1,\dots,(h-1)+j-2\}} \nonumber
\end{eqnarray}
and the lemma holds.
\end{proof}

\begin{corollary}\label{cor:alphah}
For a given $j$, let $h$ be the minimum value such that $\alpha_{h}\ge h+j$.  Then
\begin{eqnarray}
&&
\partial_{\omega_{m+1}} \pi_{\omega_{(m+1)^c}}\hat{\pi}_{\omega_m}
\hat{\pi}_{[{m+1},{m+j-1}]}e_\ell^{\{1,\dots, m+j-1\}}
\left(\prod_{i=m}^1\hat{\pi}_{[i,{\alpha_{i}-1}]}\right) \\
&&\quad=
(-1)^{m+1-h} \partial_{\omega_{m+1}} \pi_{\omega_{(m+1)^c}}\hat{\pi}_{\omega_m}
\left(\prod_{i=m}^h \hat \pi_{[i+1,\alpha_{i}-1]}\right)\hat\pi_{[h,h+j-2]}
\left(\prod_{i=h-1}^1\hat{\pi}_{[i,{\alpha_{i}-1}]}\right)
e_\ell^{\{1,\dots,h+j-2\}} \nonumber
\end{eqnarray}
\end{corollary}
\begin{proof}
We first break up the product  as
$\displaystyle{\prod_{i=m}^1\hat{\pi}_{[i,{\alpha_{i}-1}]}=\left(\prod_{i=m}^h\hat{\pi}_{[i,{\alpha_{i}-1}]}\right)\left(\prod_{i=h-1}^1\hat{\pi}_{[i,{\alpha_{i}-1}]}\right)}$.
Using Lemma \ref{lem:pihatprod} to multiply by the first factor, we get
$$(-1)^{m+1-h} \partial_{\omega_{m+1}} \pi_{\omega_{(m+1)^c}}\hat{\pi}_{\omega_m}
\left(\prod_{i=m}^h \hat \pi_{[i+1,\alpha_{i}-1]}\right)\hat\pi_{[h,h+j-2]}
e_\ell^{\{1,\dots,h+j-2\}}
\left(\prod_{i=h-1}^1\hat{\pi}_{[i,{\alpha_{i}-1}]}\right).$$
By hypothesis $\alpha_{h-1}<h-1+j$, which implies that the rightmost product in the previous expression commutes with $e_\ell^{\{1,\dots,h+j-2\}}$. Therefore we
have 
$$(-1)^{m+1-h} \partial_{\omega_{m+1}} \pi_{\omega_{(m+1)^c}}\hat{\pi}_{\omega_m}
\left(\prod_{i=m}^h \hat \pi_{[i+1,\alpha_{i}-1]}\right)\hat\pi_{[h,h+j-2]}
\left(\prod_{i=h-1}^1\hat{\pi}_{[i,{\alpha_{i}-1}]}\right)
e_\ell^{\{1,\dots,h+j-2\}}.$$
\end{proof}

\begin{lemma} We have \label{lemfinal}
\begin{eqnarray} \label{eqfinal}
&&\partial_{\omega_{m+1}} \pi_{\omega_{(m+1)^c}}
\sum_{j=1}^{N-m}\hat{\pi}_{[{m+1},{m+j-1}]}e_\ell^{\{1,\dots, m+j-1\}}
\hat{\pi}_{\omega_m}\hat{\pi}_{[m,{\alpha_m-1}]}\ldots\hat{\pi}_{[1,{\alpha_1-1}]}x^{\Lambda^*}\nonumber \\
&&\qquad\qquad \qquad =
\sum_{r\not \in \{\alpha_1,\dots,\alpha_m \}}(-1)^{{\rm pos}(r)}\mathcal{P}_{N,[\alpha_1,\dots,r,\dots,\alpha_m]}
e_\ell^{\{1,\dots,r-1\}}x^{\Lambda^*}
\end{eqnarray}
\end{lemma}

\begin{proof}
For each $j$ in the sum on the left hand side of \eqref{eqfinal}, define $h(j) \geq 1$ to be the minimum value such that
$\alpha_{h(j)}\ge h(j)+j$. If no such $h(j)$ exists, set $h(j)=m+1$. Then, using Corollary \ref{cor:alphah} and straightforward manipulations, the left hand side of \eqref{eqfinal} can be reexpressed as
\begin{eqnarray*}
&& \sum_{j=1}^{N-m}
(-1)^{m+1-h(j)} \partial_{\omega_{m+1}} \pi_{\omega_{(m+1)^c}}\hat{\pi}_{\omega_m}\times\\
&&\quad\qquad\qquad\left(\prod_{i=m}^{h(j)} \hat \pi_{[i+1,\alpha_{i}-1]}\right)\hat\pi_{[h(j),h(j)+j-2]}
\left(\prod_{i=h(j)-1}^1\hat{\pi}_{[i,{\alpha_{i}-1}]}\right)
e_\ell^{\{1,\dots,h(j)+j-2\}} x^{\Lambda^*} \\
&& \qquad =\sum_{j=1}^{N-m}
(-1)^{m+1-h(j)} \partial_{\omega_{m+1}} \pi_{\omega_{(m+1)^c}}\hat{\pi}_{\omega_m}\times\\
&&\quad\qquad\qquad\left(\prod_{i=m}^{h(j)} \hat \pi_{[i+1,\alpha_{i}-1]}\right)\hat\pi_{[h(j),r(j)-1]}
\left(\prod_{i=h(j)-1}^1\hat{\pi}_{[i,{\alpha_{i}-1}]}\right)
e_\ell^{\{1,\dots,r(j)-1\}} x^{\Lambda^*}
\end{eqnarray*}
where we used the substitution $r(j)=h(j)+j-1$ in the last sum.  Observe that we
almost have the form for 
\begin{eqnarray*}
\mathcal{P}_{N,[\alpha_1,\dots,r(j),\dots,\alpha_m]}=
\partial_{\omega_{m+1}} \pi_{\omega_{(m+1)^c}}\hat{\pi}_{\omega_{m+1}}
\left(\prod_{i=m}^{h(j)} \hat \pi_{[i+1,\alpha_{i}-1]}\right)\hat\pi_{[h(j),r(j)-1]}
\left(\prod_{i=h(j)-1}^1\hat{\pi}_{[i,{\alpha_{i}-1}]}\right)
\end{eqnarray*} 
in the last expression except that we need
$\hat \pi_{\omega_{m+1}}$ instead of $\hat \pi_{\omega_m}$. Using $\hat \pi_{\omega_{m+1}}= \hat \pi_1 \cdots \hat \pi_m \hat \pi_{\omega_m}$,
we have
$$\partial_{\omega_{m+1}} \pi_{\omega_{(m+1)^c}}\hat{\pi}_{\omega_m}=(-1)^{m}\partial_{\omega_{m+1}} \pi_{\omega_{(m+1)^c}}\, \hat{\pi}_{\omega_{m+1}}$$
The left hand side of \eqref{eqfinal} can thus be rewritten as
\begin{eqnarray*}
&&\sum_{j=1}^{N-m}(-1)^{m+1-h(j)}(-1)^{m}\mathcal{P}_{N,[\alpha_1,\dots,r(j),\dots,\alpha_m]}
e_\ell^{\{1,\dots,r(j)-1\}} x^{\Lambda^*} \\
&&\quad\quad\quad=\sum_{j=1}^{N-m}(-1)^{h(j)-1}\mathcal{P}_{N,[\alpha_1,\dots,r(j),\dots,\alpha_m]}
e_\ell^{\{1,\dots,r(j)-1\}} x^{\Lambda^*}
\end{eqnarray*}
Note here that by construction, $\{r(1),\dots,r(N-m)\}$ is precisely the complement of $\{\alpha_1,\dots,\alpha_m\}$.
Therefore, we can change the summation to sum over the complement of $\{\alpha_1,\dots,\alpha_m\}$ and
furthermore, we can set $\text{pos}(r(j))=h(j)-1$, which happens to agree with the definition of $\text{pos}(r)$
given in Section~\ref{secpieri} (which corresponds to the number of elements of $\{\alpha_1,\dots,\alpha_m\}$ smaller than $r$). 
Hence the left hand side of (\ref{eqfinal})
becomes
$$\sum_{r\not \in \{\alpha_1,\dots,\alpha_m \}}(-1)^{{\rm pos}(r)}\mathcal{P}_{N,[\alpha_1,\dots,r,\dots,\alpha_m]}
e_\ell^{\{1,\dots,r-1\}}x^{\Lambda^*}$$
and the lemma holds.
\end{proof}

\end{document}